\documentclass[11pt]{amsart}

\usepackage{amsmath,amsfonts}
\usepackage{amssymb}
\usepackage{amscd}
\usepackage{amsthm}
\usepackage{yhmath}
\usepackage{subfigure}
\usepackage{graphicx}   

\usepackage{comment}

\usepackage[all]{xy}
\usepackage{color}

\usepackage{marginnote}
\usepackage{hyperref}

\usepackage{todonotes}

\numberwithin{equation}{section}

\newtheorem{theorem}[equation]{Theorem}
\newtheorem{thm}[equation]{Theorem}
\newtheorem{proposition}[equation]{Proposition}

\newtheorem{lemma}[equation]{Lemma}

\newtheorem{corollary}[equation]{Corollary}

\theoremstyle{remark}
\newtheorem{remark}[equation]{Remark}

\theoremstyle{definition}

\newtheorem{Def}[equation]{Definition} 

\newtheorem{definition}[equation]{Definition}

\newtheorem{example}[equation]{Example}

\def\XXint#1#2#3{{\setbox0=\hbox{$#1{#2#3}{\int}$}
	\vcenter{\hbox{$#2#3$}}\kern-.5\wd0}}

\newcommand{\N}{\mathbb N}

\newcommand{\R}{\mathbb R}

\newcommand{\Z}{\mathbb Z}
\newcommand{\Q}{\mathbb Q}

\newcommand{\g}{\mathfrak{g}}

\newcommand{\h}{\mathfrak{h}}

\newcommand{\HH}{\mathbb H}

\newcommand{\diam}{\operatorname{diam}}

\renewcommand{\ker}{\operatorname{Ker}}

\newcommand{\card}{\operatorname{Card}}

\newcommand\carset{1\negmedspace{\rm l}}

\newcommand{\Span}{\operatorname{span}}

\newcommand{\norm}[1]{\Vert#1\Vert}

\def\eps{\epsilon}

\begin{document} 

\title[BCP on graded groups]{Besicovitch Covering Property on graded groups and applications to measure differentiation}

\author{Enrico Le Donne}

\address[Le Donne]{Department of Mathematics and Statistics, P.O. Box 35,
FI-40014,
University of Jyv\"askyl\"a, Finland}
\email{ledonne@msri.org}

\author{S\'everine Rigot}
\address[Rigot]{Laboratoire de Math\'ematiques J.A. Dieudonn\'e UMR CNRS 7351,  Universit\'e Nice Sophia Antipolis, 06108 Nice Cedex 02, France}
\email{rigot@unice.fr}

\thanks{The first author acknowledges the support of the Academy of Finland project no.~288501. The second author is partially supported by ANR grants ANR-12-BS01-0014-01 and ANR-15-CE40-0018}

\renewcommand{\subjclassname} {\textup{2010} Mathematics Subject Classification}


\begin{abstract}
We give a complete answer to which homogeneous groups admit homogeneous distances for which the Besicovitch Covering Property (BCP) holds. In particular, we prove that a stratified group admits homogeneous distances for which BCP holds if and only if the group has step 1 or 2. These results are obtained as consequences of a more general study of homogeneous quasi-distances on graded groups. Namely, we prove that a positively graded group admits continuous homogeneous quasi-distances satisfying BCP if and only if any two different layers of the associated positive grading of its Lie algebra commute. The validity of BCP has several consequences. Its connections with the theory of differentiation of measures is one of the main motivations of the present paper. As a consequence of our results, we get for instance that a stratified group can be equipped with some homogeneous distance so that the differentiation theorem holds for each locally finite Borel measure if and only if the group has step 1 or 2. The techniques developed in this paper allow also us to prove that sub-Riemannian distances on stratified groups of step 2 or higher never satisfy BCP. Using blow-up techniques this is shown to imply that on a sub-Riemannian manifold the differentiation theorem does not hold for some locally finite Borel measure.
\end{abstract}

\date{December 15, 2015}
\maketitle
\setcounter{tocdepth}{2}
\tableofcontents

\section{Introduction} \label{sect:intro}

Covering theorems are known to be among the fundamental tools in analysis and geometry. They reflect, in a certain sense, the geometry of the space and are commonly used to establish connections between local and global properties. All covering theorems are based on the same principle: from an arbitrary cover of a set, one tries to extract a subcover that is as disjointed as possible. Among classical covering theorems, the Besicovitch Covering Property (BCP), in which we are interested in the present paper, originates from works of A.~Besicovitch in connection with the theory of differentiation of measures in Euclidean spaces (\cite{Besicovitch_1}, \cite{Besicovitch_2}, see also~\cite[2.8]{Federer}, \cite{Preiss83} and Section~\ref{sect:differentiation-measures}).

The development of analysis and geometry on abstract metric spaces leads naturally to the question of the validity of suitable covering theorems on non-Euclidean spaces. Graded groups provide a natural framework for many developments. A graded group is a Lie group equipped with an appropriate family of dilations. A homogeneous quasi-distance on a graded group is a left-invariant quasi-distance that is one-homogeneous with respect to the family of dilations. Graded groups equipped with homogeneous quasi-distances naturally generalize finite dimensional normed vector spaces. Due to the presence of translations and dilations, they provide a setting where many aspects of classical analysis and geometry can be carried out. Beyond such a priori considerations, these spaces form an important framework because of their occurrences in many settings. There is a characterization of positively graduable Lie groups as connected locally compact groups admitting contractive automorphisms (\cite{Siebert}, see also Theorem~\ref{thm:siebert}). The interest of E.~Siebert for groups admitting contractive automorphisms was motivated by phenomenon appearing in probability theory on groups. Such groups have more generally been considered by various authors. In particular, homogeneous groups equipped with homogeneous quasi-distances as considered in~\cite{Folland-Stein} (see also Definition~\ref{def:homogeneous-groups}) fit within this framework. They have been considered mainly in connection with their applications in harmonic analysis, complex analysis of several variables, and study of some non-elliptic differential operators. We refer to ~\cite{Folland-Stein} and the references therein for a detailed presentation of these aspects. A class of homogeneous groups equipped with homogeneous distances of particular interest are stratified groups equipped with sub-Riemannian distances. They are also known as Carnot groups according to a terminology due to P.~Pansu. See Section~\ref{sec:nobcp-ccdist} where we use the more explicit terminology sub-Riemannian Carnot groups. One of their occurrences is as metric tangent spaces to sub-Riemannian manifolds where they play in some sense the role Euclidean spaces play in Riemannian geometry (see e.g.~\cite{Mitchell},~\cite{bellaiche}). These structures have also connections with optimal control theory (see e.g.~\cite{Montgomery},~\cite{Agrachev_Sachkov}).

In the present paper, we give a complete answer to which graded groups admit continuous homogeneous quasi-distances for which BCP holds, see Theorem~\ref{thm:intro-main-result}. Characterizations of homogeneous and stratified groups admitting homogeneous distances satisfying BCP follow as particular cases of our results on graded groups, see Corollary~\ref{cor:intro-homogeneous-groups} and Corollary~\ref{cor:intro-stratified-groups}. We also complete previous results about the non-validity of BCP for sub-Riemannian distances on stratified groups of step $\geq 2$, see Theorem~\ref{thm:intro-no:BCP:Carnot}. Finally, we give applications to measure differentiation which is one of the main motivations for this paper, see Theorem~\ref{thm:intro-preiss-improved}, Theorem~\ref{thm:intro-diff-homogeneous}, Corollary~\ref{cor:intro-diff-stratified}, and Theorem~\ref{thm:intro-subriemannian-diff-measures}.

To explain these results, we first recall the Besicovitch Covering Property (BCP) in the general quasi-metric setting. We refer to Section~\ref{sect:bcp-vs-wbcp} for a more detailed discussion about this covering property. See also Section~\ref{subsec:homogenous-qdist} for our conventions about quasi-metric spaces, which are the classical ones. A quasi-metric space $(X,d)$ satisfies BCP if there exists a constant $N\geq 1$ such that the following holds. For any bounded set $A\subset X$ and any family $\mathcal{B}$ of balls such that each point of $A$ is the center of some ball of $\mathcal{B}$, there is a finite or countable subfamily $\mathcal{F}\subset \mathcal{B}$ such that the balls in $\mathcal{F}$ cover $A$, and every point in $X$ belongs to at most $N$ balls in $\mathcal{F}$.

We briefly recall now the definitions of graded groups and homogeneous quasi-distances. We refer to Section~\ref{sect:graded-groups} for a complete presentation. A graded group is a simply connected Lie group $G$ whose Lie algebra $\g$ is endowed with a positive grading $\g=  \oplus_{t\in (0,+\infty)} V_t$ where $[V_s, V_t]\subset V_{s+t}$ for all $s,t >0.$ At the level of the Lie algebra, the associated dilation of factor $\lambda$ is defined as the unique linear map $\delta_\lambda:\g\to\g$ such that $\delta_\lambda (X) = \lambda^t X$ for all $X\in V_t$. Denoting also by $\delta_\lambda$ the unique Lie group homomorphism induced by this Lie algebra homomorphism, a quasi-distance $d$ on $G$ is said to be homogeneous if it is left-invariant and one-homogeneous with respect to the associated family of dilations (where the latter means $d(\delta_\lambda(p),\delta_\lambda(q)) = \lambda d(p,q)$ for all $p$, $q\in G$ and all $\lambda>0$).

To state the main results of this paper and for later convenience, we introduce the following definition that will be shown to algebraically characterize the validity of BCP for some homogeneous quasi-distances.

\begin{definition}[Graded groups with commuting different layers] \label{def:groups-commuting-layers}
Let $G$ be a graded group and let $\oplus_{t>0} V_t$ be the associated positive grading of its Lie algebra. We say that $G$ has {\em commuting different layers} if $[V_t,V_s] = \{0\}$ for all $t$, $s>0$ such that $t\not= s$.
\end{definition}

Our main results read as follows.

\begin{theorem} \label{thm:intro-main-result}
Let $G$ be a graded group. There exist continuous homogeneous quasi-distances on $G$ for which BCP holds if and only if $G$ has commuting different layers.
\end{theorem} 

Homogeneous groups are those graded groups that can be equipped with homogeneous distances, i.e., homogeneous quasi-distances that satisfy the triangle inequality. Equivalently, all layers $V_t$ with $t<1$ of the associated positive grading of their Lie algebra are $\{0\}$, see Definition~\ref{def:homogeneous-groups} and Proposition~\ref{prop:homogeneous-groups}. For such groups, we prove that homogeneous distances are continuous, see Corollary~\ref{cor:continuity-homogeneous-distance}, and we get the following corollary.

\begin{corollary} \label{cor:intro-homogeneous-groups}
Let $G$ be a homogeneous group. There exist homogeneous distances on $G$ for which BCP holds if and only if $G$ has commuting different layers.
\end{corollary} 

An important class of homogeneous groups are stratified groups, which are those for which the degree-one layer of the associated positive grading generates the Lie algebra, see Definitions~\ref{def:stratified-algebras} and~\ref{def:graded-stratified-groups}. A stratified group  has commuting different layers if and only if it has (nilpotency) step 1 or 2. For such groups, Corollary~\ref{cor:intro-homogeneous-groups} reads then as follows.

\begin{corollary} \label{cor:intro-stratified-groups}
Let $G$ be a stratified group. There exist homogeneous distances on $G$ for which BCP holds if and only if $G$ is of step $\leq 2$.
\end{corollary} 

One of the main motivations for studying the validity of BCP on graded groups is its connection with the theory of differentiation of measures. If $\mu$ is a locally finite Borel measure on a metric space $(X,d)$, we say that the differentiation theorem holds on $(X,d)$ for $\mu$ if
\begin{equation*}
\lim_{r\downarrow 0^+} \frac{1}{\mu(B_d(p,r))} \int_{B_d(p,r)} f(q) \, d\mu(q) = f(p)
\end{equation*}
for $\mu$-almost every $p\in X$ and all $f\in L_{\rm loc}^1(\mu)$. For homogeneous groups equipped with homogeneous distances, we prove the following characterization.

\begin{theorem} \label{thm:intro-preiss-improved}
Let $G$ be a homogeneous group and let $d$ be a homogeneous distance on $G$. The  differentiation theorem holds on $(G,d)$ for all locally finite Borel measures if and only if $(G,d)$ satisfies BCP.
\end{theorem}

This characterization is a consequence of a characterization of the validity of the  differentiation theorem for all locally finite Borel measures on metric spaces due to D.~Preiss, taking the additional structure into account, namely, using left-translations and dilations, see Section~\ref{sect:differentiation-measures}. Together with Corollary~\ref{cor:intro-homogeneous-groups} and Corollary~\ref{cor:intro-stratified-groups}, we get the following results.

\begin{theorem} \label{thm:intro-diff-homogeneous}
Let $G$ be a homogeneous group. There exists some homogeneous distance $d$ on $G$ such that the differentiation theorem holds on $(G,d)$ for all locally finite Borel measures if and only if $G$ has commuting different layers.
\end{theorem}

\begin{corollary} \label{cor:intro-diff-stratified}
Let $G$ be a stratified group. There exists some homogeneous distance $d$ on $G$ such that the differentiation theorem holds on $(G,d)$ for all locally finite Borel measures if and only if $G$ is of step $\leq 2$.
\end{corollary}

To put our results in perspective, let us first recall that it is well-known since the works of Besicovitch in the 40's that BCP holds in Euclidean spaces, and more generally in finite dimensional normed vector spaces. It is also known that the Riemannian distance on a Riemannian manifold of class $C^2$ satisfies a property that generalizes BCP, see~\cite[2.8]{Federer} and Sections~\ref{sect:differentiation-measures} and \ref{sec:nobcp-ccdist}. On the contrary, it is also well-known that BCP does not hold on infinite dimensional  normed vector spaces. Until recently only few results were known for graded groups equipped with homogeneous quasi-distances. It was proved independently and at the same time in~\cite{Sawyer_Wheeden} and~\cite{Koranyi-Reimann95} that BCP does not hold on the Heisenberg groups equipped with the Cygan-Kor\'anyi distance. Later on it was proved in~\cite{Rigot} that BCP also fails for sub-Riemannian distances on stratified groups under some regularity assumptions on the sub-Riemannian distance. After these negative answers about the validity of BCP, it was commonly believed that there would probably not exist homogeneous quasi-distances satisfying BCP on graded groups.

Let us now recall that any two homogeneous distances on a homogeneous group, and more generally any two homogeneous quasi-distances on a graded group, are biLipschitz equivalent. However, it turns out that the validity of BCP is not stable under a biLipschitz change of quasi-distance.

\begin{theorem}[{\cite[Theorem 1.6]{LeDonne_Rigot_Heisenberg_BCP}}] \label{thm:destroybcp}
Let $(X,d)$ be a metric space. Assume that there exists an accumulation point in $(X,d)$. Then, for all $0<c<1$, there exists a distance $d_c$ on $X$ such that $c \, d \leq d_c \leq d$ and for which BCP does not hold.
\end{theorem}

See also~\cite[Theorem~3]{Preiss83} from which Theorem~\ref{thm:destroybcp} is inspired. Notice that Theorem~\ref{thm:destroybcp} can be extended to quasi-distances. As a consequence the non-validity of BCP for some homogeneous quasi-distances cannot give any hint towards the existence or non-existence of some other homogeneous quasi-distance satisfying BCP. Since for many purposes the choice of a specific quasi-distance up to biLipschitz equivalence does not really matter, the question of the existence of some homogeneous quasi-distance for which BCP holds on a graded group remained meaningful.

The present paper follows two previous papers,~\cite{LeDonne_Rigot_Heisenberg_BCP} and~\cite{LeDonne_Rigot_rmknobcp}. The existence of some homogeneous distances that satisfy BCP on the Heisenberg groups is proved in~\cite{LeDonne_Rigot_Heisenberg_BCP}. On the contrary it is proved in~\cite{LeDonne_Rigot_rmknobcp} that natural analogues of these distances on stratified groups of step $\geq 3$ do not satisfy BCP. These two cases strongly suggested that the structure of the dilations, which comes from the structure of the grading of the Lie algebra, plays a crucial role. Theorem~\ref{thm:intro-main-result} characterizes precisely in which sense the structure of the grading plays a role for our purposes.

Let us now say few words about the proof of Theorem~\ref{thm:intro-main-result}. Among the simplest examples of positively graduable groups are the Abelian ones, the Heisenberg groups, and free-nilpotent groups of step 2, see Example~\ref{ex:abelian},~Example~\ref{ex:heisenberg}, and Section~\ref{subsect:BCP-free-step2}. They play a key role in our proof of Theorem~\ref{thm:intro-main-result}. A first step is indeed the study of the validity or non-validity of BCP for homogeneous quasi-distances relatively to various possible positive gradings on these groups.

To prove the existence of homogeneous quasi-distances satisfying BCP on graded groups with commuting different layers, we first prove that Hebisch-Sikora's quasi-distances satisfy BCP on stratified free-nilpotent groups of step 2, see Theorem~\ref{thm:BCP-free-step2}. Hebisch-Sikora's quasi-distances are those homogeneous quasi-distances whose unit ball centered at the identity is a Euclidean ball, see Examples~\ref{ex:hebisch-sikora-distance} and~\ref{ex:qdist-hebsich_sikora}. Theorem~\ref{thm:BCP-free-step2} extends Theorem~1.14 in~\cite{LeDonne_Rigot_Heisenberg_BCP} to stratified free-nilpotent groups of step 2 and any rank $r\geq 2$. Theorem~1.14 in~\cite{LeDonne_Rigot_Heisenberg_BCP} gives indeed the conclusion for the stratified first Heisenberg group, i.e., the stratified free-nilpotent group of step 2 and rank 2. The proof of Theorem~\ref{thm:BCP-free-step2} is in spirit inspired by the proof of this previous result but requires a slightly different approach.

To prove the non-existence of continuous homogeneous quasi-distances satisfying BCP on graded groups for which there exists two different layers that do not commute, we first consider the case of the non-standard Heisenberg groups, see Theorem~\ref{thm:heisenberg-case}. A non-standard Heisenberg group is the first Heisenberg group viewed as a graded group whose Lie algebra is endowed with a positive grading that is not a stratification, see Example~\ref{ex:heisenberg}. As already mentioned, the fact that there does not exist continuous homogeneous quasi-distances satisfying BCP on non-standard Heisenberg groups was suggested by Theorem~1.6 in~\cite{LeDonne_Rigot_rmknobcp}. However, we stress that the proof of Theorem~\ref{thm:heisenberg-case} is not a technical modification of the arguments in~\cite{LeDonne_Rigot_rmknobcp}. It requires indeed a completely new approach, see Section~\ref{sect:nobcp}.

In both cases, the general conclusion, see Theorems~\ref{thm:sec-BCP-commuting-layers} and~\ref{thm:sect-nobcp-self-similar-noncommuting}, follows using structure properties about graded groups, using submetries, that plays a central role here, and using some constructions on metric spaces that preserve the validity of BCP. These tools are given in Sections~\ref{sect:graded-groups} and~\ref{sect:bcp-vs-wbcp}.

In Section~\ref{sect:graded-groups} we establish preliminary results about graded groups. In Sections~\ref{subsect:graded-groups} and~\ref{subsect:dilations} we fix the definitions and the terminology we shall use throughout the paper for graded and stratified Lie algebras and Lie groups and for the associated families of dilations. Section~\ref{subsect:examples} is devoted to various examples and to the description of some constructions on graded groups for later use. In particular the Heisenberg groups and various positive gradings of their Lie algebra to be used later in the paper are given in Example~\ref{ex:heisenberg}. In Section~\ref{subsect:structure-graded-groups}, we prove some structure properties for graded groups. Proposition~\ref{prop:structure-commuting-layers} gives a description of graded groups with commuting different layers. Proposition~\ref{prop:structure-algebra-noncommuting-layers} explains how every graded group with some different layers not commuting gives rise a non-standard Heisenberg group. Section~\ref{subsec:homogenous-qdist} is devoted to homogeneous quasi-distances. The meaning of the terminology homogeneous groups we use in this paper is in particular given in~Definition~\ref{def:homogeneous-groups}. We stress that working with quasi-distances rather than with distances naturally occurs in applications but may lead to  topological issues. In our setting, we prove that homogeneous quasi-distances on graded groups induce the manifold topology, see Proposition~\ref{prop:homogeneous-quasi-distance-topology}. As a consequence, homogeneous distances on homogeneous groups are continuous, see Corollary~\ref{cor:continuity-homogeneous-distance}. These results seem not to have been previously noticed in the literature and may be of independent interest. On the contrary we stress that homogeneous quasi-distances may or may not be continuous and  Proposition~\ref{prop:continuity-quasi-dist} characterizes continuous homogeneous quasi-distances.

In Section~\ref{sect:bcp-vs-wbcp} we first recall general facts about the Besicovitch Covering Property (BCP) and one of its variants which we call the Weak Besicovitch Covering Property (WBCP). We raise that these two variants are not equivalent in general metric spaces, see Example~\ref{ex:wbcp-but-nobcp}. However, in our setting, and more generally on doubling metric spaces, BCP and WBCP are equivalent, see Proposition~\ref{prop:BCP-WBCP-Doubling}. It turns out that WBCP is for our purposes technically more convenient to work with. Next, we consider some constructions that preserve the validity of (W)BCP. In particular products of metric spaces are considered in Theorem~\ref{thm:WBCP-productspaces}. The role of surjective morphisms of Lie algebra and submetries is given in Propositions~\ref{prop:bcp-submetry} and~\ref{prop:submetries-morphism-graded-algebra}. Results in this section will be used together with the structure properties proved in Section~\ref{subsect:structure-graded-groups} to deduce Theorem~\ref{thm:intro-main-result} from the particular cases mentioned above.

Sections~\ref{sect:BCP-commutative-layers} and~\ref{sect:nobcp} are devoted to the proof of Theorem~\ref{thm:intro-main-result} together with Corollaries~\ref{cor:intro-homogeneous-groups} and~\ref{cor:intro-stratified-groups}. In Section~\ref{sect:BCP-commutative-layers} we prove the existence of continuous homogeneous quasi-distances for which BCP holds on graded groups with commuting different layers, following the scheme already described above, see Theorem~\ref{thm:sec-BCP-commuting-layers}. In Section~\ref{sect:nobcp} we consider more general quasi-distances, called self-similar, which are only required to be one-homogeneous with respect to some dilation, see Definition~\ref{def:selfsimilar-dist}. We prove that continuous self-similar quasi-distances do not satisfy BCP on graded groups for which there exist two different layers of the associated positive grading that do not commute, see Theorem~\ref{thm:sect-nobcp-self-similar-noncommuting}. Self-similar, rather than homogeneous, quasi-distances may occur naturally. We stress that in this case, additional topological issues have to be taken into account, see Section~\ref{subsect:self-similar-topology}.

In Section~\ref{sect:differentiation-measures} we give applications to measure differentiation on graded groups and we prove Theorem~\ref{thm:intro-preiss-improved}.

In Section~\ref{sec:nobcp-ccdist} we consider sub-Riemannian distances on stratified groups. We complete the already known results of~\cite{Rigot} with the following general negative answer. We refer to Definition~\ref{def:sub-riem-carnot-group} for the definition of sub-Riemannian Carnot groups.

\begin{theorem}\label{thm:intro-no:BCP:Carnot}
Let $(G,d)$ be a sub-Riemannian Carnot group of step $\geq 2$. Then BCP does not hold on $(G,d)$.
\end{theorem}

The proof of this result is independent of Theorem~\ref{thm:intro-main-result} but uses some techniques developed in Section~\ref{sect:bcp-vs-wbcp}, in particular Proposition~\ref{prop:submetries-morphism-graded-algebra}. Using the fact that sub-Riemannian Carnot groups appear as metric tangent spaces to sub-Riemannian manifolds, we get the following consequence about measure differentiation on sub-Riemannian manifolds. We refer to Section~\ref{subsect:diff-mesaures-subriemannian-manifolds} for the definition of sub-Riemannian manifold.

\begin{theorem} \label{thm:intro-subriemannian-diff-measures}
Let $M$ be a sub-Riemannian manifold and let $d$ be its sub-Riemannian distance. Then there exists some locally finite Borel measure for which the differentiation theorem on $(M,d)$ does not hold.
\end{theorem}

\subsection*{Acknowledgement}
The authors would like to thank Tapio Rajala 
   for fruitful conversations and improving feedback.





\section{Preliminaries on graded groups} \label{sect:graded-groups}

\subsection{Graded and stratified Lie algebras and Lie groups} \label{subsect:graded-groups}

All Lie algebras considered here are over $\R$ and finite-dimensional. 

\begin{Def}[Positively graduable Lie algebras] \label{def:graduable-algebras} A \emph{positive grading} of a Lie algebra $\g$ is a family $(V_t)_{t\in (0,+\infty)}$ of vector subspaces of $\g$, where all but finitely many of the $V_t$'s are $\{0\}$, such that 
$$\g=  \bigoplus_{t\in (0,+\infty)} V_t $$
and where $[V_s, V_t]\subset V_{s+t}$ for
all $s,t >0.$
Here
$
[V,W] := \Span\{[X,Y] ;\; X\in V,\ Y\in W\} .
$
We say that a Lie algebra is \emph{positively graduable} if it admits a positive grading. 
\end{Def}

A positively graduable Lie algebra may admit several positive gradings that are not isomorphic, see for instance Example~\ref{ex:heisenberg} (and Definition~\ref{def:morphism-graded-algebra} for the definition of isomorphisms of graded Lie algebras). We will use the terminology ``graded Lie algebra'' when considering a positively graduable Lie algebra equipped with a given positive grading as stated in the following definition.

\begin{Def}[Graded Lie algebras] \label{def:graded-algebras} We say that a Lie algebra is \textit{graded} when it is positively graduable and endowed with a given positive grading called the \textit{associated positive grading}.
\end{Def}

Given a positive
grading $\g=  \oplus_{t>0} V_t $ and given $t\in (0,+\infty)$, the subspace $V_t$ is called the \textit{degree-t layer} of the grading.  

Recall that, for a Lie algebra $\g$, the lower central serie is defined inductively by 
$\g^{(1)}=\g$, $\g^{(k+1)}=[\g,\g^{(k)}]$. 
A Lie algebra $\g$ is called {\em nilpotent} if $\g^{(s+1)}=\{0\}$ for some integer $s \geq 1$. We say that $\g$ is {\em nilpotent of step $s$} if $\g^{(s+1)}=\{0\}$ but $\g^{(s)}\neq \{0\}$. Positively graduable Lie algebras are nilpotent (\cite{Goodman_77}, \cite{Folland-Stein}). However, nilpotent Lie algebras that are not positively graduable do exist (\cite{Dyer}).

\begin{Def}[Stratifiable Lie algebras] \label{def:stratifiable-algebras} A \textit{stratification} of step $s$ of a Lie algebra $\g$ is a direct-sum decomposition
$$\g = V_1 \oplus V_2 \oplus \cdots \oplus V_s$$
for some integer $s\geq 1$ where $V_s \not= \{0\}$, $[V_1,V_j] = V_{j+1}$ for all integers $j\in \{1,\dots,s\}$, and where we have set $V_{s+1} := \{0\}$. We say that a Lie algebra is \textit{stratifiable} if it admits a stratification. 
\end{Def}

Equivalently, a stratifiable Lie algebra $\g$ is a positively graduable Lie algebra that admits a positive grading whose degree-one layer generates $\g$ as a Lie algebra. A stratification is uniquely determined by its degree-one layer $V_1$. Moreover, one has $\g = V_1 \oplus [\g,\g]$. Recall that the rank of a nilpotent Lie algebra is defined as $\dim \g - \dim [\g,\g]$. For a stratified Lie algebra it coincides with the dimension of the degree-one layer of any of its stratification. However, we stress that an arbitrary vector space $V$ of a stratifiable Lie algebra $\g$ that is in direct sum with $[\g,\g]$, i.e., satisfies $\g=V\oplus [\g,\g]$, may not generate a stratification, see Example~\ref{ex:graded-vs-stratified}. Note also that a positive grading of a stratifiable Lie algebra may not be a stratification, see Example~\ref{ex:heisenberg}.

Any two stratifications of a Lie algebra are isomorphic, see \cite{LeDonne:Carnot}. In particular they have equal step that we will call the \textit{step of the stratifiable Lie algebra}. Note that a stratifiable Lie algebra of step $s$ is nilpotent of step $s$. When we fix a given stratification of a stratifiable Lie algebra, we will use the terminology ``stratified Lie algebra'' as stated in the following definition. 

\begin{Def}[Stratified Lie algebras] \label{def:stratified-algebras}
We say that a Lie algebra is \textit{stratified} when it is stratifiable and endowed with a given stratification called the \textit{associated stratification}.
\end{Def}

\begin{Def}[Positively graduable, graded, stratifiable, stratified groups] \label{def:graded-stratified-groups} We say that a Lie group $G$ is a \emph{positively graduable} (respectively  \emph{graded}, {\em  stratifiable}, \emph{stratified}) \textit{group} if $G$ is a connected and simply connected Lie group whose Lie algebra is positively graduable (respectively graded, stratifiable, stratified). 
\end{Def}

For the sake of completeness, we will give in Theorem~\ref{thm:siebert} below a characterization of positively graduable groups in terms of existence of a contractive group automorphism. This characterization is due to E.~Siebert.

\subsection{Dilations on graded algebras and graded groups} \label{subsect:dilations}

\begin{Def}[Dilations on graded Lie algebras] \label{def:dilations-algebras}
Let 
$\g$ be a graded Lie algebra with associated positive grading $\g = \oplus_{t>0} V_t $. For $\lambda>0$, 
we define the \emph{associated dilation of factor $\lambda$} as the unique linear map $\delta_\lambda:\g\to\g$ such that $\delta_\lambda (X) = \lambda^t X$ for all $X\in V_t$.
\end{Def}

Dilations $\delta_\lambda:\g\to\g$ are Lie algebra automorphisms. Moreover, the family of dilations $(\delta_\lambda)_{\lambda>0}$ is a one-parameter group of Lie algebra automorphisms, i.e., $\delta_{\lambda}\circ\delta_{\eta}=\delta_{\lambda\eta}$ for all $\lambda, \eta >0$. 
	
Throughout this paper, given a Lie group homomorphism $\varphi:G\rightarrow H$, we will denote by $\varphi_*:\g \rightarrow \mathfrak{h}$ the associated Lie algebra homomorphism. Recall that, if $G$ is simply connected, given a Lie algebra homomorphism $\phi:\g \rightarrow \mathfrak{h}$, there exists a unique Lie group homomorphism $\varphi:G\rightarrow H$ such that $\varphi_* = \phi$ (see \cite[Theorem 3.27]{Warner}). This allows to define dilations on graded groups as stated in the following definition. 

\begin{Def}[Dilations on graded groups] \label{def:dilations_groups}
 	Let $G$ be a graded group with Lie algebra $\g$. 
	Let $\delta_\lambda:\g\to\g$ be the associated dilation of factor $\lambda>0$. 
	The \emph{associated dilation of factor $\lambda$ on $G$} is the unique Lie group automorphism, also denoted by $\delta_\lambda:G\to G$, such that $(\delta_\lambda)_*=\delta_\lambda$.
\end{Def}

For technical simplicity, we keep the same notation for both dilations on the Lie algebra $\g$ and the group $G$. There will be no ambiguity here. Indeed, graded groups being nilpotent and simply connected, the exponential map $\exp: \g \to G$ is a diffeomorphism (see \cite[Theorem~1.2.1]{Corwin-Greenleaf} or \cite[Proposition~1.2]{Folland-Stein}) and one has $\delta_\lambda\circ\exp = \exp\circ\,\delta_\lambda$ (see \cite[Theorem 3.27]{Warner}), hence dilations on $\g$ and dilations on $G$ coincide in exponential coordinates.

For the sake of completeness, we give now a equivalent characterization of positively graduable groups due to Siebert. If $G$ is a topological group with identity $e$ and $\tau:G\rightarrow G$ is a group automorphism, we say that $\tau$ is \textit{contractive} if, for all $g\in G$, one has $\lim_{k\to\infty}\tau^k(g)=e$. We say that $G$ is \textit{contractible} if $G$ admits a contractive automorphism.

For graded groups, associated dilations of factor $\lambda \in (0,1)$ are contractive automorphisms. Hence positively graduable groups are contractible. Conversely, E.~Siebert proved (see Theorem~\ref{thm:siebert} below) that if $G$ is a connected locally compact group and $\tau:G\rightarrow G$ is a contractive automorphism then $G$ is a connected and simply connected Lie group and $\tau$ induces a positive grading on the Lie algebra $\g$ of $G$. Note however that $\tau$ itself  may not be a dilation associated to the induced grading.

\begin{thm} [{\cite[Corollary~2.4]{Siebert}}] \label{thm:siebert} A topological group $G$ is a positively graduable Lie group if and only if $G$ is a connected locally compact contractible group.
\end{thm}

\subsection{Examples} \label{subsect:examples}

We first introduce the definition of basis adapted to a positive grading for later use. 

\begin{definition} \label{def:basis-adapted}
Let $\g$ be a graded Lie algebra with associated positive grading given by $\g = \oplus_{t>0} V_t$. Let $n:= \dim \g$ and $0<  t_1 < \dots < t_l$ be such that $V_{t_i} \not= \{0\}$ for all $i=1,\dots,l$ and $V_t = \{0\}$ for all $t\not\in \{t_1,\dots,t_l\}$. We say that a basis $(X_1, \dots , X_n)$ of $\g$ is adapted to the positive grading if $(X_{m_{i-1}+1},\dots,X_{m_i})$ is a basis of $V_{t_i}$ for all $1\leq i \leq l$. Here $m_0 = 0$ and $m_i - m_{i-1} = \dim V_{t_i}$.
\end{definition}

\begin{example} [Abelian Lie algebras and Lie groups] \label{ex:abelian}
Abelian Lie algebras are stratifiable Lie algebras of step 1 and any direct-sum decomposition is a positive grading. 

In particular, if $\g$ is an Abelian $n$-dimensional Lie algebra, the trivial direct-sum decomposition $\g =  \oplus_{t>0} V_t$ where $V_1 = \g$ and $V_t = \{0\}$ for all $t\not=1$ gives the stratification. The connected and simply connected Lie group with Lie algebra $\g$ can be identified with $\R^n$ equipped with the Abelian group law. Associated dilations are the usual multiplication by a scalar positive real number given by $x  \mapsto \lambda x$ for $\lambda >0$.

More generally, for any real numbers $0< d_1 \leq \cdots \leq d_n$, the maps $$(x_1,\dots,x_n) \mapsto (\lambda^{d_1} x_1, \dots , \lambda^{d_n} x_n)~,$$
define a family of dilations on the Abelian group $\R^n$ associated to some positive grading of its Lie algebra.
\end{example}

\begin{example} [Heisenberg Lie algebras and Lie groups]\label{ex:heisenberg}
The $n$-th Heisenberg Lie algebra $\h_n$ is the $(2n+1)$-dimensional Lie algebra that admits a basis $(X_1,\dots,X_n,Y_1,\dots,Y_n,Z)$ where the only non-trivial bracket relations are $[X_j,Y_j] = Z$ for all $1\leq j \leq n$. We call such a basis a \textit{standard basis} of $\h_n$.

The $n$-th Heisenberg group $\HH ^n$ is the connected and simply connected Lie group whose Lie algebra is $\h_n$. Using exponential coordinates of the first kind, we write $p \in \HH^n$ as $p=\exp((\sum_{j=1}^n x_j X_j + y_j Y_j) + z Z) $ and we identify $p$ with $(x_1,\dots,x_n,y_1,\dots,y_n,z)$. Using the Baker-Campbell-Hausdorff formula, the group law is given by 
\begin{multline*}\label{e:law-heisenberg}
(x_1,\dots,x_n,y_1,\dots,y_n,z)\cdot (x_1',\dots,x_n',y_1',\dots,y_n',z') \\
= (x_1 + x_1', \dots, x_n + x_n',y_1+y_1',\dots , y_n + y_n', z+z'+\frac{1}{2} \sum_{j=1}^n (x_j y_j' - y_j x_j'))~.
\end{multline*}

Heisenberg Lie algebras are stratifiable of step 2. Using a standard basis, 
\begin{equation*} 
\h_n = V_1 \oplus V_2 \quad \text{where }\; V_1:=\Span\{X_j,Y_j;\; 1\leq j \leq n\},\;  V_2:=\Span{Z}
\end{equation*}
is a stratification.
Dilations associated to this stratification are given by
\begin{equation*} 
(x_1,\dots,x_n,y_1,\dots,y_n,z) \mapsto (\lambda x_1,\dots,\lambda x_n,\lambda y_1,\dots,\lambda y_n,\lambda^2 z)~.
\end{equation*}

Heisenberg Lie algebras also admit positive gradings that are not stratifications. We will in particular consider in this paper such gradings on the first Heisenberg Lie algebra. Namely, for $\alpha \in (1,+\infty)$, we call \textit{non-standard Heisenberg Lie algebra of exponent $\alpha$} the first Heisenberg Lie algebra equipped with the following \textit{non-standard grading of exponent} $\alpha$
\begin{equation*} 
\h_1 =W_1 \oplus W_\alpha \oplus W_{\alpha+1}
\end{equation*}
where $$ W_1:=\Span\{X_1\} ,\;\; W_\alpha:=\Span\{Y_1\},\;\; W_{\alpha+1}:=\Span\{Z\}~,$$
and where $(X_1,Y_1,Z)$ is a standard basis of $\h_1 $. Note that up to isomorphisms of graded Lie algebras (see Definition~\ref{def:morphism-graded-algebra}) and up to powers (see Example~\ref{ex:k-power}), these non-standard gradings give all the possible positive gradings of $\h_1$ that are not a stratification.
Dilations associated to the non-standard grading of exponent $\alpha$ are given by 
\begin{equation*} 
(x_1,y_1,z) \mapsto (\lambda x_1, \lambda^\alpha y_1,\lambda^{\alpha+1} z)~.
\end{equation*}
We will use the terminology \textit{non-standard Heisenberg group (of exponent $\alpha$)} when considering the first Heisenberg group as a graded group whose Lie algebra is endowed with the non-standard grading (of exponent $\alpha$).

\end{example}

\begin{example} \label{ex:graded-vs-stratified} There are examples of stratifiable Lie algebras $\g$ for which one can find a subspace $V$ in direct sum with $[\g,\g]$ but that does not generate a stratification. One can for instance consider the stratifiable Lie algebra $\g$ of step 3 generated by $e_1$, $e_2$ and $e_3$ and with the relation $[e_2,e_3]=0$. If  $V:=\Span\{e_1, e_2 + [e_1,e_2], e_3\}$, one has $\g = V \oplus [\g,\g]$ but $V$ does not generate a stratification of $\g$, see
\cite{LeDonne:Carnot}.

\end{example}

\begin{example} [Direct product of graded Lie groups] \label{ex:direct-product} The direct product of graded groups is positively graduable. If $G$ and $H$ are graded groups with associated positive grading of their Lie algebras given by $\g=\oplus_{t>0} V_t$ and $\h = \oplus_{t>0} W_t$ respectively, then $\oplus_{t>0} (V_t\oplus W_t)$ is a positive grading of the Lie algebra of $G \times H$. This can be extended to the direct product of finitely many graded groups in the obvious way. In the rest of this paper, we will always consider the direct product of graded groups as graded groups with associated positive grading given by the above mentioned grading.
\end{example}

\begin{example} [Power of graded Lie algebras and of graded groups] \label{ex:k-power}
Let $\g$ be a graded Lie algebra with associated positive grading $\g = \oplus_{s>0} V_s$ and let $t>0$ be a real positive number. Then $\g=\oplus_{s>0} W_s$ where $W_{ts}:=V_s$ is a positive grading of $\g$, which we call the \textit{$t$-power} of the initial positive grading. 

For $\lambda >0$, let $\delta_\lambda$ denote the dilation of factor $\lambda$ associated to the initial positive grading and $\tilde\delta_\lambda$ denote the dilation of factor $\lambda$ associated to the grading of its $t$-power. In exponential coordinates of the first kind associated to a basis adapted to these gradings, we have $\delta_\lambda (x_1,\dots,x_n) = (\lambda^{s_1} x_1,\dots,\lambda^{s_n} x_n)$ and $\tilde \delta_\lambda (x_1,\dots,x_n) = (\lambda^{ts_1} x_1,\dots,\lambda^{ts_n} x_n)$ for some $0<s_1 \leq \cdots \leq s_n$. 

If $G$ is a graded group, we call \textit{$t$-power of $G$} the group $G$ considered as the graded group whose Lie algebra is endowed with the $t$-power of the initial positive grading.
\end{example}

\subsection{Structure of graded algebras and graded groups} \label{subsect:structure-graded-groups}
We give in this section some results about the structure of graded algebras and groups to be used later in this paper. They may be more generally of independent interest. 

First, we consider graded groups with commuting different layers, see Definition~\ref{def:groups-commuting-layers}. Notice that for such graded groups, a layer of the positive grading may not commute with itself.

\begin{proposition} \label{prop:structure-commuting-layers}
Let $G$ be a graded group with commuting different layers. Then $G$ is the direct product of powers of stratified groups of step $\leq 2$.
\end{proposition}

See Example~\ref{ex:k-power} for the definition of powers of a graded group, Example~\ref{ex:direct-product} for the definition of direct product graded groups  and Definition~\ref{def:graded-stratified-groups} for the definition of stratified groups.

\begin{proof}
Let $\g$ denote the Lie algebra of $G$. Let $0<t_1<\cdots < t_m$ be such that $V_{t_k} \not= \{0\}$ for all $k=1,\dots,m$ and $V_t = \{0\}$ for all $t\not\in \{t_1,\dots,t_m\}$. We have $\g=V_{t_1} \oplus V_{t_2}\oplus \ldots \oplus V_{t_m}$. 

If $[V_{t_1},V_{t_1}] = \{0\}$ then $[V_{t_1},\g] = \{0\}$ since $[V_{t_1},V_s] = \{0\}$ for all $s\not= t_1$. It follows that $\mathfrak{h}:=V_{t_1}$ and $\mathfrak{h}':=V_{t_2} \oplus \ldots \oplus V_{t_m}$ are ideals of $\g$. Hence (see~\cite[p.388]{Pontryagin_topgrp}) $G$ is the direct product of $\exp(\mathfrak h)$ and $\exp(\mathfrak h')$. Moreover $\exp(\mathfrak{h})$ is the $t_1$-power of an Abelian stratified group. 

If $[V_{t_1},V_{t_1}] \not= \{0\}$, we set $\mathfrak h:= V_{t_1}\oplus [V_{t_1},V_{t_1}] \subseteq V_{t_1}\oplus V_{2t_1}$. Then we consider $V_{2t_1}'$ any complement of $[V_{t_1},V_{t_1}] $ in  $V_{2t_1}$, i.e.,
$[V_{t_1},V_{t_1}] \oplus V_{2t_1}' = V_{2t_1}$. We set $V_{t_1}'=\{0\}$, $V_t' = V_t$ for $t\not = t_1,2 t_1$ and $\mathfrak h':= \oplus_{t>0} V_t'$. We have 
$\g = \mathfrak h\oplus \mathfrak h'$. We prove now that $G$ is the direct product of $\exp(\mathfrak h)$ and $\exp(\mathfrak h')$. As before, to get the conclusion, we prove that both $\mathfrak h$ and $ \mathfrak h'$ are ideals of $\g$. 

To show that $\mathfrak h$ is an ideal, take $X\in \mathfrak h$ and $Y\in \g$. It is enough to consider the following 4 cases. First, assume $X\in V_{t_1}$ and $Y\in V_{t_1}$. Then $[X,Y]  \in  \mathfrak h$ by definition of $\mathfrak{h}$. Second, assume $X\in V_{t_1}$ and $Y\in V_{t}$, with $t\neq t_1$. Then $[X,Y] =0$ since $G$ has commuting different layers. Third, assume $X=[X_a, X_b]$ with 
$X_a, X_b\in V_{t_1}$ and $Y\in V_{2 t_1}$. By Jacobi's identity and since ${2 t_1}\neq t_1$ and $G$ has commuting different layers, we get
$[[X_a, X_b],Y] =[[X_a, Y], X_b] + [[Y,X_b], X_a] =0$. By bilinearity of the Lie bracket, it follows that $[X,Y]=0 \in \mathfrak{h}$ for all $X\in [V_{t_1},V_{t_1}]$ and $Y\in V_{2 t_1}$. Finally, assume $X\in [V_{t_1},V_{t_1}]$ and $Y\in V_{t}$ with $t\neq 2 t_1$. Then $[X,Y] =0$ since $G$ has commuting different layers. All together it follows that $[X,Y] \in \mathfrak{h}$ for all $X\in \mathfrak{h}$ and all $Y\in \g$ hence $\mathfrak{h}$ is an ideal.

To show that $\mathfrak h'$ is an ideal, we note that if $X\in V_t$ with $t>t_1$ and $Y\in V_s$ for some $s\geq t_1$, we have $t+s > 2t_1$ and hence $[X,Y]\in \oplus_{l>2t_1} V_l  = \oplus_{l>2t_1} V_l' \subset  \mathfrak h'$. Since $\mathfrak{h}' \subset \oplus_{t>t_1} V_t$ and by bilinearity of the Lie bracket, it follows that $[\mathfrak{h}',\g] \subset \mathfrak{h}'$ hence $\mathfrak h'$ is an ideal. It follows that $G$ is the direct product of $\exp(\mathfrak h)$ and $\exp(\mathfrak h')$. Moreover $\exp(\mathfrak h)$ is the $t_1$-power of a stratified group of step 2.

Finally, arguing by induction on the dimension of $\g$, we get the conclusion.
\end{proof}

We will next consider the case where there exist two different layers of the positive associated grading that do not commute. We first introduce the notions of morphisms of graded Lie algebras and of graded subalgebras. 

\begin{definition} [Morphism of graded Lie algebras] \label{def:morphism-graded-algebra}
Let $\g =\oplus_{t>0} V_t$ and $\h =\oplus_{t>0} W_t$ be graded Lie algebras. We say that $\phi:\g\rightarrow\h$ is a \textit{morphism (respectively isomorphism) of graded Lie algebras} if $\phi$ is a Lie algebra homomorphism (respectively isomorphism) such that $\phi(V_t) \subset W_t$ for all $t>0$.
\end{definition}

Let $\g =\oplus_{t>0} V_t$ be a graded Lie algebra with associated dilations $(\delta_\lambda)_{\lambda>0}$. We say that a Lie subalgebra $\hat{\g}$ of $\g$ is \textit{homogeneous} if $\delta_\lambda(\hat{\g}) = \hat{\g}$ for all $\lambda>0$. If $\hat{\g}$ is a homogeneous Lie subalgebra of $\g$, then $\oplus_{t>0} (V_t \cap \hat{\g})$ is a positive grading of $\hat{\g}$ called the \textit{induced positive grading} and associated dilations are the restriction of $\delta_\lambda$ to $\hat{\g}$.

\begin{definition} [Graded subalgebra] \label{def:graded-subalgebra}
Let $\g$ be a graded Lie algebra. We say that $\hat{\g}$ is a \textit{graded subalgebra of the graded algebra} $\g$ if $\hat{\g}$ is a homogeneous Lie subalgebra of $\g$ endowed with the induced positive grading.
\end{definition}

\begin{proposition} \label{prop:structure-algebra-noncommuting-layers}
Let $\g =\oplus_{t>0} V_t$ be a graded Lie algebra. Assume that $[V_t,V_s]\not = \{0\}$ for some $t<s$. Then there exist a graded subalgebra $\hat{\g}$ of $\g$ and a surjective morphism of graded Lie algebras from $\hat{\g}$ to $\h$ where $\h$ is the $t$-power of the non-standard Heisenberg Lie algebra of exponent $s/t$.
\end{proposition}

See Example~\ref{ex:heisenberg} for the definition of non-standard Heisenberg Lie algebras and Example~\ref{ex:k-power} for the definition of the $t$-power of a graded Lie algebra.

\begin{proof}
Let $X_1\in V_t$ and $X_2\in V_s$ be such that $[X_1,X_2]\not= 0$. Let $\hat{\g}$ denote the Lie subalgebra of $\g$ generated by $X_1$ and $X_2$. We have $\delta_\lambda(\hat{\g})  =\hat{\g}$ for all $\lambda >0$ where $\delta_\lambda$ are the associated dilations on $\g$. Hence $\hat{\g}$ is homogeneous. We endow it with the induced positive grading $\hat{\g} = \oplus_{u>0} \hat{V}_u$ where
$\hat{V}_u := V_u \cap \hat{\g}$ to make it a graded subalgebra of $\g$. We have
$$\hat{\g} = \hat{V}_t \oplus \hat{V}_s \oplus \hat{V}_{t+s} \oplus (\oplus_{u>t+s} \hat{V}_u)$$
with $\hat{V}_t = \Span\{X_1\}$, $\hat{V}_s = \Span\{X_2\}$ and $\hat{V}_{t+s} = \Span\{X_3\}$ where $X_3:=[X_1,X_2]$. 
Let $n:=\dim \hat{\g}$ and $(X_4,\dots,X_n)$ be a basis of $\oplus_{u>t+s} \hat{V}_u$ such that $(X_1,\dots,X_n)$ is a basis of $\hat{\g}$ adapted to its positive grading (see Definition~\ref{def:basis-adapted}).

Let $\h:= W_t \oplus W_s \oplus W_{s+t}$ be the $t$-power of the non-standard  Heisenberg Lie algebra of exponent $s/t$ (see Example~\ref{ex:heisenberg} and Example~\ref{ex:k-power}). Let $Y_1 \not = 0 \in W_t$, $Y_2 \not= 0 \in W_s$, and set $Y_3:=[Y_1,Y_2]$. Then $W_{s+t} = \Span \{Y_3\}$.

Let $\phi:\hat{g}\rightarrow \h$ be the linear map defined by $\phi(X_i):= Y_i$ for $i=1,2,3$ and $ \phi(X_i):=0$ for $i\geq 4$. It can easily be checked that $\phi$ is a Lie algebra homomorphism. It can also easily be checked that $\phi(\hat{V}_u) = W_u$ for all $u>0$. Here we set $W_u:=\{0\}$ for $u\not\in\{t,s,s+t\}$. Hence $\phi$ is a surjective morphism of graded Lie algebras.
\end{proof}

\begin{remark} \label{rmk:free-vs-morphism}
To conclude this section, let us mention the following general fact. For any graded Lie algebra $\g$, there exist a positive grading of a free-nilpotent Lie algebra $\mathfrak{f}$ and a surjective morphism $\phi:\mathfrak{f} \rightarrow \g$ of graded Lie algebras. We shall use this fact in  the simple case of stratified Lie algebras of step 2 where all what we need can be easily constructed by hand, see the proof of Theorem ~\ref{thm:BCP-stratified-step2}. We refer to~\cite[Chapter II]{Bourbaki_Lie},~\cite{Jacobson},~\cite{Serre_Lie},~\cite{Reutenauer},
~\cite{Drutu-Kapovich} for more details on the subject.
\end{remark}

\subsection{Homogeneous quasi-distances} \label{subsec:homogenous-qdist}

Given a nonempty set $X$, we say that $d:X\times X \rightarrow [0,+\infty)$ is a quasi-distance on $X$ if it is symmetric, $d(p,q) =0$ if and only if $p=q$, and there exists a constant $C\geq 1$ such that $d(p,q) \leq C(d(p,p') + d(p',q))$ for all $p$, $p'$, $q\in X$ (quasi-triangle inequality with multiplicative constant $C$). We call $(X,d)$ a quasi-metric space. When speaking of a ball $B$ in $(X,d)$, it will be understood that $B$ is a set of the form $B=B_{d}(p,r)$ for some $p\in X$ and some $r>0$ where $B_{d}(p,r) := \{q\in X;\; d(q,p)\leq r\}$. When $d$ satisfies the triangle inequality, i.e., the quasi-triangle inequality with a multiplicative constant $C=1$, then $d$ is a distance on $X$.

\begin{definition} [Homogeneous quasi-distances on graded groups]\label{def:homogeneous-dist2} 
Let $G$ be a graded group with associated dilations $(\delta_\lambda)_{\lambda>0} $.
We say that a quasi-distance $d$ on $G$ is 
 \textit{homogeneous} if $d$ is left-invariant, i.e., $d(p\cdot q,p\cdot q') = d(q,q')$ for all $p$, $q$, $q'\in G$, 
 and one-homogeneous with respect to all dilations $(\delta_\lambda)_{\lambda>0} $, i.e., $d(\delta_\lambda(p),\delta_\lambda(q)) = \lambda d(p,q)$ for all $p$, $q\in G$ and all $\lambda>0$.
\end{definition}

Note that, in this definition, we do not require any topological property, and in particular any continuity property, of a homogeneous quasi-distance with respect to the manifold topology on the group. We will discuss these topological issues below, see Proposition~\ref{prop:homogeneous-quasi-distance-topology}, Corollary~\ref{cor:continuity-homogeneous-distance}, and Proposition~\ref{prop:continuity-quasi-dist}. Let us stress that we will consider in Section~\ref{sect:nobcp} a more general class of quasi-distances, called self-similar quasi-distances in the present paper. Additional topological issues occur for self-similar quasi-distances. For the sake of clarity, we devote the present section to homogeneous quasi-distances. We postpone the discussion about topological properties of self-similar quasi-distances to Section~\ref{sect:nobcp}, see especially Section~\ref{subsect:self-similar-topology}.

Homogeneous quasi-distances on arbitrary graded groups do exist. One can for instance follow the arguments in~\cite[Chapter 1]{Folland-Stein}. Note however that our terminology is slightly different from the terminology adopted for graded groups in~\cite{Folland-Stein}. See also Example~\ref{ex:quasi-dist-graded} below for another construction of homogeneous quasi-distances on arbitrary graded groups. 

On the other hand, homogeneous distances do exist if and only if, for all $t<1$, degree-$t$ layers of the associated positive grading are $\{0\}$. These groups are called homogeneous in~\cite{Folland-Stein} and we will follow here this terminology.

\begin{definition} [Homogeneous groups]
\label{def:homogeneous-groups}
We say that $G$ is a homogeneous group if $G$ is a graded group whose associated positive grading $\oplus_{t>0} V_t$ of its Lie algebra is such that $V_t = \{0\}$ for all $t\in (0,1)$.
\end{definition}

As already mentioned, we have the following proposition.

\begin{proposition} \label{prop:homogeneous-groups}
Let $G$ be a graded group. There exists a homogeneous distance on $G$ if and only if $G$ is a homogeneous group.
\end{proposition}

\begin{proof} If some degree-$t$ layer of the positive grading is non-trivial for some $t<1$, a map $d:G \times G \rightarrow [0,+\infty)$ that is left-invariant and one-homogeneous with respect to some non-trivial associated dilation cannot satisfy the triangle inequality (i.e., the quasi-triangle inequality with a multiplicative constant $C=1$). On  the other hand, W.~Hebisch and A.~Sikora proved in~\cite{Hebisch-Sikora} the existence of homogeneous distances on homogeneous groups. 
\end{proof}

Homogeneous distances considered by Hebisch and Sikora play a central role in Section~\ref{sect:BCP-commutative-layers} and we describe them below.

\begin{example} [Hebisch-Sikora's homogeneous distances on homogeneous groups]
\label{ex:hebisch-sikora-distance}
Let $G$ be a homogeneous group with identity $e$, associated positive grading of its Lie algebra given by $\g = \oplus_{t>0} V_t $ and associated dilations $(\delta_\lambda)_{\lambda >0}$. Let $n:= \dim \g$ and let $(X_1, \dots , X_n)$ be a basis of $\g$ adapted to the positive grading (see Definition~\ref{def:basis-adapted}). Using exponential coordinates of the first kind, we identify $p \in G$ with $(p_1,\dots,p_n)$ where $p = \exp (p_1 X_1 + \dots + p_n X_n)$. For $R>0$, we set $$A_R := \left\{p\in G;\; \sum_{i=1}^n p_i^2 \leq R^2\right\}$$ and $$d_R(p,q) := \inf\left\{\lambda >0;\; \delta_{1/\lambda} (p^{-1} \cdot q) \in A_R\right\}~.$$
Hebisch and Sikora proved in~\cite{Hebisch-Sikora} that there exists $R^*>0$ such that for all $0 < R < R^*$, $d_R$ defines a homogeneous distance on $G$. 
\end{example}

\begin{example} [Homogeneous quasi-distances on powers of graded groups]
\label{ex:homogeneous-qdist-kpower}
If $G$ is a graded group, $d$ a homogeneous quasi-distance on $G$ and $t\in (0,+\infty)$ then $d^{1/t}$ is a homogeneous quasi-distance on its $t$-power (see Example~\ref{ex:k-power} for the definition of powers of graded groups). Notice that when $G$ is a homogeneous group, $d$ a homogeneous distance and $t>1$, then $(G,d^{1/t})$ is a snowflake of $(G,d)$.
\end{example}

\begin{example} [Existence of homogeneous quasi-distances on graded groups]
\label{ex:quasi-dist-graded}
Let $G$ be a graded group with associated positive grading $\oplus_{s>0} V_s$ of its Lie algebra. All $t$-powers of $G$ where $t \min\{s>0;\; V_s \not = \{0\}\} \geq 1$ are homogeneous groups. If $d$ is a homogeneous distance on such a $t$-power of $G$, it follows from Example~\ref{ex:homogeneous-qdist-kpower} that $d^t$ is a homogeneous quasi-distance on $G$.
\end{example}

A quasi-distance $d$ on a set $X$ induces a topology on $X$ declaring a set $O$ to be open if and only if for all $x\in O$ there exists  $r>0$ such that $B_d(x,r)\subset O$. Equivalently a set $F$ is said to be closed if and only if for all sequences $(x_k)$ of points in $F$ such that $d(x_k,x)$ converges to $0$ for some $x\in X$, we have $x\in F$. In the following proposition, we prove that on a graded group the topology induced by a homogeneous quasi-distance and the manifold topology coincide.

\begin{proposition} \label{prop:homogeneous-quasi-distance-topology}
The topology induced by a homogeneous quasi-distance on a graded group coincides with the manifold topology of the group.
Moreover, a set is relatively compact 
if and only if it is bounded  with respect to the homogeneous quasi-distance.
\end{proposition}

\begin{proof}
Let $G$ be a graded group with identity $e$, associated positive grading of its Lie algebra given by $\g = \oplus_{t>0} V_t $ and associated dilations $(\delta_\lambda)_{\lambda >0}$. Let $d$ be a homogeneous quasi-distance satisfying the quasi-triangle inequality with multiplicative constant $C$. 

To prove that the topology $\mathcal{T}_d$ induced by $d$ and the manifold topology $\mathcal{T}_m$ coincide, it is enough to show that $d(e,p)$ goes to 0 if and only if $p$ converges to $e$. Here, and in the rest of this proof, the latter convergence (and more generally the convergence of some sequence of points) means convergence with respect to the manifold topology.

First, we show that the quasi-distance $d(e,\cdot)$ from $e$ is continuous at $e$ with respect to $\mathcal{T}_m$. Since graded groups are nilpotent and simply connected, we can consider exponential coordinates of second kind with respect to a suitable choice of basis $(X_1, \ldots, X_n)$ of $\g$, see~\cite{Corwin-Greenleaf}. Namely, first for all $i=1,\ldots, n$, there exists $d_i>0$ such that $X_i\in V_{d_i}$, and consequently, $ \delta_\lambda (X_i)= \lambda^{d_i} X_i$. Second, there is a diffeomorphism $p \mapsto (P_1(p) ,  \ldots,   P_n(p) )$  from $G$ onto $\R^n$ such that for all $p\in G$, $$p= \exp (P_1(p) X_1)  \cdot \ldots \cdot \exp (P_n(p) X_n)~.$$ In particular $P_i(e)=0$ for all $i=1,\ldots, n$. Then, using the quasi-triangle inequality, the left-invariance and the homogeneity of the quasi-distance, we get
\begin{equation*}
\begin{split}
d(e,p)&= d(e,\exp (P_1(p) X_1)  \cdot\ldots\cdot \exp (P_n(p) X_n) )   \\
&\leq  \sum_{i=1}^n C^i\;  d(e, \exp (P_i(p) X_i) ) \\
&\leq \sum_{i=1}^n C^i\; d(e, \exp ( \delta_{|P_i(p)|^{1/d_i}} (\operatorname{sgn} (P_i(p))X_i) ) \\
&= \sum_{i=1}^n C^i\; | P_i(p) | ^{1/d_i}\; d(e, \operatorname{sgn} (P_i(p))\exp ( X_i) ),
\end{split}
\end{equation*}
where $\operatorname{sgn} (P_i(p))$ denotes the sign of $P_i(p)$. The last upper bound goes to 0 when $p$ converges to $e$ and this proves the claim.

Second, we show that if $(p_k)$ is a sequence for which $d(e,p_k)$ goes to 0 then $p_k$ converges to $e$. We set 
\begin{equation} \label{e:norm}
\norm{p }:= \sum_{i=1}^n | P_i(p) |~.
\end{equation}
Arguing by contradiction, up to a subsequence, there would exist $\eps>0$ such that 
$\norm{p_k} >\eps$ for all $k$. Since the map $\lambda\mapsto \norm{\delta_\lambda(q)}$ is continuous with respect to $\mathcal{T}_m$ for all $q\in G$, we get that, for all $k$, one can find $\lambda_k \in (0,1)$ such that 
$\norm{ \delta_{\lambda_k}(p_k)} =\eps$.
By compactness with respect to $\mathcal{T}_m$ of $\{p\in G;\; \norm{p}=\eps\}$, up to a subsequence,  we would get that 
$ \delta_{\lambda_k}(p_k)$ converges to some  $q\in G$ with $\norm{q}=\eps$. In particular $q\neq e$ and hence $d(e, q)>0$.
However,
\begin{eqnarray*}
0< d(e, q)&\leq & C (d(e, \delta_{\lambda_k}(p_k))  +  d(\delta_{\lambda_k}(p_k), q)) \\
&= & C (\lambda_k d(e,  p_k)  +  d(e,q^{-1}\cdot \delta_{\lambda_k} (p_k)) )\\
&\leq & C (d(e,  p_k) + d(e,q^{-1}\cdot \delta_{\lambda_k} (p_k)) )~.
\end{eqnarray*}
By continuity of $d(e,\cdot)$ at $e$ with respect to $\mathcal{T}_m$ and since $q^{-1}\cdot \delta_{\lambda_k}(p_k)$ converges to $e$, we have that $d(e,q^{-1}\cdot \delta_{\lambda_k}(p_k))$ goes to 0. Hence the last upper bound in the above inequalities goes to 0. This gives a contradiction and proves the claim. All together we get that both topologies coincide.

Relative compactness with respect to the manifold topology is equivalent to boundedness with respect to $\norm{\cdot}$. Hence, to show that relative compactness is equivalent to boundedness with respect to the quasi-distance, we show that boundedness with respect to $d$ and boundedness with respect to $\norm{\cdot}$ are equivalent. By contradiction, assume that one can find a sequence $(p_k)$ such that, for some $M>0$, $d(e,p_k)\leq M$ for all $k$, but $\norm{p_k}$ goes to $+\infty$. Arguing as above, one can find a positive sequence $(\lambda_k)$ converging to 0 such that $\norm{\delta_{\lambda_k} (p_k)} = 1$ for all $k$. On the other hand, since $d(e,\delta_{\lambda_k} (p_k)) =\lambda_k d(e,p_k) \leq M \lambda_k$, $d(e,\delta_{\lambda_k} (p_k))$ goes to 0. As shown before, this implies that $\delta_{\lambda_k} (p_k)$ converges to $e$ and gives a contradiction. One shows in a similar way that boundedness with respect to $\norm{\cdot}$ implies boundedness with respect to the quasi-distance.
\end{proof}

In the rest of this paper, when not specified, all topological properties on graded groups will be understood as defined with respect to the manifold topology, or, equivalently in view of Proposition~\ref{prop:homogeneous-quasi-distance-topology}, with respect to the topology induced by any homogeneous quasi-distance.

If $d$ is a distance on a set $X$, then $d$ is continuous on $X\times X$ with respect to the topology it induces and subsets of the form $\{y\in X;\; d(y,x) < r\}$, respectively $\{y\in X;\; d(y,x) \leq r\}$, are open, respectively closed, with respect to the topology induced by $d$. In particular, in view of Proposition~\ref{prop:homogeneous-quasi-distance-topology}, we have the following consequence.

\begin{corollary} \label{cor:continuity-homogeneous-distance}
Every homogeneous distance on a homogeneous group $G$ is continuous on $G\times G$ with respect to the manifold topology. 
\end{corollary}

On the contrary, we stress that, if $d$ is a quasi-distance on a set $X$, subsets of the form $\{y\in X;\; d(y,x) < r\}$, respectively $\{y\in X;\; d(y,x) \leq r\}$, may not be open, respectively closed, for the topology induced by $d$. Moreover, even when one of these properties holds, the quasi-distance may not be continuous on $X \times X$ with respect to the topology it induces, see Remark~\ref{rmk:continuity-dist} below. In particular, in view of Proposition~\ref{prop:homogeneous-quasi-distance-topology}, a homogeneous quasi-distance on a graded group $G$ may not be continuous on $G\times G$ with respect to the manifold topology. The first part of the proof of Proposition~\ref{prop:homogeneous-quasi-distance-topology} only says that the quasi-distance $d(e,\cdot)$ from $e$ is continuous at $e$ (or equivalently by left-invariance, the quasi-distance $ d(q,\cdot)$ from $q$ is continuous at $q$ for all $q\in G$). In other terms, it only says that a ball $B_d(p,r)$ contains its center $p$ in its interior. We show in the next proposition that a homogeneous quasi-distance on a graded group $G$ is continuous on $G\times G$ if and only if its spheres are closed, or equivalently its unit sphere centered at $e$ is closed . 

\begin{proposition} [Continuity of homogeneous quasi-distances]
 \label{prop:continuity-quasi-dist}
 A homogeneous quasi-distance $d$ on a graded group $G$ is continuous on $G\times G$ if and only if its unit sphere centered at $e$ is closed.
 \end{proposition}
 
\begin{proof}
First, note that a homogeneous quasi-distance $d: G\times G \rightarrow [0,+\infty)$ is continuous on $G\times G$ if and only if the quasi-distance $ d(e,\cdot)$ from $e$ is continuous on $G$. If $d(e,\cdot)$ is continuous on $G$ then its unit sphere $S_d(e,1) := \{p\in G;\; d(e,p) = 1\}$ is obviously closed. Conversely assume that $S_d(e,1)$ is closed. We already know from the proof of Proposition~\ref{prop:homogeneous-quasi-distance-topology} that $d(e,\cdot)$ is continuous at $e$. To prove the continuity of $d(e,\cdot)$ at an arbitrary point $p$ with $p\not=e$, it is sufficient to show that for any sequence $(p_k)$ such that $d(p_k,p)$ goes to $0$, one can extract a subsequence whose quasi-distance to $e$ converges to $d(e,p)$. Set $\lambda_k:=d(e,p_k)$. We have $\lambda_k \leq C(d(e,p) + d(p,p_k))$ hence the sequence $(\lambda_k)$ is bounded. Up to a subsequence, one can thus assume that $\lambda_k$ converges to some $\lambda\geq 0$. If $\lambda = 0$, then $\lambda_k = d(e,p_k)$ goes to $0$, hence $p_k$ converges to $e$ (see Proposition~\ref{prop:homogeneous-quasi-distance-topology}) and we would have $p=e$. Since we are considering $p\not = e$, we thus have $\lambda >0$. Then $\norm{\delta_{1/\lambda_k} (p_k) - \delta_{1/\lambda} (p)}$ goes to 0 (remember~\eqref{e:norm} for the definition of $\norm{\cdot }$), i.e., $\delta_{1/\lambda_k} (p_k)$ converges to $\delta_{1/\lambda} (p)$. Since $\delta_{1/\lambda_k} (p_k) \in S_d(e,1)$ and $S_d(e,1)$ is closed we get that  $\delta_{1/\lambda} (p) \in S_d(e,1)$. Hence $\lambda = d(e,p)$ which concludes the proof.
\end{proof}

\begin{remark} \label{rmk:continuity-dist} In view of Proposition~\ref{prop:continuity-quasi-dist}, it is easy to construct examples of homogeneous quasi-distances that are not continuous. In such a case, the unit ball centered at the identity may or may not be closed. For instance, in $\R^2$ equipped with its trivial abelian stratification, consider the homogeneous quasi-distance $d_1$ whose unit ball centered at the origin is the union of the Euclidean closed unit disk centered at the origin with the interval $[-2,2]\times \{0\}$ (see Example~\ref{ex:qdist-unitball} below for a characterization of homogeneous quasi-distances on graded groups in terms of their unit ball). Its unit sphere at the origin is the union of two points $\{(-2,0), (2,0)\}$ with the set $\{(x,y)\in \R^2;\; x^2 +y^2 = 1\} \setminus \{(1,0),(-1,0)\}$, which is not closed. In this example, the unit ball (and hence any ball) is (are) closed and sets of the form $\{p\in\R^2;\; d_1(0,p) < r\}$ are not open. One can also consider the homogeneous quasi-distance $d_2$ whose unit ball centered at the origin is the Euclidean closed unit disk centered at the origin minus the segments $[-1,-1/2) \cup (1/2,1]$, which is not closed. Its unit sphere at the origin is the union of two points $\{(-1/2,0), (1/2,0)\}$ with the set $\{(x,y)\in \R^2;\; x^2 +y^2 = 1\} \setminus \{(1,0),(-1,0)\}$, which is not closed as well. However sets of the form $\{p\in\R^2;\; d_2(0,p) < r\}$ are open. In both examples, the quasi-distance from the origin is not continuous at points $p\in \R^{*} \times \{0\}$.
\end{remark}

We conclude this section with a characterization of homogeneous quasi-distances by means of their unit ball. Together with Proposition~\ref{prop:continuity-quasi-dist}, we get the existence of continuous quasi-distances on graded groups. Finally, we also recall a characterization of homogeneous distances by means of their unit ball. 

\begin{example} [Characterization of homogeneous quasi-distances by means of their unit ball] 
\label{ex:qdist-unitball}
Let $G$ be a graded group with identity $e$ and with associated dilations $(\delta_\lambda)_{\lambda>0} $. Let $d$ be a homogeneous quasi-distance on $G$. By Proposition~\ref{prop:homogeneous-quasi-distance-topology}, $e$ belongs to the interior of $B_d(e,1)$ and $B_d(e,1)$ is relatively compact. By left-invariance, $B_d(e,1)$ is symmetric, i.e., $p\in B_d(e,1)$ implies $p^{-1} \in B_d(e,1)$. Finally, it follows from the homogeneity of $d$ that, for all $p\in G$, the set $\{\lambda>0 ;\; \delta_{1/\lambda}(p) \in B_d(e,1)\}$ is a closed subinterval of $(0,+\infty)$ (for the relative topology on $(0,+\infty)$).

Conversely, assume that $K$ is a subset of $G$ that contains $e$ in its interior, $K$ is relatively compact, symmetric, and such that the set $\{\lambda>0 ;\; \delta_{1/\lambda}(p) \in K\}$ is a closed subinterval of $(0,+\infty)$ for all $p\in G$. Then $$d(p,q) := \inf\left\{\lambda >0;\; \delta_{1/\lambda}(p^{-1}\cdot q) \in K\right\}$$ defines a homogeneous quasi-distance on $G$. It is the homogeneous quasi-distance whose unit ball centered at $e$ is the set $K$.

For the sake of completeness, we give below a detailed proof of this claim. Although the general scheme of the proof is a classical one, we stress that some of the arguments use the topological properties proved in Proposition~\ref{prop:homogeneous-quasi-distance-topology}. For $p\in G$, we set
$$I(p):=\left\{\lambda >0; \; \delta_{1/\lambda}(p) \in K\right\}$$
and
 $$\rho(p) :=  \inf\left\{\lambda >0; \; \delta_{1/\lambda}(p) \in K\right\}~.$$
Note that since $K$ contains $e$ in its interior and $\delta_{1/\lambda}(p)$ converges to $e$ when $\lambda$ goes to $+\infty$, we have $I(p)\not = \emptyset$ and $\rho(p) <+\infty$ for all $p\in G$.
Obviously, one has $\rho(e)=0$. Conversely let $p\in G$ with $p\not = e$. Then $\norm{\delta_{1/\lambda}(p)}$ goes to $+\infty$ when $\lambda$ goes to 0 (remember~\eqref{e:norm} for the definition of $\norm{\cdot }$). Since $K$ is relatively compact, and hence bounded with respect to $\norm{\cdot}$, it follows that $\delta_{1/\lambda}(p) \not\in K$ for all $\lambda >0$ small enough. Hence $\rho(p) \not= 0$ and consequently, we get that 
\begin{equation} \label{e:qd1}
\rho(p) = 0 \quad \text{if and only if} \quad p=e~.
\end{equation}
Next, since $K$ is symmetric, we have $I(p) = I(p^{-1})$. Hence
\begin{equation} \label{e:qd2}
\rho(p) = \rho(p^{-1})~.
\end{equation}
Third, since $\delta_{\lambda}\circ\delta_{\eta}=\delta_{\lambda\eta}$ for all $\lambda, \eta >0$, one has $I(\delta_{\lambda}(p)) = \lambda I(p)$ for all $p\in G$ and all $\lambda >0$. Hence
\begin{equation} \label{e:qd3}
\rho(\delta_{\lambda}(p)) = \lambda \rho(p)~.
\end{equation}
Finally, the fact $\rho$ satisfies the quasi-triangle inequality, i.e., there exists some constant $C>0$ such that 
\begin{equation} \label{e:qd4}
\rho(p\cdot q ) \leq C \, (\rho(p) + \rho(q))
\end{equation}
for all $p$, $q\in G$, follows from the fact that $\rho$ is bi-Lipschitz equivalent to any homogeneous quasi-norm on $G$. Indeed let $d_0$ be a homogeneous quasi-distance on $G$ (remember that homogeneous quasi-distances on graded groups do exist, see Example~\ref{ex:quasi-dist-graded}) and set $\rho_0(p):=d_0(e,p)$. Since the topology induced by $d_0$ and the manifold topology coincide (see Proposition~\ref{prop:homogeneous-quasi-distance-topology}) and since $e$ belongs to the interior of $K$, there exists $\eta >0$, such that $p\in K$ as soon as $\rho_0(p) < \eta$. By homogeneity, it follows that $2\rho_0(p)/\eta \in I(p)$ for all $p\not= e$. Hence $\rho(p)\leq 2 \rho_0(p)/\eta$ for all $p\in G$. On the other hand, since $K$ is relatively compact, it is bounded with respect to $\rho_0$ by Proposition~\ref{prop:homogeneous-quasi-distance-topology}. Hence one can find $M>0$ such that $\rho_0(p) \leq M$ for all $p\in K$. It follows that for $p\not = e$, $\delta_{2M/\rho_0(p)}(p) \not\in K$ and hence $\rho_0(p) / 2M \not \in I(p)$. By assumption, $I(p)$ is closed subinterval of $(0,+\infty)$. Moreover, since $K$ contains $e$ in its interior and  $\delta_{1/\lambda}(p)$ converges to $e$ when $\lambda$ goes to $+\infty$, $I(p)$ is an unbounded closed subinterval of $(0,+\infty)$, i.e.,
$$I(p) = [\rho(p),+\infty)$$
for all $p\in G$ with $p\not=e$. It follows that $\rho_0(p) / 2M  \leq \rho(p)$ for all $p\in G$.  All together we get that one can find a constant $L>0$ such that 
$$L^{-1} \rho_0(p) \leq \rho(p) \leq L\; \rho_0(p)$$
for all $p\in G$. Then the fact that $\rho_0$ satisfies the quasi-triangle inequality implies that $\rho$ satisfies the quasi-triangle inequality as well. This proves \eqref{e:qd4}.

All together \eqref{e:qd1}, \eqref{e:qd2}, \eqref{e:qd3} and \eqref{e:qd4} imply that $d$ is a homogeneous quasi-distance on $G$. Finally, one has $\rho(p)\leq 1$ if and only if $1\in I(p)$, i.e., $p\in K$, hence $K=B_d(e,1)$.
\end{example}

\begin{example} [Existence of continuous homogeneous quasi-distances on graded groups] \label{ex:qdist-hebsich_sikora}
The characterization given in Example~\ref{ex:qdist-unitball} together with Proposition~\ref{prop:continuity-quasi-dist} gives an effective way to construct continuous homogeneous quasi-distances on arbitrary graded groups. In particular, one can extend Hebisch and Sikora's construction of Example~\ref{ex:hebisch-sikora-distance}. Namely, following the notations of Example~\ref{ex:hebisch-sikora-distance}, we get that, for all $R>0$, $d_R$ induces a homogeneous quasi-distance on an arbitrary graded group. It is the homogeneous quasi-distance whose unit ball centered at the identity is a Euclidean ball of radius $R$, using exponential coordinates of the first kind relative to some basis of the Lie algebra adapted to the positive grading. Moreover, its unit sphere centered at the identity is a Euclidean sphere of radius $R$ hence is closed. It follows that $d_R$ is continuous.
\end{example}

\begin{example} [Characterization of homogeneous distances on homogeneous groups]
\label{ex:dist-unitball}
We recall here a characterization, already contained in a slightly different form in~\cite{Hebisch-Sikora}, of homogeneous distances on homogeneous groups in terms of their unit ball. Let $G$ be a homogeneous group with associated dilations $(\delta_\lambda)_{\lambda>0} $ and with identity element $e$. If $d$ is a homogeneous distance on $G$, then $e$ belongs to the interior of $B_d(e,1)$, $B_d(e,1)$ is  compact and symmetric. Since $d$ satisfies the quasi-triangle inequality with a multiplicative constant $C=1$, we have $\delta_\lambda (p) \cdot \delta_{1-\lambda} (q) \in B_d(e,1)$ for all $p$, $q \in B_d(e,1)$ and all $\lambda \in [0,1]$.

Conversely, assume that $K$ is a subset of $G$ that contains $e$ in its interior, $K$ is compact, symmetric and such that $\delta_\lambda (p) \cdot \delta_{1-\lambda} (q) \in K$ for all $p$, $q \in K$ and all $\lambda \in [0,1]$. Then $$d(p,q) := \inf\left\{\lambda >0;\; \delta_{1/\lambda}(p^{-1}\cdot q) \in K\right\}$$ defines a homogeneous distance on $G$. It is the homogeneous distance whose unit ball centered at $e$ is the set $K$. We refer to~\cite{Hebisch-Sikora} for the proof of the fact that $d$ satisfies the quasi-triangle inequality with a multiplicative constant $C=1$ and to Example~\ref{ex:qdist-unitball} for all other properties that must be satisfied by a homogeneous distance.
\end{example}

\section{Besicovitch Covering Property}  \label{sect:bcp-vs-wbcp}

\subsection{BCP and WBCP}

Recall from the introduction the definition of the Besicovitch Covering Property in the general quasi-metric setting. See Section~\ref{subsec:homogenous-qdist} for the definition and our conventions about quasi-metric spaces. 

\begin{definition}[Besicovitch Covering Property] \label{def:bcp}
Let $(X,d)$ be a quasi-metric space. We say that $(X,d)$ satisfies the Besicovitch Covering Property (BCP) if there exists a constant $N\geq 1$ such that the following holds. Let $A$ be a bounded subset of $X$ and $\mathcal{B}$ be a family of balls such that each point of $A$ is the center of some ball of $\mathcal{B}$, then there is a finite or countable subfamily $\mathcal{F}\subset \mathcal{B}$ such that the balls in $\mathcal{F}$ cover $A$, and every point in $X$ belongs to at most $N$ balls in $\mathcal{F}$, that is, 
\begin{equation*}
\carset_A \leq \sum_{B \in \mathcal{F}} \carset_B \leq N,
\end{equation*}
where $\carset_A$ denotes the characteristic function of the set $A$.
\end{definition}

The Besicovitch Covering Property originates from works of Besicovitch in connection with the theory of differentiation of measures in Euclidean spaces (\cite{Besicovitch_1}, \cite{Besicovitch_2}, see also Section~\ref{sect:differentiation-measures}). Finite dimensional normed vector spaces satisfy BCP (see~\cite[Chapter 2.8]{Federer}) whereas infinite dimensional normed vector spaces do not satisfy BCP.

Definition~\ref{def:bcp} for BCP is a common and classical one, even though one can find various variants in the literature. In the Euclidean setting, these variants are equivalent. One of them, called in the present paper the Weak Besicovitch Covering Property (WBCP), see Definition~\ref{def:wbcp}, turns out to be equivalent to BCP in our setting of graded groups equipped with homogeneous quasi-distances, and more generally for doubling quasi-metric spaces, see Proposition~\ref{prop:BCP-WBCP-Doubling}. For our purposes, WBCP is actually technically more convenient to work with. 

For the sake of completeness, we discuss in more details in the rest of this section  the relationships between BCP and WBCP, first pointing out that BCP and WBCP may happen to be non equivalent for general quasi-metric spaces, see Example~\ref{ex:wbcp-but-nobcp}. This might be of independent interest, and, to our knowledge, cannot be found explicitly written  in the literature. We will next prove that for doubling quasi-metric spaces, and hence for graded groups equipped with homogeneous quasi-distances, BCP and WBCP are equivalent. We first introduce some convenient terminology.

\begin{definition}[Family of Besicovitch balls]\label{def:BesicovitchBalls}
Let $(X,d)$ be a quasi-metric space. We say that a family $\mathcal{B}:=\{B=B_d(x_B,r_B)\}$ of balls in $(X,d)$ is a {\em  family of Besicovitch balls} if $\mathcal{B}$ is a finite family of balls such that, for all $B$, $B'\in \mathcal{B}$ with $B\not=B'$, one has $x_B \not \in B'$, and for which $\bigcap_{B\in \mathcal{B}} B \not= \emptyset$.
\end{definition} 

\begin{definition}[Weak BCP] \label{def:wbcp}
Let $(X,d)$ be a quasi-metric space. We say that $(X,d)$ satisfies the Weak Besicovitch Covering Property (WBCP) if there exists a constant $Q\geq 1$ such that $\card \mathcal{B} \leq Q$ for every family $\mathcal{B}$ of Besicovitch balls in $(X,d)$.
\end{definition}

If $(X,d)$ satisfies BCP, then $(X,d)$ satisfies WBCP. One can indeed take $Q=N$ where $N$ is given by Definition \ref{def:bcp}. Conversely, as already mentioned, WBCP is in general strictly weaker than BCP, as the following example shows.

\begin{example} \label{ex:wbcp-but-nobcp} Here is an example of a metric space that does not satisfy BCP and for which $\card \mathcal{B} =1$ for every family $\mathcal{B}$ of Besicovitch balls. Let $X=\{x_1,x_2,\dots\}$ be a countable set of points. Let us define $d:X\times X \rightarrow [0,+\infty)$ as follows. We set $d(x_i,x_i)=0$ for all $i\geq 1$ and $$d(x_i,x_j) = 1 - \frac{1}{\max(i,j)} \quad \text{for } i\not=j\,.$$

We first check that $d$ defines a distance on $X$. The fact that $d(x,y)=0$ if and only if $x=y$ and $d(x,y)=d(y,x)$ are obvious from the definition. To prove the triangle inequality, let $j<i$ and $k\geq 1$ be fixed. If $k<i$, we have
$$d(x_i,x_j)  = 1-\frac{1}{i} \leq 1-\frac{1}{i} + 1-\frac{1}{\max(j,k)} = d(x_i,x_k) + d(x_k,x_j) \,.$$

\noindent If $i<k$, then $k\geq 3$, hence $i/(i+1) \leq 1 \leq k/2$ and so $1-1/i \leq 2(1-1/k)$. It follows that $$d(x_i,x_j) = 1-\frac{1}{i} \leq 2\left(1-\frac{1}{k}\right) = d(x_i,x_k) + d(x_k,x_j) \,.$$

\noindent Hence $d$ satisfies the triangle inequality.

We claim that BCP does not hold in $(X,d)$. Indeed, set $$r_i:=1-\frac{1}{i} \quad \text{for } i=1,2,\dots$$
and let us consider $A:=\{x_i;\; i\geq 2\}$ and the family $\mathcal{B}:=\{B_d(x_i,r_i);\, i \geq 2\}$. Since $d(x_j,x_i) = r_i$ for all $j<i$ and $d(x_j,x_i)=r_j>r_i$ for $j>i$, we have
$$B_d(x_i,r_i) = \{x_1,\dots,x_i\}\quad \text{for } i=2,3,\dots .$$
It follows that for any subfamily $\mathcal{F}\subset \mathcal{B}$ whose balls cover the set $A$, we have 
$$\sup\{i\geq 2;\; B_d(x_i,r_i) \in  \mathcal{F}\} = +\infty~,$$
that is, $\card{\mathcal{F}} = +\infty$. On the other hand, $x_1 \in \cap_{B\in \mathcal{F}} B$. In particular $x_1$ belongs to infinitely many balls in $\mathcal{F}$ which shows that $(X,d)$ does not satisfy BCP. 

Let us now check that $\card \mathcal{B} =1$ for every family $\mathcal{B}$ of Besicovitch balls and hence $(X,d)$ satisfies WBCP. By contradiction, assume that $\{B_d(x_{i_l},\rho_{i_l})\}_{l = 1}^k$ is a family of Besicovitch balls with $k\geq 2$. Assume with no loss of generality that $i_1<i_2<\cdots<i_k$. We have $d(x_{i_j},x_{i_k}) =r_{i_k}$ and $x_{i_j}  \not\in B_d(x_{i_k},\rho_{i_k})$ for all $j=1,\dots,k-1$, so $r_{i_k} >\rho_{i_k} $. It follows that $d(x_{i_k},x_l) = \max(r_{i_k},r_l) \geq r_{i_k}>\rho_{i_k}$  for all $l\not = i_k$ and thus $$B_d(x_{i_k},\rho_{i_k}) = \{x_{i_k}\} \subset X \setminus \bigcup_{j=1}^{k-1} B(x_{i_j},\rho_{i_j})$$
which contradicts the fact that $\cap_{l = 1}^k B(x_{i_l},\rho_{i_l}) \not= \emptyset$.
\end{example}

\begin{remark}\label{rmk:WBCP:cover} If $(X,d)$ satisfies WBCP then it satisfies a weak form of BCP that can be stated as follows. There is a constant $N\geq 1$ such that the following holds. Let $A$ be a bounded subset of $X$. Let $\mathcal{B}$ be a family of balls such that each point of $A$ is the center of some ball of $\mathcal{B}$ and such that either $\sup\{r_B;\; B\in \mathcal{B}\}=+\infty$ or $B\in \mathcal{B} \mapsto r_B$ attains only an isolated set of values in $(0,+\infty)$. Then there is a finite or countable subfamily $\mathcal{F}\subset \mathcal{B}$ such that the balls in $\mathcal{F}$ cover $A$, and every point in $X$ belongs to at most $N$ balls in $\mathcal{F}$ (see~\cite{Preiss83}). Note that in Example~\ref{ex:wbcp-but-nobcp}, the number 1 is an accumulation point of the set $\{r_i;\; i \geq 2\}$ in $(0,+\infty)$.
\end{remark}

As already mentioned, for doubling quasi-metric spaces, BCP and WBCP are equivalent. Let us recall the definition of doubling quasi-metric spaces.

\begin{definition}[Doubling quasi-metric space]
A quasi-metric space $(X,d)$ is said to be doubling if there is a constant $C\geq 1$ such that for each $r>0$, each ball in $(X,d)$ with radius $2r$ can be covered by a family of at most $C$ balls of radius $r$.
\end{definition}

\begin{proposition} \label{prop:BCP-WBCP-Doubling}
Let $(X,d)$ be a doubling quasi-metric space. Then $(X,d)$ satisfies BCP if and only if $(X,d)$ satisfies WBCP.
\end{proposition}

As a classical fact, graded groups equipped with homogeneous quasi-distances are doubling. The next corollary hence follows.

\begin{corollary} \label{cor:bcp-wbcp-homogeneous-qdist}
Let $G$ be a graded group and let $d$ be a homogeneous quasi-distance on $G$. Then $(G,d)$ satisfies BCP if and only if $(G,d)$ satisfies WBCP.
\end{corollary}

For the sake of completeness, we give below a proof of Proposition~\ref{prop:BCP-WBCP-Doubling}. This proof follows closely the arguments of the proof of Theorem 2.7 in \cite{mattila} about the validity of BCP in Euclidean spaces, using the following well-known property of doubling quasi-metric spaces (see e.g.~\cite{LuukkainenSaksman} for more details about doubling (quasi-)metric spaces).

\begin{remark} \label{rmk:doubling}
If $(X,d)$ be a doubling quasi-metric space, then there are constants $c\geq 1$ and $s\geq 0$ such that if $x\in X$, $r>0$ and $\lambda \geq 1$, the cardinality of every set in $B_d(x,\lambda r)$ whose points are at least $r$ apart is at most $c\lambda^s$.
\end{remark}

\begin{proof}[Proof of Proposition \ref{prop:BCP-WBCP-Doubling}]
Since any quasi-metric space satisfying BCP also satisfies WBCP, we only need to prove that $(X,d)$ satisfies BCP when $(X,d)$ is a doubling quasi-metric space satisfying WBCP. 

Let $A$ be a bounded subset of $X$ and $\mathcal{B}$ be a family of balls such that each point of $A$ is the center of some ball of $\mathcal{B}$. For each $x\in A$ choose one ball $B_d(x,r(x))$ in $\mathcal{B}$. As $A$ is bounded, we claim that we may assume that
$$M_1:=\sup_{x\in A} r(x) < +\infty~.$$
Indeed, otherwise pick some point $x$ in $A$ with $r(x) \geq \diam A$. Then $\mathcal{F}:=\{B_d(x,r(x))\}$ is obviously a subfamily of $\mathcal{B}$ which shows that BCP holds in $(X,d)$.

Choose $x_1\in A$ with $r(x_1) \geq M_1/2$
and then inductively 
$$x_{j+1} \in A \setminus \bigcup_{i=1}^j B_d(x_i,r(x_i)) \quad \text{with } r(x_{j+1}) \geq M_1/2$$
as long as possible. Since $A$ is bounded and points $x_i$'s are at least $M_1/2$ apart, it follows from Remark \ref{rmk:doubling} that the process terminates and we get a finite sequence $x_1,\dots,x_{k_1}$.

Let $$M_2:=\sup \{r(x);\; x\in A \setminus \bigcup_{i=1}^{k_1} B_d(x_i,r(x_i))\}~,$$
and choose $$x_{k_1+1} \in A \setminus \bigcup_{i=1}^{k_1} B_d(x_i,r(x_i)) \quad \text{with } r(x_{k_1+1}) \geq M_2/2$$
and again inductively 
$$x_{j+1} \in A \setminus \bigcup_{i=1}^j B_d(x_i,r(x_i)) \quad \text{with } r(x_{j+1}) \geq M_2/2$$
as long as possible. 

Continuing this process we get a finite or infinite increasing sequence of integers $0=k_0 < k_1 < k_2 < \cdots$, a decreasing sequence of positive numbers $M_i$ with $2M_{i+1} \leq M_i$, and a sequence of balls $B_i := B_d(x_i,r(x_i)) \in \mathcal{B}$ with the following properties. If $I_j:=\{k_{j-1}+1,\dots,k_j\}$ for $j=1,2,\dots$, then 
\begin{align}
&M_j/2 \leq r(x_i) \leq M_j &&\text{for } i\in I_j, \label{e:p1}\\
&x_{j+1} \in A \setminus \bigcup_{i=1}^j B_i \quad &&\text{for }j=1,2,\dots \label{e:p2},\\
&x_i \in A \setminus \bigcup_{m\not= k} \bigcup_{j\in I_m} B_j &&\text{for } i\in I_k. \label{e:p3}
\end{align}

The first two properties \eqref{e:p1} and \eqref{e:p2} follow from the construction. To prove \eqref{e:p3}, let $m\not = k$, $i\in I_k$ and $j\in I_m$. If $m<k$, $x_i\not \in B_j$ by \eqref{e:p2}. If $k<m$, then, by construction $r(x_j)<r(x_i)$, and $x_j \not \in B_i$ by \eqref{e:p2}, and so $x_i\not\in B_j$.

Let us now check that this subfamily of balls satisfies the conditions for BCP to hold. If the sequence $k_0,k_1,\dots$ is finite, it follows immediately from the construction that the balls $B_i$'s cover $A$. If the sequence is infinite, then $M_j$ converges to $0$, \eqref{e:p1} implies that $r(x_i)$ converges to $0$, and it follows as well from the construction that $$A\subset \bigcup_{i=1}^{+\infty} B_i~.$$

To verify the other property for the validity of BCP, assume that a point $x\in X$ belongs to $p$ balls $B_{m_1},\dots,B_{m_p}$. Since WBCP holds in $(X,d)$ we have by \eqref{e:p3} that the indices $m_i$ can belong to at most $Q$ different blocks $I_j$, where $Q$ is given by Definition \ref{def:wbcp}, that is,
$$\card \{j;\; I_j \cap \{m_1,\dots,m_p\} \not= \emptyset\} \leq Q~.$$
To conclude, let us check that 
$$\card (I_j \cap \{m_1,\dots,m_p\}) \leq M \quad \text{for } j=1,2,\dots$$
for some constant $M$ depending only on the doubling constant of $(X,d)$. Let $j$ be fixed. The points $x_l$, $l\in I_j \cap \{m_1,\dots,m_p\}$, are at least $M_j/2$ apart by \eqref{e:p1} and \eqref{e:p2} and are all contained in  $B_d(x,M_j)$, and so the claim follows from Remark \ref{rmk:doubling}.
\end{proof}

\begin{remark} Note that the subfamily constructed in the previous proof satisfies the following additional property:  $B_d(x_i,r(x_i)/4) \cap B_d(x_j,r(x_j)/4) =\emptyset$ for all $i\not = j$. Indeed let $i<j$. Then $r(x_j) \leq 2 r(x_i)$ and $x_j \not \in B_i$ by \eqref{e:p2}, hence $d(x_j,x_i) > r(x_i) > r(x_i)/4 + r(x_j)/4$.
\end{remark}

\subsection{Preserving BCP}
\label{subsec:presersing-BCP}

Let us first recall that the validity of (W)BCP is not stable under a biLipschitz change of (quasi-)distance, see Theorem~\ref{thm:destroybcp}. More generally, the validity of (W)BCP might not be stable under natural operations on quasi-metric spaces. See for instance the example before Theorem~\ref{thm:WBCP-productspaces} about product of quasi-metric spaces. The fact that (W)BCP holds on a quasi-metric space $(X,d)$ depends indeed on the precise  shape of balls in $(X,d)$. We give in this section cases where the validity of (W)BCP is preserved, to be used later. This might be more generally of independent interest.

We begin with the following simple remark. If $d_1$ and $d_2$ are two quasi-distances on a space $X$ such that any ball with respect with $d_2$ is a ball with respect to $d_1$, with the same center but possibly with a different radius, then the validity of (W)BCP in $(X,d_1)$ implies the validity of (W)BCP in $(X,d_2)$. 

For instance if (W)BCP holds in $(X,d)$, then, for any $s >0$, $d^s$ defines a quasi-distance on $X$ and (W)BCP holds on $(X,d^s)$. Note that it is well-known that a metric space $(X,d)$ and its snowflakes $(X,d^s)$, $0<s<1$, have for many over purposes significantly different properties. For graded groups, we get the following proposition, to be used later.

\begin{proposition} \label{prop:BCP-kpower}
Let $G$ be a graded group and let $d$ be a homogeneous quasi-distance on $G$. Let $t>0$. If BCP holds on $(G,d)$ then BCP holds on the $t$-power of $G$ equipped with the homogeneous quasi-distance $d^{1/t}$.
\end{proposition}

See Example~\ref{ex:k-power} for the definition of the $t$-power of a graded group and Example~\ref{ex:homogeneous-qdist-kpower} for homogeneous quasi-distances on $t$-powers.

Another simple remark is the fact that a subset of a quasi-metric space that satisfies (W)BCP also satisfies (W)BCP when equipped with the restricted quasi-distance. We state it below for later reference.

\begin{proposition} \label{prop:BCP-subset}
Let $(X,d_X)$ be a quasi-metric space. Let $Y\subset X$. If BCP (resp. WBCP) holds on $(X,d_X)$ then BCP (resp. WBCP) holds on $(Y,d_Y)$ where $d_Y$ denotes the quasi-distance $d_X$ restricted to $Y$.
\end{proposition}

Given two quasi-metric spaces $(X,d_X)$ and $(Y,d_Y)$, there are many ways to define quasi-distances on $X\times Y$. If $(X,d_X)$ and $(Y,d_Y)$ both satisfy WBCP, then WBCP may fail for classical choices of quasi-distances on $X\times Y$, as shows the following example. We know that, for $s \geq 1$, $\R$ equipped with the snowflake distance $d_s(x,x'):=|x'-x|^{1/s}$ satisfies WBCP. Let $s>1$, $r \geq 1$, and let $d_{s,r}$ be the distance on $\R \times \R$ given by $d_{s,r}((x,y),(x',y')):= ( d_1(x,x')^r + d_s(y,y')^r)^{1/r}$. Then, if $r\in [1,s)$, WBCP does not hold on $(\R \times \R, d_{s,r})$. Indeed, $d_{s,r}$ is a left-invariant distance on $\R \times \R$ (equipped with the Abelian group law) and is one-homogeneous with respect to the dilations $\delta_\lambda(x,y):= (\lambda x, \lambda^s y)$. Its unit ball centered at the origin is given by $B_{d_{s,r}}(0,1) = \{ (x,y) \in \R \times \R;\; |x'-x|^r + |y'-y|^{r/s} \leq 1\}$. It follows from~\cite[Lemma~3.2]{LeDonne_Rigot_rmknobcp} that if WBCP holds in $(\R \times \R, d_{s,r})$, one would have $r\geq s$.

However, the following theorem, to be used later, shows that one can always find a quasi-distance satisfying WBCP on a product of quasi-metric spaces that satisfy WBCP.

\begin{theorem} \label{thm:WBCP-productspaces} Let $(X,d_X)$ and $(Y,d_Y)$ be two metric spaces. Assume that WBCP holds in  $(X,d_X)$ and in $(Y,d_Y)$. Then $X\times Y$ equipped with the max distance $$d_{X\times Y}((x,y),(x',y')):=  \max(d_X(x,x'),d_Y(y,y'))$$ satisfies WBCP.
\end{theorem}

\begin{proof}
Let $Q\in \N$ be such that $\card \mathcal{F} \leq Q$ for any family $\mathcal{F}$ of Besicovitch balls in $(X,d_X)$ or in $(Y,d_Y)$. Let $\mathcal{B}:=\{B_{d_{X\times Y}} (p_i,r_i)\}_{i=1}^N$ be a family of Besicovitch balls in $(X\times Y,d_{X\times Y})$. Let $i, j \in \{1,\cdots,N\}$, $i\not =j$ and let $p_i:=(x_i,y_i)$ and $p_j:=(x_j,y_j)$. By definition of families of Besicovitch balls and by definition of $d_{X\times Y}$, we have $d_{X\times Y}(p_i,p_j) = \max (d_X(x_i,x_j), d_Y(y_i,y_j)) > \max (r_i,r_j)$ hence $d_X(x_i,x_j) > \max (r_i,r_j)$  or $d_Y(y_i,y_j)> \max (r_i,r_j)$. In other terms, for any pair of indices $(i,j)$ with $i\not= j$, we have
\begin{equation} \label{e:besicovitch-in-X}
x_i \not \in B_{d_X}(x_j,r_j) \quad \text{and} \quad x_j \not \in B_{d_X}(x_i,r_i)
\end{equation}
or
\begin{equation} \label{e:besicovitch-in-Y}
y_i \not \in B_{d_Y}(y_j,r_j) \quad \text{and} \quad y_j \not \in B_{d_Y}(y_i,r_i)~.
\end{equation}

Let us consider the graph $\Gamma$ with $N$ vertices $i=1,\cdots,N$ and where $i$ is connected to $j$ if and only if $i\not = j$ and \eqref{e:besicovitch-in-X} holds. Then, for any complete subgraph $\gamma$ of $\Gamma$ (a complete graph is a graph where any two vertices are connected), $\{B_{d_X} (x_i,r_i)\}_{i\in \gamma}$ is a family of Besicovitch balls in $(X,d_X)$.
Let $\Gamma'$ be the complementary graph of $\Gamma$, that is, the graph with the same vertices as $\Gamma$ and where two vertices are connected in $\Gamma'$ if and only if they are not connected in $\Gamma$. Since \eqref{e:besicovitch-in-Y} holds whenever \eqref{e:besicovitch-in-X} does not, $\{B_{d_Y} (y_i,r_i)\}_{i\in \gamma'}$ is a family of Besicovitch balls in $(Y,d_Y)$ for any complete subgraph $\gamma'$ of $\Gamma'$.

As a special case of Ramsey's theorem  stated in the language of graph theory, there exists a function $f(k,l)$ such that for any given graph $\Gamma$ with $N\geq f(k,l)$ vertices, then either $\Gamma$ contains a complete subgraph of order $k$ or its complementary graph $\Gamma'$ contains a complete subgraph of order $l$ (the order of a complete graph is the number of its vertices). An upper bounded for $f(k,k)$ for $k\geq 3$ has been proved by P.~Erdos and G.~Szekeres (\cite{erdos-szekeres}), namely $f(k,k) < 4^{k-1}$.

Going back to the family of Besicovitch balls $\mathcal{B}$, it follows that if the numbers $N$ of balls in $\mathcal{B}$ is larger than $4^{Q}$ (we may assume with no loss of generality that $Q\geq 2$), there would exist either a family of Besicovitch balls in $(X,d_X)$ with cardinality $Q+1$ or a family of Besicovitch balls in $(Y,d_Y)$ with cardinality $Q+1$. This contradicts the fact that by assumption any family of Besicovitch balls in $(X,d_X)$ or in $(Y,d_Y)$ has cardinality at most $Q$. Hence $(X\times Y,d_{X\times Y})$ satisfies WBCP.
\end{proof}

Submetries, also known as metric submersions, will play a important role in our arguments. They are indeed well adapted tools for our purposes. We first recall the definition. 

\begin{definition}[Submetry]
Let $(X,d_X)$ and $(Y,d_Y)$ be quasi-metric spaces. We say that $\pi:X\rightarrow Y$ is a submetry if $\pi$ is a surjective map such that $\pi(B_{d_X}(p,r)) = B_{d_Y}(\pi(p),r)$ for all $p\in X$ and all $r>0$.
\end{definition}

We recall the following property of submetries related to WBCP.

\begin{proposition}[{\cite[Proposition~2.7]{LeDonne_Rigot_rmknobcp}}] \label{prop:bcp-submetry}
Let $(X,d_X)$ and $(Y,d_Y)$ be quasi-metric spaces. Assume that there exists a submetry from $(X,d_X)$ onto $(Y,d_Y)$. If $(X,d_X)$ satisfies WBCP then $(Y,d_Y)$ satisfies WBCP.
\end{proposition}

Proposition~\ref{prop:bcp-submetry} will be used in the proof of our main results together with Proposition~\ref{prop:submetries-morphism-graded-algebra} below.

\begin{proposition} \label{prop:submetries-morphism-graded-algebra}
Let $\hat{G}$ and $G$ be graded groups with graded Lie algebra $\hat{\g}$ and $\g$ respectively. Assume that there exists a surjective morphism of graded Lie algebras $\phi:\hat{\g}\rightarrow\g$. Let $\varphi:\hat{G}\rightarrow G$ denote the unique Lie group homomorphism such that $\varphi_* = \phi$ and let $\hat{d}$ be a 
homogeneous distance, respectively a continuous homogeneous quasi-distance, on $\hat{G}$. Then
$$d(p,q):= \hat{d}(\varphi^{-1}(\{p\}),\varphi^{-1}(\{q\}))$$
defines a homogeneous distance, respectively a continuous homogeneous quasi-distance, on $G$ and $\varphi:(\hat{G},\hat{d})\rightarrow (G,d)$ is a submetry. 
\end{proposition}

We stress that continuity of the quasi-distance $\hat{d}$, which means global continuity on $\hat{G}\times \hat{G}$, is necessary in order to get that $d$ is a quasi-distance on $G$, which turns out to be continuous on $G\times G$ as well, and also in order to get that $\varphi:(\hat{G},\hat{d})\rightarrow (G,d)$ is a submetry. Indeed, consider the homogeneous quasi-distance $d_2$ on $\mathbb{R}^2$ given in Remark~\ref{rmk:continuity-dist}, which is not globally continuos. The projection of $B_{d_2}(0,1)$ onto the $x$-axis is the open segment $I:=(-1,1)$. If $x\in \mathbb{R}^*$, we have $\{\lambda >0,\; x/\lambda \in I\} = (|x|, +\infty)$ which is not a closed subinterval of $(0,+\infty)$. It follows from Example~\ref{ex:qdist-unitball} that $I$ is not the unit ball of some homogeneous quasi-distance on $\mathbb{R}$. However, when $\hat{d}$ is a homogeneous distance, recall that $\hat{d}$ is continuous on $\hat{G}\times \hat{G}$ (see Corollary~\ref{cor:continuity-homogeneous-distance}). 

\begin{proof}[Proof of Proposition~\ref{prop:submetries-morphism-graded-algebra}]
First, we prove that $d$ defines a quasi-distance on $G$ and that $\varphi:(\hat{G},\hat{d})\rightarrow (G,d)$ is a submetry. By~\cite[Proposition~2.8]{LeDonne_Rigot_rmknobcp}, it is sufficient to prove that for all $p$, $q \in G$ and all $\hat{p}\in \varphi^{-1}(\{p\})$, one can find $\hat{q} \in \varphi^{-1}(\{q\})$ such that $d(p,q) = \hat{d}(\hat{p},\hat{q})$. Set $\hat{K}:=\ker \varphi$. We have $\varphi^{-1}(\{p\}) = \hat{K}  \cdot \hat{p}$ and $\hat{k}\cdot \varphi^{-1}(\{q\}) = \varphi^{-1}(\{q\})$ for all $\hat{k}\in \hat{K}$. By left-invariance of $\hat{d}$, it follows that
\begin{equation*}
\hat{d}(\hat{p},\varphi^{-1}(\{q\})) = \hat{d}(\hat{k}\cdot \hat{p} , \hat{k}\cdot \varphi^{-1}(\{q\}) = \hat{d}(\hat{k}\cdot \hat{p} ,  \varphi^{-1}(\{q\}))~.
\end{equation*} 
In other words, the function $\hat{p}'\in \varphi^{-1}(\{p\}) \mapsto \hat{d}(\hat{p}',\varphi^{-1}(\{q\}))$ is constant. Hence, by definition of $d$, we get
\begin{equation*}
d(p,q) = \hat{d}(\hat{p},\varphi^{-1}(\{q\})) = \inf_{\hat{q}\in \varphi^{-1}(\{q\})} \hat{d}(\hat{p},\hat{q})~.
\end{equation*}
Bounded sets with respect to $\hat{d}$ are relatively compact (see Proposition~\ref{prop:homogeneous-quasi-distance-topology}), the set $\varphi^{-1}(\{q\})$ is closed and $\hat{d}$ is assumed to be continuous, hence one can find $\hat{q}\in \varphi^{-1}(\{q\})$ such that $\hat{d}(\hat{p},\hat{q})= \inf_{\hat{q}'\in \varphi^{-1}(\{q\})} \hat{d}(\hat{p},\hat{q}') = d(p,q)$. This proves that $d$ defines a quasi-distance on $G$ and that $\varphi:(\hat{G},\hat{d})\rightarrow (G,d)$ is a submetry.

Note that if $\hat{d}$ satisfies the quasi-triangle inequality with multiplicative constant $C$, then $d$ satisfies the quasi-triangle inequality with the same multiplicative constant (see the proof of~\cite[Proposition 2.8]{LeDonne_Rigot_rmknobcp}). In particular if $\hat{d}$ is a distance on $\hat{G}$ then $d$ is a distance on $G$.

Next, one can easily check that $d$ is homogeneous. This follows from the fact that $ \hat{d}$ is a homogeneous quasi-distance together with the fact that $\phi$ is a surjective morphism of graded Lie algebra.

To conclude the proof, it remains to prove that $d$ is globally continuous on $G \times G$. Since $d$ is left-invariant, it is sufficient to prove that $d(e,\cdot)$ is globally continuous on $G$. Here $e$ denotes the identity in $G$ and below $\hat{e}$ will denote the identity in $\hat{G}$. Let $p \in G$ and let $(p_k)$ be a sequence converging  to $p$. First, let $\hat{p}_k \in \varphi^{-1}(\{p_k\})$ be such that $d(e,p_k) = \hat{d}(\hat{e},\hat{p}_k)$. Since the sequence $(p_k)$ is relatively compact, it is bounded with respect to $d$ (see Proposition~\ref{prop:homogeneous-quasi-distance-topology}). Hence $(\hat{p}_k)$ is bounded with respect to $\hat{d}$ and, once again by 
Proposition~\ref{prop:homogeneous-quasi-distance-topology}, relatively compact. Up to a subsequence, one can thus assume that $\hat{p}_k$ converges to some $\hat{p} \in \varphi^{-1}(\{p\})$. Since $\hat{d}$ is continuous, it follows that $\hat{d}(\hat{e},\hat{p}_k)$ converges to $\hat{d}(\hat{e},\hat{p})$ and 
one gets
$$d(e,p) \leq \hat{d}(\hat{e},\hat{p}) = \lim_{k\rightarrow + \infty} \hat{d}(\hat{e},\hat{p}_k) = \lim_{k\rightarrow + \infty} d(e,p_k)~.$$
Next, let $\hat{p}'\in \varphi^{-1}(\{p\})$ be such that $d(e,p) = \hat{d}(\hat{e},\hat{p}')$. Since $\varphi = \exp \circ \phi \circ \exp^{-1}$ is an open map (see~\cite[p.104]{Warner}), one can find a sequence $\hat{p}'_k \in \varphi^{-1}(\{p_k\})$ converging to $\hat{p}'$. Since $\hat{d}$ is assumed to be continuous, $\hat{d}(\hat{e},\hat{p}'_k)$ goes to $\hat{d}(\hat{e},\hat{p}')$. On the other hand, we have $d(e,p_k)\leq  \hat{d}(\hat{e},\hat{p}'_k)$. Hence $$\lim_{k\rightarrow + \infty} d(e,p_k) \leq \hat{d}(\hat{e},\hat{p}') = d(e,p)~.$$
All together we finally get that $\lim_{k\rightarrow + \infty} d(e,p_k) = d(e,p)$ which concludes the proof.\end{proof}

\section{Graded groups with commuting different layers} \label{sect:BCP-commutative-layers}

In this section we consider graded groups with commuting different layers, see Definition~\ref{def:groups-commuting-layers}, and we prove the following results.

\begin{theorem} \label{thm:sec-BCP-commuting-layers}
Let $G$ be a graded group with commuting different layers. There exist continuous homogeneous quasi-distances on $G$ for which BCP holds. 
\end{theorem} 

\begin{corollary} \label{cor:sec-bcp-homogeneous-groups-commuting-layers}
Let $G$ be a homogeneous group  with commuting different layers. There exist homogeneous distances on $G$ for which BCP holds.
\end{corollary}

The proof of Theorem~\ref{thm:sec-BCP-commuting-layers} and Corollary~\ref{cor:sec-bcp-homogeneous-groups-commuting-layers} is divided into three steps. First, we consider stratified free-nilpotent Lie groups of step 2. We prove that for such groups, some homogeneous (quasi-)distances that satisfy BCP are those whose unit ball centered at the origin coincides with a Euclidean ball centered at the origin in exponential coordinates of the first kind associated to a canonical choice of basis of the Lie algebra, see Theorem~\ref{thm:BCP-free-step2}. Next, we prove the existence of homogeneous distances satisfying BCP for stratified groups of step 2, see Theorem~\ref{thm:BCP-stratified-step2}. These homogeneous distances are induced by homogeneous distances satisfying BCP on stratified free-nilpotent Lie groups of step 2 via submetries. Finally, the general case follows from Theorem~\ref{thm:BCP-stratified-step2} together with the structure property given by Proposition~\ref{prop:structure-commuting-layers}, and Theorem~\ref{thm:WBCP-productspaces}, see Section~\ref{subsect:BCP-commuting-layers}.

\subsection{Free-nilpotent groups of step 2}
 \label{subsect:BCP-free-step2}

Let $r\geq 2$ be an integer. We denote by $\mathbb F_{r2}$ the stratified free-nilpotent Lie group of step 2 and rank $r$ whose Lie algebra $\mathfrak{f}_{r2}$ is endowed with a given stratification
$\mathfrak f_{r2} = V \oplus W$ where $[V,V] = W$ and where
\begin{equation*}
\dim V = r  \; ,\quad \dim W = \frac{r(r-1)}{2}~.
\end{equation*}

We set $n:= \dim \mathfrak{f}_{r2}$. We fix a basis $(X_1,\dots,X_r)$ of $V$ and we set $X_{ij}:=[X_i,X_j]$. Then $(X_{ij})_{1\leq i <j \leq r}$ is a basis of $W$. Using exponential coordinates of the first kind associated to the basis $(X_1,\dots,X_r,(X_{ij})_{1\leq i <j \leq r})$ adapted to the given stratification of $\mathfrak{f}_{r2}$, we write $p\in \mathbb F_{r2}$ as $p = \exp (\sum_{i=1}^r p_i X_i + \sum_{1\leq i <j \leq r} p_{ij} X_{ij})$ and we identify $p$ with $(p_1,\dots,p_r,(p_{ij})_{1\leq i <j \leq r})=[v_p, w_p]$ where $v_p := (p_1,\dots,p_r)$ and $w_p:=(p_{ij})_{1\leq i <j \leq r}$. 

The group law is given by $v_{p\cdot q} = v_p + w_q$ and $w_{p\cdot q} = ((p\cdot q)_{ij})_{1\leq i <j \leq r}$ where
\begin{equation} \label{e:grouplaw-Fr2}
 (p\cdot q)_{ij} = p_{ij}+q_{ij} +\dfrac{1}{2}( p_i q_j
- q_i p_j )
\end{equation}
for $1\leq i <j \leq r$. The identity element is the origin.

The associated dilations are given by 
\begin{equation*}
\delta_\lambda(p) = (\lambda v_p ,\lambda^2 w_p).
\end{equation*}

We denote by the same notations $\norm{\cdot}$ and $\langle \cdot,\cdot \rangle$ the Euclidean norm and scalar product in $\R^n$, $\R^r$ and $\R^{n-r}$ with respect to our choice of coordinates. Equivalently, we equip $\mathfrak{f}_{r2}$ with a Euclidean structure for which our chosen basis $((X_1,\dots,X_r),(X_{ij})_{1\leq i <j \leq r})$ is orthonormal and we consider on $V$ and $W$ the induced Euclidean structures.

For $R>0$, we consider the homogeneous quasi-distance $d$ on $\mathbb F_{r2}$ whose unit ball centered at the origin is given by 
\begin{equation}\label{eq:ball:R}
B_d(0,1) := \{ p\in \mathbb F_{r2}\,;\;
\norm{p}^2 \leq R^2\}~.
\end{equation}
Such a quasi-distance is well-defined and is continuous, see Example~\ref{ex:qdist-hebsich_sikora} (we drop here the index $R$ for simplicity of notations). Recall also that it follows from \cite{Hebisch-Sikora} (see also Example~\ref{ex:hebisch-sikora-distance}) that $d$ is a distance whenever $R<R^*$ for some $R^*>0$.

This section is devoted to the proof of the validity of BCP on $\mathbb F_{r2}$ equipped with such a quasi-distance.

\begin{theorem} \label{thm:BCP-free-step2}
Let $R>0$ be fixed. Let $d$ be the homogeneous quasi-distance on $\mathbb F_{r2}$ whose unit ball centered at the origin is given by~\eqref{eq:ball:R}. Then BCP  holds on $(\mathbb F_{r2},d)$.
\end{theorem} 

From now on in this section, we let $R>0$ be fixed and $d$ denotes the homogeneous quasi-distance on $\mathbb F_{r2}$ whose unit ball centered at the origin is given by~\eqref{eq:ball:R}. We set $B:=B_d(0,1)$. We begin with a series of remarks for later use. 

First, for $p\in \mathbb F_{r2}$, we define the function $A_p:\mathbb F_{r2}\rightarrow \R$ by
\begin{equation}\label{def:A}
A_p(q) := \norm{q}^2 - 2 \langle p,q\rangle
+ \sum_{1\leq i <j \leq r} (  p_{ij} - q_{ij})(p_i q_j - q_i p_j  ) +\frac{1}{4}(p_i q_j - q_i p_j )^2.
\end{equation}

\begin{lemma} 
\label{lemma:Aq}
Let $p \in \partial B$. Then $q\in B_d(p,1)$ if and only if $A_p(q)\leq 0$.
\end{lemma}

\begin{proof}
Let $p \in \partial B$. We have $\norm{p}^2 = \norm{v_p}^2+\norm{w_p}^2 = R^2$ and $q\in B_d(p,1)$ if and only if $\norm{p^{-1}\cdot q}^2 = \norm{v_{p^{-1}\cdot q }}^2 +\norm{w_{p^{-1}\cdot q }}^2 \leq R^2$.
Then the lemma follows from the specific form of the group law given in~\eqref{e:grouplaw-Fr2}.
\end{proof}

Next, we denote by $\angle ( \cdot,\cdot)$ the (non-oriented) angle $\in [0,\pi]$ between two vectors in a Euclidean space. In the proof of Theorem~\ref{thm:sec-BCP-commuting-layers}, we are going to look at points $p$, $q$ for which the angles  $\angle ( v_p,v_q)  $ and $\angle ( w_p,w_q)$ are small. These two angles are dilation invariant. Namely, we have
\begin{equation} \label{e:angle-dilation-invariant}
\angle ( v_{\delta_\lambda(p)},v_{\delta_\lambda(q)}) = \angle ( v_p,v_q)  \quad \text{and} \quad 
\angle ( w_{\delta_\lambda(p)},w_{\delta_\lambda(q)}) = \angle ( w_p,w_q)
\end{equation}
for all $p$, $q \in \mathbb F_{r2}$ and all $\lambda>0$.
Note that, on the contrary, the angle $\angle (p,q)$ is  not dilation invariant. 

\begin{lemma} \label{lem:small-angles}
For all $\eps>0$, there exists $\delta>0$ such that, for all $p$, $q\in \mathbb F_{r2}$, if $\angle ( v_p,v_q)<\delta $
and $\angle ( w_p,w_q)<\delta $, then 
\begin{equation}\label{eq:area}
|p_i q_j -   q_i p_j|    \leq \eps \norm{v_p} \norm{v_q}
\end{equation}
for all $1\leq i <j \leq r$,
\begin{equation}\label{eq:scalar1}
\langle v_p , v_q\rangle   \geq (1-\eps)  \norm{v_p} \norm{v_q} 
\end{equation}
and 
\begin{equation}\label{eq:scalar2}
\langle w_p , w_q\rangle   \geq (1-\eps)  \norm{w_p} \norm{w_q}~.
\end{equation}
\end{lemma}

\begin{proof}
To prove~\eqref{eq:area}, note that $|p_i q_j -   q_i p_j|$ 
represents the area of the planar quadrilateral generated by the 2-dimensional vectors 
$(p_i ,p_j)$ and $(q_i ,q_j)$.
For fixed length of these vectors, this area goes to zero when the angle between the two vectors goes to zero.
This angle is smaller than $\angle (v_p , v_q)$, and the length of the two vectors is bounded by 
$ \norm{v_p} $ and $ \norm{v_q}$, respectively. This implies~\eqref{eq:area}.  
To prove~\eqref{eq:scalar1} and \eqref{eq:scalar2}, note that if $a$ and $b$ are two vectors in a Euclidean space such that $\angle (a,b) < \pi / 2 $, then $\langle a,b\rangle $ is positive and 
represents the product between $\norm{a}$ and the norm of the orthogonal projection of $b$ onto $a$. In particular, if  
$\angle (a,b)$ goes to 0,
this product converges 
to  $\norm{a}\norm{b}$. 
\end{proof}

\begin{lemma} Let $\eps>0$ and $p$, $q \in B$ be such that \eqref{eq:area},  \eqref{eq:scalar1} and \eqref{eq:scalar2} hold. Let $A_p(q)$ be defined by~\eqref{def:A}. Then we have
\begin{align}
A_p(q) &\leq \norm{v_q}^2+  \norm{w_q}^2 - 2 (1-\eps)  \norm{v_p} \norm{v_q} - 2 (1-\eps)  \norm{w_p} \norm{w_q}  \label{e:bound1-Ap(q)}\\
& \phantom{\leq \norm{v_q}^2+  \norm{w_q}^2 - 2 (1-\eps)  \norm{v_p} \norm{v_q} } + 2 r^2 R \eps  \norm{v_p}  \norm{v_q} +\frac{r^2}{4}\eps^2 
R^2  \norm{v_p}  \norm{v_q} \notag \\
&\leq  \norm{v_q}^2+ \norm{w_q}^2 - 2 (1-\eps)  \norm{v_p} \norm{v_q} - 2 (1-\eps)  \norm{w_p} \norm{w_q}  \label{e:bound2-Ap(q)}\\
&\phantom{\leq \norm{v_q}^2+  \norm{w_q}^2 - 2 (1-\eps)  \norm{v_p} \norm{v_q}- 2 (1- } +2 r^2 R^2 \eps  \norm{v_q} +\frac{r^2}{4}\eps^2 
R^3 \norm{v_q}~. \notag
\end{align}
\end{lemma}

\begin{proof} By definition of $A_p(q)$, we have
\begin{multline*}
A_p(q)  = \norm{v_q}^2+ \norm{w_q}^2 - 2 \langle v_p,v_q\rangle - 2 \langle w_p,w_q\rangle \\
 + \sum_{1\leq i <j \leq r} (  p_{ij} - q_{ij} )(p_i q_j -   q_i p_j) +\frac{1}{4}(p_i q_j -   q_i p_j)^2~.
\end{multline*}
Then~\eqref{e:bound1-Ap(q)} and~\eqref{e:bound2-Ap(q)} follow from  \eqref{eq:area}, \eqref{eq:scalar1}
 and \eqref{eq:scalar2} together with 
the following simple observations. First, $\dim W \leq r^2$, which is used to bound the number of terms in the sum. Second, since $p$, $q \in B$, we have $\norm{v_p} ,  \norm{w_p}, \norm{v_q} ,  \norm{w_q} \leq R$ , which is used to bound some of the terms.
\end{proof} 

Next, we will consider in the proof of Theorem~\ref{thm:BCP-free-step2} the following parabolic regions. For $a>0$, we set 
\begin{equation*}
\mathcal P_a:= \{p\in\mathbb F_{r2} \; ;\;  R \norm{w_p} > a \norm{v_p}^2 \}.
\end{equation*}

These regions, as well as their complement, are invariant under dilations. Namely,
\begin{equation} \label{e:parabolic-dilation-invariant}
\delta_\lambda (\mathcal P_a) = \mathcal P_a \quad \text{and} \quad \delta_\lambda (\mathcal P_a^c) = \mathcal P_a^c
\end{equation}
for all $a>0$ and all $\lambda >0$.

\begin{lemma} We have the following bounds.

(i) Let $p\in B$. If $p \notin \mathcal P_a$,
then 
\begin{equation}\label{eq:bound21}
\norm{w_p} \leq  \frac{R}{2a}\left(\sqrt{1+4a^2} -1 \right)  .
\end{equation}

(ii) Let $p\in \partial B$. If $p\in \mathcal P_a$, 
then 
\begin{equation}\label{eq:bound1}
\norm{w_p} \geq  \frac{R}{2a}\left(\sqrt{1+4a^2} -1\right )  .
\end{equation}

(iii) Let $p\in \partial B$. If $p\notin \mathcal P_a$, 
then  
\begin{equation}\label{eq:bound22}
\norm{v_p} \geq \frac{R}{a}  \sqrt{ \frac{      \sqrt{1+4a^2}  -1   }{2} } .
\end{equation} 
\end{lemma}

\begin{proof}[Proof of $(i)$]
From the assumptions, we have
$$  \norm{w_p}^2 +   \dfrac{R}{a}  \norm{w_p}  \leq   \norm{w_p}^2 +\norm{v_p}^2\leq R^2.$$
Hence $ a \norm{w_p}^2 +   R \norm{w_p} - a R^2\leq 0$ and $\norm{w_p}\geq 0$, which is equivalent to 
\begin{equation*}
0 \leq \norm{w_p} \leq \frac{R}{2a} \left(-1 +\sqrt{1+4a^2} \right)
\end{equation*}
and gives~\eqref{eq:bound21}. \end{proof}

\begin{proof}[Proof of $(ii)$]
From the assumptions, we have
$$  \norm{w_p}^2 +   \dfrac{R}{a}  \norm{w_p} >   \norm{w_p}^2 +\norm{v_p}^2= R^2.$$
Hence $ a \norm{w_p}^2 +  R  \norm{w_p} -a R^2> 0$ and $\norm{w_p}\geq 0$, which is equivalent to 
\begin{equation*}
\norm{w_p} > \frac{R}{2a} \left(-1 +\sqrt{1+4a^2} \right)
\end{equation*}
and implies~\eqref{eq:bound1}. \end{proof}

\begin{proof}[Proof of $(iii)$]
From the assumptions, we have
$$  \dfrac{a^2}{R^2}   \norm{v_p}^4+ \norm{v_p}^2    \geq  
 \norm{w_p}^2 +\norm{v_p}^2= R^2.$$
Hence  $  a^2 \norm{v_p}^4 + R^2\norm{v_p}^2 - R^4\geq 0$, which is equivalent to 
\begin{equation*}
\norm{v_p}^2 \geq   \frac{R^2}{ 2a^2} \left(-1 + \sqrt{1+4a^2} \right) 
\end{equation*}
and gives~\eqref{eq:bound22}. \end{proof}

To prove Theorem~\ref{thm:BCP-free-step2}, we are going to partition $\mathbb{F}_{r2}$ into three disjointed regions,
$$\mathbb{F}_{r2} = \mathcal P_{a'} \sqcup (\mathcal P_a \setminus \mathcal P_{a'}) \sqcup \mathcal P_a^c~,$$
for some suitable choice of $0<a<a'$. 

The next three lemmas show that, for a suitable choice of $0<a<a'$, if $p$ and $q$ are two points for which the angles  $\angle ( v_p,v_q)  $ and $\angle ( w_p,w_q)$ are small and both belong to one of these regions, then either $q\in B_d(p,d(0,p))$ or $p\in B_d(q,d(0,q))$. We first consider the case where $p, q \in \mathcal P_a^c$.

\begin{lemma} \label{lemma:away}
Let $a=0.9$. There exists 
$\delta> 0$ such that,  
if $p, q \in \mathcal P_a^c$ are such that $\angle ( v_p,v_q)<\delta $ and $\angle ( w_p,w_q)<\delta $, then $q\in B_d(p,d(0,p))$ or $p\in B_d(q,d(0,q))$.
\end{lemma}

\begin{proof}
The value $a$ has been chosen in such a way that 
\begin{equation*}
  1+ \dfrac{\sqrt{1+4 a^2}-1}{2}  - 2 \sqrt{\dfrac{\sqrt{1+4 a^2} -1}{2a^2}  }  <0.
\end{equation*}
Hence, we can fix some $\eps>0$ so that
$$1 + \frac{\sqrt{1+4a^2} -1}{2}
 - 2 (1-\eps)  \sqrt{ \frac{      \sqrt{1+4a^2}  -1   }{2a^2} }      + 2 r^2 R \eps +\frac{r^2}{4}\eps^2 
R^2 <0~.$$

 Let $p, q \in \mathcal P_a^c$. Assume that $\angle ( v_p,v_q)<\delta $ and $\angle ( w_p,w_q)<\delta $ where $\delta>0$ is given by Lemma~\ref{lem:small-angles}, i.e., is such that \eqref{eq:area}, \eqref{eq:scalar1}, and \eqref{eq:scalar2} hold for our choice of $\eps$. Using dilations together with~\eqref{e:angle-dilation-invariant} and~\eqref{e:parabolic-dilation-invariant}, and exchanging the role of $p$ and $q$ if necessary, one can assume with no loss of generality that $p \in \partial B$ and $q \in B$. 

Then let us prove that $q\in B_d(p,1)$. By Lemma~\ref{lemma:Aq}, this is equivalent to $A_p(q) \leq 0$. By~\eqref{e:bound2-Ap(q)}, we have
\begin{multline*}
A_p(q) \leq \norm{v_q}^2+ \norm{w_q}^2 - 2 (1-\eps)  \norm{v_p} \norm{v_q} - 2 (1-\eps)  \norm{w_p} \norm{w_q} \\
 \quad + 2 r^2 R^2 \eps  \norm{v_q} +\frac{r^2}{4}\eps^2 
R^3 \norm{v_q}~.
\end{multline*}
To bound the 1st term, we use that $\norm{v_q} \leq R$ since $q\in B$. To bound the 2nd one, we use that $q \notin \mathcal P_a$ both through~\eqref{eq:bound21} and the fact that $\norm{w_q}  <  \dfrac{a}{R} \norm{v_q}^2\leq  a\norm{v_q}$. In the 3rd term, we use that $p\notin ~\mathcal P_a$ through~\eqref{eq:bound22}. Since the 4th term is non positive, we get
\begin{equation*}
\begin{split}
A_p(q)  & \leq R\norm{v_q}  + \frac{R}{2a}(\sqrt{1+4a^2} -1 ) 
a
\norm{v_q} - 2 (1-\eps)\frac{R}{a}  \sqrt{ \frac{      \sqrt{1+4a^2}  -1   }{2} } \norm{v_q} \\
& \phantom{R\norm{v_q}  + \frac{R}{2a}(\sqrt{1+4a^2} -1 ) 
a
\norm{v_q} - 2 (1-\eps)\frac{R}{a}  \sqrt{ }\quad\quad} + 2r^2 R^2\eps   \norm{v_q} +\frac{r^2}{4}\eps^2 R^3  \norm{v_q} \\
& =  R \norm{v_q} 
    \left(1 + \frac{\sqrt{1+4a^2} -1}{2}
 - 2 (1-\eps)  \sqrt{ \frac{      \sqrt{1+4a^2}  -1   }{2a^2} }      + 2 r^2 R \eps +\frac{r^2}{4}\eps^2 
R^2    \right)  \\
&\leq 0 
\end{split}
\end{equation*}
by the choice of $\eps$.
\end{proof}

Next, we consider the case where $p, q \in \mathcal P_{a'}$.

\begin{lemma} \label{lemma:near2a}
Let $a'=1.9$. There exists $\delta> 0$ such that 
if $p, q \in \mathcal P_{a'}$ are such that $\angle ( v_p,v_q)<\delta $ and $\angle ( w_p,w_q)<\delta $, then $q\in B_d(p,d(0,p))$ or $p\in B_d(q,d(0,q))$.
\end{lemma}

\begin{proof}
The value $a'$ has been chosen in such a way that 
$$
 \dfrac{1}{a'} +1  
- \frac{ \sqrt{1+4a'^2} -1   }{a'}<0~.$$
Hence, we can fix some $\eps>0$ so that
\begin{equation*}
  \dfrac{1}{a'} +1  
-  (1-\eps)   \frac{ \sqrt{1+4a'^2} -1   }{a'}
 <0 
 \end{equation*}
and
\begin{equation*}
 -    2 (1-\eps) 
  +2 r^2 R \eps    +\frac{r^2}{4}\eps^2 
R^2 <0~.
\end{equation*}

 Let $p, q \in \mathcal P_{a'}$. Assume that $\angle ( v_p,v_q)<\delta $ and $\angle ( w_p,w_q)<\delta $ where $\delta>0$ is given by Lemma~\ref{lem:small-angles}, i.e., is such that \eqref{eq:area}, \eqref{eq:scalar1}, and \eqref{eq:scalar2} hold for our choice of $\eps$. Arguing as in the proof of Lemma~\ref{lemma:away}, one can assume with no loss of generality that $p \in \partial B$ and $q \in B$. 
 
 Let us prove that $q\in B_d(p,1)$. By Lemma~\ref{lemma:Aq}, this is equivalent to  $A_p(q)\leq 0$. By~\eqref{e:bound1-Ap(q)}, we have
\begin{multline*}
A_p(q) \leq \norm{v_q}^2+  \norm{w_q}^2 - 2 (1-\eps)  \norm{v_p} \norm{v_q} - 2 (1-\eps)  \norm{w_p} \norm{w_q}  \\
\quad +2 r^2 R \eps  \norm{v_p}  \norm{v_q} +\frac{r^2}{4}\eps^2 
R^2  \norm{v_p}  \norm{v_q}~.
\end{multline*}
To bound the 1st term, we use that $q\in \mathcal P_{a'} $, i.e., $  \norm{v_q}^2< \dfrac{R}{a'}   \norm{w_q} $. To bound the 2nd one, we use that $\norm{w_q} \leq R$ since $q\in B$. In the 4th term, we use that $p\in \mathcal P_{a'}$ through~\eqref{eq:bound1}. This gives
\begin{equation*}
\begin{split}
A_p(q) & \leq \dfrac{R}{a'}   \norm{w_q}  +  R \norm{w_q}  - 2 (1-\eps)  \norm{v_p} \norm{v_q} 
- 2 (1-\eps)   \frac{R}{2a'}\left(\sqrt{1+4a'^2} -1\right ) \norm{w_q} \\
&\phantom{\dfrac{R}{a'}   \norm{w_q}  +  R \norm{w_q}  - 2 (1-\eps)  \norm{v_p} \norm{v_q} 
- 2 \quad} +2 r^2 R \eps  \norm{v_p}  \norm{v_q} +\frac{r^2}{4}\eps^2 
R^2  \norm{v_p}  \norm{v_q} \\
&= R  \norm{w_q} \left( \dfrac{1}{a'} +1  
-  (1-\eps)   \frac{ \sqrt{1+4a'^2} -1   }{a'}
\right) \\
&\phantom{R  \norm{w_q}  \dfrac{1}{a'} +1  
-  (1-\eps)\quad \quad\quad \quad\quad \quad } + \norm{v_p} \norm{v_q} 
 \left(
 -    2 (1-\eps) 
  +2 r^2 R \eps    +\frac{r^2}{4}\eps^2 
R^2  \right) \\
&\leq 0
\end{split}
\end{equation*}
by the choice of $\eps$.
\end{proof}

Finally, we consider the case where $p$, $q\in \mathcal P_{a} \setminus \mathcal P_{a'}$.

\begin{lemma} \label{lemma:inbetween}
Let $a=0.9$ and $a'=1.9$. There exists 
$\delta> 0$ such that 
if $p, q \in \mathcal P_{a} \setminus \mathcal P_{a'}$ are such that $\angle ( v_p,v_q)<\delta $ and $\angle ( w_p,w_q)<\delta $, then $q\in B_d(p,d(0,p))$ or $p\in B_d(q,d(0,q))$.
\end{lemma}

\begin{proof}
The values of $a $ and $a'$  have been chosen in such a way that 
\begin{equation*}
\dfrac{1}{2}+ \dfrac{\sqrt{1+4 a'^2}-1}{4}   -   \dfrac{2}{a'}
 \sqrt{\dfrac{ \sqrt{1+4 a'^2}  -1 }{2 } }   <0
\end{equation*}
and
\begin{equation*}
\frac{1}{2a}+\frac{1}{2}-\frac{\sqrt{4 a^2+1}-1}{a} <0~.
\end{equation*}
Hence, we can fix some $\eps>0$ so that
$$
   \dfrac{1}{2} +\dfrac{ \sqrt{1+4 a'^2}-1 }{4} 
  - 
  2 (1-\eps) 
  \dfrac{1}{a'}
 \sqrt{\dfrac{ \sqrt{1+4 a'^2}  -1 }{2 } } 
 +2 r^2 R \eps    +\frac{r^2}{4}\eps^2 
R^2   <0$$
and
$$\dfrac{1}{2a}
 + \dfrac{1}{2} 
- (1-\eps) 
\dfrac{\sqrt{4 a^2+1} -1}{a} <0~.$$ 

Let $p, q \in \mathcal P_{a} \setminus \mathcal P_{a'}$. Assume that $\angle ( v_p,v_q)<\delta $ and $\angle ( w_p,w_q)<\delta $ where $\delta>0$ is given by Lemma~\ref{lem:small-angles}, i.e., is such that \eqref{eq:area}, \eqref{eq:scalar1}, and \eqref{eq:scalar2} hold for our choice of $\eps$. Arguing as in the proof of Lemma~\ref{lemma:away}, one can assume with no loss of generality that $p \in \partial B$ and $q \in B$. 

Let us prove that $q\in B_d(p,1)$. By Lemma~\ref{lemma:Aq}, this is equivalent to  $A_p(q)\leq 0$. By~\eqref{e:bound2-Ap(q)}, we have
\begin{multline*}
A_p(q) \leq \dfrac{1}{2}\norm{v_q}^2+ \dfrac{1}{2} \norm{v_q}^2+ \dfrac{1}{2} \norm{w_q}^2+ \dfrac{1}{2}\norm{w_q}^2 - 2 (1-\eps)  \norm{v_p} \norm{v_q} \\
\quad - 2 (1-\eps)  \norm{w_p} \norm{w_q}
+2 r^2 R^2 \eps  \norm{v_q} +\frac{r^2}{4}\eps^2 
R^3 \norm{v_q} ~.
\end{multline*}
To bound the 1st term, we use that $q\in \mathcal P_{a} $, i.e., $\norm{v_q}^2< \dfrac{R}{a}   \norm{w_q} $. To bound the 2nd term, we use that $q\in B$, hence $\norm{v_q} \leq R$. To bound the 3rd term, we use that $q\not \in \mathcal P_{a'}$ through both~\eqref{eq:bound21} and the fact that $\norm{w_q}  <  \dfrac{a'}{R} \norm{v_q}^2\leq  a'\norm{v_q} $. In the 4th one, we use that $q\in B$, hence $\norm{w_q} \leq R$. In the 5th term, we use that $p\notin \mathcal P_{a'}$ through~\eqref{eq:bound22} and in the 6th one we use that $p\in \mathcal P_{a}$ through~\eqref{eq:bound1}. This gives
\begin{equation*}
\begin{split}
A_p(q)& \leq 
 \dfrac{R}{2a}\norm{w_q}+ \dfrac{R}{2} \norm{v_q}
 + \dfrac{R}{4a'} ( \sqrt{1+4 a'^2}-1)
 a'
 \norm{v_q}
 + \dfrac{R}{2}\norm{w_q}  
  \\
& \phantom{\dfrac{R}{2a}\norm{w_q}+ } - 2 (1-\eps) \dfrac{R}{a'}
 \sqrt{\dfrac{ \sqrt{1+4 a'^2}  -1 }{2 } }   \norm{v_q} 
- 2 (1-\eps) 
\dfrac{R}{2a} (\sqrt{4 a^2+1}-1) \norm{w_q} \\
&\phantom{\dfrac{R}{2a}\norm{w_q}+ \dfrac{R}{2} \norm{v_q}
 + \dfrac{R}{4a'} \quad \quad \quad  \quad \quad \quad \quad \quad \quad  \quad  \quad \quad\quad\;} + 2 r^2 R^2\eps  \norm{v_q} +\frac{r^2}{4}\eps^2 
R^3  \norm{v_q} \\
  & =
  R \norm{v_q}
  \left(
   \dfrac{1}{2} +\dfrac{ \sqrt{1+4 a'^2}-1 }{4} 
  - 2 (1-\eps) 
  \dfrac{1}{a'}
 \sqrt{\dfrac{ \sqrt{1+4 a'^2}  -1 }{2 } } 
 +2 r^2 R \eps    +\frac{r^2}{4}\eps^2 
R^2   
\right) \\
& \phantom{R \norm{v_q}
   \dfrac{1}{2} +\dfrac{ \sqrt{1+4 a'^2}-1 }{4} 
  - 2 (1-\eps) 
  \dfrac{1}{a'}
 \sqrt{ } }  +
  R \norm{w_q}   \left(
 \dfrac{1}{2a}
 + \dfrac{1}{2} 
- (1-\eps) 
\dfrac{\sqrt{4 a^2+1} -1}{a}  
\right) \\
&\leq 0
\end{split}
\end{equation*}
by the choice of $\eps$.
\end{proof}

We are now going to conclude the proof of Theorem~\ref{thm:BCP-free-step2}. 

\begin{proof}[Proof of Theorem~\ref{thm:BCP-free-step2}] Let $a = 0.9$, $a'=1.9$ and let $\delta>0$ be small enough so that Lemmas~\ref{lemma:away},~\ref{lemma:near2a} and~\ref{lemma:inbetween} hold. Let $N$ be an upper bound of the maximum number of vectors in a $\max(r,n-r)$-dimensional Euclidean space that pairwise make an angle larger than $\delta /2$. Such a bound exists and is finite by compactness of finite dimensional Euclidean spheres. 

We are going to prove that a family of Besicovitch balls in $(\mathbb{F}_{r2},d)$ cannot have a cardinality larger than $3N^2$. This will imply that WBCP, and hence BCP by Corollary~\ref{cor:bcp-wbcp-homogeneous-qdist}, holds in $(\mathbb{F}_{r2},d)$. Let $\{B_d(p_i,r_i)\}_{i\in I}$ be a family of Besicovitch balls. Using left-translations and shrinking balls if necessary, one can assume with no loss of generality that $0 \in \cap_{i\in I}B_d(p_i,r_i)$ and $r_i=d(0,p_i)$ for all $i\in I$. If $\card I > 
3N^2$, by the pigeonhole principle 
one can find $I_1\subseteq I$ with 
$\card I_1 \geq  \card I / N>
3N$
such that
 $\angle ( v_{p_i},v_{p_j})<\delta$
 for all $i, j \in I_1$. Then, once again by the pigeonhole principle, one can find 
  $I_2\subseteq I_1$ with 
$\card I_2 \geq  \card I_1 / N>
3$
such that
$\angle ( w_{p_i},w_{p_j})<\delta$
for all $i, j \in I_2$. Finally,
there exists at least 2 distinct points $p_i$ and $p_j$ with $i, j\in I_2$ that both belong either to
 $\mathcal P_{a'}$, or 
 $\mathcal P_{a} \setminus \mathcal P_{a'}$ or 
$\mathcal P_{a}^c$.
Then Lemmas~\ref{lemma:away}, ~ \ref{lemma:near2a} and~\ref{lemma:inbetween}, lead to a contradiction since by definition of a family of Besicovitch balls, we have $p_i\notin B_d(p_j,r_j) = B_d(p_j,d(0,p_j))$ for all $i\not= j \in I$.
\end{proof}

\subsection{Stratified groups of step 2}
\label{subsect:BCP-stratified-step-2} 

This section is devoted to the proof of Corollary~\ref{cor:sec-bcp-homogeneous-groups-commuting-layers} in the case of stratified groups of step 2 as restated below.

\begin{theorem} \label{thm:BCP-stratified-step2}
Let $G$ be a stratified group of step 2. There exist homogeneous distances on $G$ for which BCP holds.
\end{theorem}

\begin{proof}
Let $G$ be a stratified group of step 2 and rank $r$ whose Lie algebra is endowed with a given stratification $\g = V_1 \oplus V_2$. Let $(Y_1,\dots,Y_r)$ be a basis of $V_1$. Let $\mathbb F_{r2}$ be the free-nilpotent Lie group of step 2 and rank $r$. With the conventions and notations of Section~\ref{subsect:BCP-free-step2}, let $\phi:\mathfrak f_{r2} \rightarrow \g$ denote the unique morphism of graded Lie algebras such that $\phi(X_i) = Y_i$ for $i=1,\dots,r$, which is surjective. Let $\varphi : \mathbb F_{r2} \rightarrow G$ denote the unique Lie group homomorphism such that $\varphi_* = \phi$. Let $d$ be a homogeneous distance on $\mathbb F_{r2}$ for which BCP holds. Such distances exist by Theorem~\ref{thm:BCP-free-step2} and~\cite{Hebisch-Sikora} (see also Example~\ref{ex:hebisch-sikora-distance}). Recall also that homogeneous distances are continuous. It follows from Proposition~\ref{prop:submetries-morphism-graded-algebra} that 
\begin{equation} \label{e:dist-BCP-statified-step2}
d_G(p,q):= d(\varphi^{-1}(\{p\}),\varphi^{-1}(\{q\})
\end{equation}
defines a homogeneous distance on $G$ and $\varphi : (\mathbb F_{r2},d) \rightarrow (G,d_G)$ is a submetry. To conclude the proof, the fact that WBCP holds on $(G,d_G)$ follows from Proposition~\ref{prop:bcp-submetry}. Hence BCP holds on $(G,d_G)$ by Corollary~\ref{cor:bcp-wbcp-homogeneous-qdist}.
\end{proof}  

\begin{remark} \label{rmk:BCP-step2-stratified} More generally, arguing as in the proof of Theorem~\ref{thm:BCP-stratified-step2}, one gets that, if $d$ is a continuous homogeneous quasi-distance satisfying BCP on $\mathbb F_{r2}$, then~\eqref{e:dist-BCP-statified-step2} defines a continuous homogeneous quasi-distance for which BCP holds on the stratified group $G$ of step 2 and rank $r$. In particular all homogeneous quasi-distances on $\mathbb F_{r2}$ whose unit ball centered at the origin are given by~\eqref{eq:ball:R} are continuous and hence induce on $G$ continuous homogeneous quasi-distances for which BCP holds. 
\end{remark}

\begin{remark} It can be checked that if $d$ is a homogeneous quasi-distance on $\mathbb F_{r2}$ whose unit ball centered at the origin is given by ~\eqref{eq:ball:R} for some $R>0$, then the unit ball centered at the origin for the quasi-distance $d_G$ given by~\eqref{e:dist-BCP-statified-step2} on the stratified group $G$ of step 2 can be described as a Euclidean ball centered at the origin in exponential coordinates of the first kind relative to a suitable choice of basis of $\g$ adapted to its stratification. More precisely, one can find a basis $(Z_1,\dots,Z_r)$ of $V_1$ and a basis $(Z_{r+1},\dots,Z_n)$ of $V_2$ such that, in exponential coordinates of the first kind relative to the basis $(Z_1,\dots,Z_n)$ of $\g$, we have
$$B_{d_G}(0,1) = \{p\in G;\; \sum_{i=1}^n p_i^2 \leq R^2\}~.$$
One can reasonably expect that for any choice of basis of $\g$ adapted to its stratification, homogeneous quasi-distances on $G$ whose unit ball centered at the origin is a Euclidean ball in exponential coordinates of the first kind relative to the chosen basis satisfy BCP. This would require technical modifications of our arguments and we do not wish here to go further about these technicalities.
\end{remark}

\subsection{Arbitrary groups with commuting different layers}

\label{subsect:BCP-commuting-layers}

We conclude in this section the proof of Theorem~\ref{thm:sec-BCP-commuting-layers} and Corollary~\ref{cor:sec-bcp-homogeneous-groups-commuting-layers}.

\begin{proof}[Proof of Theorem~\ref{thm:sec-BCP-commuting-layers} and Corollary~\ref{cor:sec-bcp-homogeneous-groups-commuting-layers}] 

Let $G$ be a graded group with commuting different layers. By Proposition~\ref{prop:structure-commuting-layers}, $G$ can be written as a direct product of powers of stratified groups of step $\leq 2$.

For an Abelian Lie group with trivial associated positive grading (i.e., stratified of step 1), the Euclidean distance (and more generally any distance induced by a norm) is a homogeneous distance that satisfies BCP. For a stratified group of step 2, we know by Theorem~\ref{thm:BCP-stratified-step2} that there exist homogeneous distances that satisfy BCP. Next, if $d$ is a homogeneous distance for which BCP holds on a graded group, then $d^{1/t}$ is a homogeneous quasi-distance on its $t$-power (see Example~\ref{ex:homogeneous-qdist-kpower}) that satisfies BCP by Proposition~\ref{prop:BCP-kpower}. In addition, since $d$ is a distance, it is continuous (see Corollary~\ref{cor:continuity-homogeneous-distance}), and hence $d^{1/t}$ is a continuous quasi-distance.

Hence, on each factor of the decomposition of $G$, there exist continuous homogeneous quasi-distances for which BCP, and hence WBCP, holds. Then Theorem~\ref{thm:sec-BCP-commuting-layers} follows from Theorem~\ref{thm:WBCP-productspaces}.

Note that if $G$ is a homogeneous group, all $t$-powers in its decomposition as a direct product are $t$-powers with $t\geq 1$. This implies the existence of homogeneous distances on $G$ for which BCP holds and proves Corollary~\ref{cor:sec-bcp-homogeneous-groups-commuting-layers}. 
\end{proof}

\section{Graded groups with two different layers not commuting}
 \label{sect:nobcp}
 
In this section we consider graded groups for which there are two different layers of the associated positive grading of their Lie algebra that do not commute, and we consider a more general class of quasi-distances defined as follows.

\begin{definition} [Self-similar quasi-distances on graded groups]\label{def:selfsimilar-dist}
Let $G$ be a graded group with associated dilations $(\delta_\lambda)_{\lambda>0} $. We say that a quasi-distance $d$ on $G$ is {\em self-similar} if it is left-invariant and one-homogeneous  with respect to some non-trivial dilation, i.e., if there exists $\lambda >0$, $\lambda\not=1$, such that $d(\delta_\lambda(p),\delta_\lambda(q)) = \lambda d(p,q)$ for all $p$, $q\in G$.
\end{definition}

Note that $d(\delta_\lambda(p),\delta_\lambda(q)) = \lambda d(p,q)$ implies $d(\delta_{\lambda^k}(p),\delta_{\lambda^k}(q)) = \lambda^k d(p,q)$ for all $k\in \Z$. In particular, the previous definition is equivalent to the existence of some $0<\lambda <1$ such that  $d(\delta_{\lambda^k}(p),\delta_{\lambda^k}(q)) = \lambda^k d(p,q)$ for all $p$, $q\in G$ and all $k\in\Z$.

We prove the following result.

\begin{theorem} \label{thm:sect-nobcp-self-similar-noncommuting}
Let $G$ be a graded group and let $\oplus_{t>0} V_t$ be the associated positive grading of its Lie algebra. Assume that $[V_t,V_s]\not = \{0\}$ for some $t\not= s$. Let $d$ be a self-similar quasi-distance on $G$ that is continuous with respect to the manifold topology. Then WBCP, and hence BCP, does not hold in $(G,d)$.
\end{theorem}

\begin{remark} Although  we will not use it here, it can be noticed that self-similar quasi-distances on graded groups are doubling, hence BCP and WBCP are equivalent in this context (see Proposition~\ref{prop:BCP-WBCP-Doubling}).
\end{remark}

Since homogeneous quasi-distances are in particular self-similar, we get the following corollary.

\begin{corollary} \label{cor:sec-nobcp-homogeneous-qdist-noncommuting}
Let $G$ be a graded group whose associated positive grading of its Lie algebra is given by $\oplus_{t>0} V_t$. Assume that $[V_t,V_s] \not= \{0\}$ for some $t\not=s$. Let $d$ be a continuous homogeneous quasi-distance on $G$. Then BCP does not hold in $(G,d)$.
\end{corollary}

Homogeneous distances on homogeneous groups are continuous with respect to the manifold topology, recall Corollary~\ref{cor:continuity-homogeneous-distance}. Hence, in such a case, one can drop the continuity assumption and we get the following corollary.

\begin{corollary} \label{cor:sec-nobcp-homogeneous-groups-noncommuting}
Let $G$ be a homogeneous group whose associated positive grading of its Lie algebra is given by $\oplus_{t\geq 1} V_t$. Assume that $[V_t,V_s] \not= \{0\}$ for some $t\not=s$. Let $d$ be a homogeneous distance on $G$. Then BCP does not hold in $(G,d)$.
\end{corollary}

The proof of Theorem~\ref{thm:sect-nobcp-self-similar-noncommuting} is divided into two steps. First, we prove that there does not exist continuous self-similar quasi-distances that satisfy WBCP on the non-standard Heisenberg groups, see Theorem~\ref{thm:heisenberg-case}. Next, we deduce Theorem~\ref{thm:sect-nobcp-self-similar-noncommuting} from Theorem~\ref{thm:heisenberg-case} together with Proposition~\ref{prop:structure-algebra-noncommuting-layers} and the use of submetries via a generalization of Proposition~\ref{prop:submetries-morphism-graded-algebra} to continuous self-similar distances, see Proposition~\ref{prop:submetries-morphism-self-similar}. Let us stress that one of the main differences between homogeneous and self-similar quasi-distances are topological issues that we will explain in Section~\ref{subsect:nobcp-noncommuting-layers}.

\subsection{Non-standard Heisenberg groups} \label{subsect:nobcp-nonstandard-heis}

This section is devoted to the proof of Theorem~\ref{thm:heisenberg-case} below. To simplify notations, we denote here by $\h$ the first Heisenberg Lie algebra, by $(X,Y,Z)$ a standard basis of $\h$, and by $\HH$ the first Heisenberg group. Recall from Example~\ref{ex:heisenberg} that for $\alpha >1$, $\HH$ is called the non-standard  Heisenberg group of exponent $\alpha$ when considered as a graded group whose Lie algebra is endowed with the non-standard grading of exponent $\alpha$, i.e., the grading given by $\h =W_1 \oplus W_\alpha \oplus W_{\alpha+1}$ where $ W_1:=\Span\{X\}$, $W_\alpha:=\Span\{Y\}$,  $W_{\alpha+1}:=\Span\{Z\}$, and where the only non-trivial bracket relation is $[X,Y] = Z$.

\begin{thm} \label{thm:heisenberg-case}
Let $\alpha>1$. There does not exist continuous self-similar quasi-distances for which WBCP holds on the non-standard Heisenberg group of exponent $\alpha$.
\end{thm}

In this statement, continuity of self-similar quasi-distances means continuity with respect to the manifold topology. We refer to Section~\ref{subsect:nobcp-noncommuting-layers} for the study of topological properties of self-similar quasi-distances.

From now on in this section, we fix $\alpha >1$. Following Example~\ref{ex:heisenberg}, we use exponential coordinates of the first kind, we write $p \in \HH$ as $p=\exp( x X + y Y + z Z) $ and we identify $p$ with $(x,y,z)$. Recall that dilations $(\delta_\lambda)_{\lambda>0}$ relative to the non-standard grading of exponent $\alpha$ are given by 
\begin{equation}
\label{e:inhom:dil}
\delta_\lambda (x,y,z) = (\lambda x, \lambda^\alpha y,   \lambda^{\alpha + 1} z)~.
\end{equation}

To prove Theorem~\ref{thm:heisenberg-case}, we argue by contradiction. We let $d$ be a self-similar quasi-distance on the non-standard Heisenberg group of exponent $\alpha$. Hence $d$ is left-invariant and, for some fixed $0<\rho<1$, we have
\begin{equation} \label{e:dist-self-similar}
d(\delta_{\rho^k}(p),\delta_{\rho^k}(q)) = \rho^k\, d(p,q) 
\end{equation}
for all $p$, $q\in \HH$ and all $k\in \Z$. 

Next, we assume that $d$ is continuous on $\HH\times \HH$ with respect to the manifold topology. We set $B:=B_d(0,1)$. The continuity of $d$ implies in particular that $B$ is closed and that its boundary $\partial B$ is given by $\partial B = \{p\in \HH;\; d(0,p)=1\}$.

Finally, arguing by contradiction, we assume that WBCP holds on $(\HH,d)$.

To get a contradiction, we first prove a series of lemmas, Lemma~\ref{lem:condition-for-bcp} to Lemma~\ref{lem:main7} below. The final conclusion will follow from these lemmas and is given at the end of this section.

First, the assumption about the validity of WBCP has the following consequence.

\begin{lemma} \label{lem:condition-for-bcp}
For all $p\in B$ and all $\overline{\lambda}>0$, there exists $0<\lambda<\overline\lambda$ such that $p\cdot \delta_{\lambda} (p^{-1}) \in B$.
\end{lemma}

\begin{proof} We first prove that for all $p\in \partial B$, there exist arbitrarily large values of $j$ such that $d(p,\delta_{\rho^j} (p)) \leq 1$. By contradiction, assume that one can find $p\in \partial B$ and $k\geq 0$ such that $d(p,\delta_{\rho^j} (p)) > 1$ for all $j\geq k$. For $l\geq 0$, set $r_l:= \rho^{lk}$ and $q_l:=\delta_{r_l}(p)$. We have $d(0,q_l)=r_l$ by \eqref{e:dist-self-similar} hence $0 \in \cap_{l\geq 0} B_d(q_l,r_l)$. We also have 
\begin{equation*}
d(q_j, q_l) = d(\delta_{r_j}(p),\delta_{r_l}(p)) = r_j\, d(p,\delta_{\rho^{(l-j)k}}(p)) > r_j = \max (r_j,r_l)
\end{equation*}
for all $0\leq j< l$. Hence $q_j \not \in B_d(q_l,r_l)$ for all $j\not= l$. It follows that for all finite set $J\subset \N$, $\{B_d(q_j,r_j)\}_{j\in J}$ is a family of Besicovitch balls, which contradicts the validity of WBCP in $(\HH,d)$.

To conclude the proof, let $p\in \partial B$. Then $p^{-1} \in \partial B$ by left-invariance of $d$ and it follows that $d(p^{-1},\delta_{\rho^j} (p^{-1})) \leq 1$ for arbitrarily large values of $j$. Hence, by left-invariance of $d$, we get that $p\cdot \delta_{\lambda} (p^{-1}) \in B$ for arbitrarily small values of $\lambda$. If $p$ belongs to the interior of $B$, the claim follows from the continuity of the map $\lambda \mapsto p\cdot \delta_{\lambda} (p^{-1})$.
\end{proof}

As a consequence of Lemma~\ref{lem:condition-for-bcp}, we prove in Lemma~\ref{lem:main1} a geometric property of the unit ball $B$. Namely,  starting at a point $p\in \partial B$ with $x_p\not= 0$, there is segment that is all contained in $B$. This segment is a part of the flow line of the vector field $-X$ starting at $p\in \partial B$.
 
\begin{lemma} \label{lem:main1} For all $p = (x_p,y_p,z_p)\in \partial B$ with $x_p\not = 0$, the segment
\begin{equation*}
\hat\sigma_p := \left\{\left((1-t) x_p, y_p , z_p + t \,\frac{x_p y_p}{2}\right); \; t\in [0,1]\right\}
\end{equation*}
is contained in $B$.
\end{lemma}

\begin{proof}
The segment $\hat\sigma_p$ is a part of the flow line of the vector field $-X$. For technical convenience, we will thus use in this proof exponential coordinates of the second kind in which $X$ is a constant vector field. Namely, for $p\in \HH$, we write $p = \exp(x_3(p) Z) \cdot \exp(x_2(p) Y) \cdot \exp(x_1(p) X)$ and we identify $p$ with $[x_1(p),x_2(p),x_3(p)]$. 
The relation between exponential coordinates of the second kind $[x_1(p),x_2(p),x_3(p)]$ and exponential coordinates of the first kind $(x_p,y_p,z_p)$ when $p$ is written as $\exp (x_p X + y_p Y +z_p Z)$ (which are used elsewhere in this section) is given by $x_1(p) = x_p$, $x_2(p) =y_p$ and $x_3(p) =  z_p + x_p y_p/2$. In exponential coordinates of the second kind, the group law is given by
\begin{equation*}
\begin{cases}
\, x_1(p\cdot p') = x_1(p) + x_1(p')\\
\, x_2(p\cdot p') = x_2(p) + x_2(p')\\
\, x_3(p\cdot p') =  x_3(p) + x_3(p') + x_1(p) \, x_2(p')~,
\end{cases}
\end{equation*}
the vector field $X$ is the constant vector field $X = \partial_{x_1}$, and  the segment $\hat\sigma_p$ writes as
$$\hat\sigma_p = \{[(1-t) x_1(p), x_2(p), x_3(p)]; \; t\in [0,1]\}~.$$

Let $\theta \in (0,\pi/2)$ be fixed. For $p = [x_1(p),x_2(p),x_3(p)]$ with $x_1(p) > 0$, we define $C_{p,\theta}$ as the portion of the half-cone with vertex $p$, axis $\{\exp(- x_1 X);\; x_1 \geq 0\}$ and aperture $2\theta$ contained in the half space $\{x_1 \geq 0\}$. If $p = [x_1(p),x_2(p),x_3(p)]$ with $x_1(p) < 0$, then $C_{p,\theta}$ is defined as the portion of the half-cone with vertex $p$, axis $\{\exp( x_1 X); x_1 \geq 0\}$ and aperture $2\theta$ contained in the half space $\{x_1 \leq 0\}$. Namely, if $x_1(p) > 0$, 
\begin{multline*}
 C_{p,\theta} := \{[x_1,x_2,x_3]\in \HH; \;  0 \leq x_1 \leq x_1(p) \text{ and } \; \\ ((x_2 - x_2(p))^2 + (x_3 - x_3(p))^2)^{1/2} \leq  (x_1(p) - x_1)\, \tan \theta \} 
\end{multline*}
and, if  $x_1(p) < 0$,
\begin{multline*}
 C_{p,\theta} := \{[x_1,x_2,x_3]\in \HH; \;  x_1(p) \leq x_1 \leq 0 \text{ and } \; \\ ((x_2 - x_2(p))^2 + (x_3 - x_3(p))^2)^{1/2} \leq  (x_1 - x_1(p))\, \tan \theta \} ~.
\end{multline*}

For $p\in \HH$,  we set 
\begin{equation*}
\gamma_p(\lambda) := p \cdot \delta_\lambda(p^{-1})~.
\end{equation*}
Noting that, in exponential coordinates of the second kind, $p^{-1} = [-x_1(p), -x_2(p), -x_3(p) + x_1(p) \, x_2(p)]$, we have 
\begin{equation} \label{e:coord-gamma}
 \begin{cases}
 \, x_1(\gamma_p(\lambda)) = (1-\lambda)\, x_1(p)\\
 \, x_2(\gamma_p(\lambda)) = (1-\lambda^\alpha)\, x_2(p)\\
 \, x_3(\gamma_p(\lambda)) = (1-\lambda^{\alpha+1})\, x_3(p) + (\lambda^{\alpha+1}-\lambda^\alpha) \, x_1(p) \, x_2(p)~.
 \end{cases}
\end{equation} 
If $x_1(p) \not = 0$, we get
\begin{equation*}
\frac{d}{d\lambda} \gamma_p(\lambda)_{|\lambda=0} = -x_1(p) \, X
\end{equation*}
hence $\gamma_p(\lambda) \in C_{p,\theta}$ for all $\lambda \geq 0$ small enough. Since $C_{q,\theta} \subset C_{p,\theta}$ for all $q \in C_{p,\theta}$ with $x_1(q) \not= 0$, it follows that, for all $q \in C_{p,\theta}$ with $x_1(q) \not= 0$,
\begin{equation} \label{e:gamma-inside-cone}
\gamma_q(\lambda) \in C_{p,\theta} \; \text{ for all } \lambda \geq 0 \text{ small enough.}
\end{equation}

For $p\in \HH$ with $x_1(p) \not= 0$ and $\theta \in (0,\pi/2)$ fixed, one can find $L_{p,\theta}>0$ such that all curves $(\gamma_q(\lambda))_{\lambda \in [0,1]}$ for $q \in C_{p,\theta}$ are $L_{p,\theta}$-Lipschitz. This can be easily checked from the explicit expression of $\gamma_q(\lambda)$, see \eqref{e:coord-gamma}.

For $t\in [0,1]$, we define $E_{p,\theta}(t)$ as the intersection of $C_{p,\theta}$ with the two dimensional plane $\{[x_1,x_2,x_3]\in \HH; x_1 = (1-t) x_1(p)\}$,
\begin{equation*}
E_{p,\theta}(t) := C_{p,\theta} \cap \{[x_1,x_2,x_3]\in \HH; \; x_1 = (1-t) x_1(p)\}~.
\end{equation*}

Let $p\in \partial B$ with $x_1(p) \not= 0$, $s\in (0,1)$ and $\theta \in (0,\pi/2)$ be fixed. We first prove that 
\begin{equation*}
E_{p,\theta}(s) \cap B \not= \emptyset~.
\end{equation*}
We set $L:= (1-s)^{-1} L_{p,\theta}$ and 
\begin{equation*}
I := \{t\in [0,s];\; \exists\, \gamma:[0,t] \rightarrow C_{p,\theta} \; \; L-\text{Lipschitz}, \, \gamma(0) = p, \; \gamma(t) \in E_{p,\theta}(t) \cap B\}~,
\end{equation*}
and we actually prove that $s\in I$ from which the claim follows.

First, $0\in I$ hence $I$ is nonempty. Second, $C_{p,\theta}$ and $E_{p,\theta}(t) \cap B$ being closed, $I$ is closed by Ascoli-Arzel\`a Theorem. Hence $\sup I \in I$. By contradiction, assume that $t:=\sup I <s$. Since $t\in I$, one can find a $L$-Lipschitz curve $\gamma:[0,t] \rightarrow C_{p,\theta}$ such that $\gamma(0) = p$ and $\gamma(t) \in E_{p,\theta}(t) \cap B$. Set $q:=\gamma(t)$. Since $q\in C_{p,\theta}$ with $x_1(q) \not = 0$, it follows from \eqref{e:gamma-inside-cone} that $\gamma_q(\lambda) \in C_{p,\theta}$ for all $\lambda \geq 0$ small enough. Since $q\in B$, it follows from Lemma~\ref{lem:condition-for-bcp} that $\gamma_q(\lambda) \in B$ for arbitrarily small positive values of $\lambda$. Hence, one can find $\overline\lambda >0$ such that $\gamma_q(\lambda) \in C_{p,\theta}$ for all $0\leq\lambda \leq \overline \lambda$ and $\gamma_q(\overline \lambda) \in B$. It follows that one can find $\eta>0$ such that the curve $c$ defined on $[0,t+\eta]$ by 
\begin{equation*}
\begin{cases}
c(u):=\gamma(u) \quad \quad \quad \; \text{ if } u\in [0,t] \\
c(u):= \gamma_q\left(\displaystyle\frac{u-t}{1-t}\right) \; \text{ if } u\in [t,t+\eta]
\end{cases}
\end{equation*}
satisfies $c([0,t+\eta]) \subset  C_{p,\theta}$ and $c(t+\eta) \in B$. Since $\gamma_q$ is $L_{p,\theta}$-Lipschitz, $c$ is $(1-t)^{-1} L_{p,\theta}$-Lipschitz, and hence $L$-Lipschitz, on $[t,t+\eta]$. Finally, by definition of $\gamma_q$ (recall \eqref{e:coord-gamma}) and since $x_1(q) = (1-t)x_1(p)$, we have 
\begin{equation*}
x_1(c(t+\eta)) = \left(1- \frac{\eta}{1-t}\right)\; x_1(q) = (1-(t+\eta))\; x_1(p)
\end{equation*}
and hence $c(t+\eta) \in E_{p,\theta}(t+\eta)$. This shows that $t+\eta \in I$, which gives a contradiction.

Hence, for $p\in \partial B$ with $x_1(p) \not= 0$ and $s\in (0,1)$, we have $E_{p,\theta}(s) \cap B \not = \emptyset$ for all $\theta\in (0,\pi/2)$. Letting $\theta \downarrow 0$ and since $B$ is closed, it follows that $[(1-s) x_1(p),x_2(p),x_3(p)] \in B$ for all $s\in (0,1)$. Using once again the fact that $B$ is closed, we finally get that the closed segment $\hat\sigma_p$ is contained in $B$ as wanted.
\end{proof}

Lemma~\ref{lem:main2} to Lemma~\ref{lem:main6} below are successive consequences of Lemma~\ref{lem:main1}. In addition to Lemma~\ref{lem:main1}, the only properties used to prove these lemmas are the left-invariance of the quasi-distance $d$ and topological properties of the unit ball $B$.

\begin{lemma} \label{lem:main2}
For all $p = (x_p,y_p,z_p)\in \partial B$ with $x_p\not = 0$, the segment
\begin{equation*}
\sigma_p := \left\{\left((1-t) x_p, y_p , z_p - t \,\frac{x_p y_p}{2}\right); \; t\in [0,1]\right\}
\end{equation*}
is contained in $B$.
\end{lemma}

\begin{proof}
By left-invariance of $d$, we have $q^{-1} \in B$ for all $q\in B$. Then it follows from Lemma~\ref{lem:main1} that for $p\in \partial B$ with $x_p\not = 0$, $\sigma_p = (\hat\sigma_{p^{-1}})^{-1} \subset B$.
\end{proof}

Given $p=(0,y_p,z_p) \in\HH$ and $w
 >0$, we set
\begin{equation*}
D(p,w) : = \{(-t,y_p,z_p -tw+u);\; t\geq 0,\; u\geq 0\}~.
\end{equation*}
It is the two dimensional region in the plane $\{y=y_p\}$ above the half-line starting at $p$ with direction $(-1,0,-w)$.

\begin{lemma} \label{lem:main3} Let $\overline y>0$ be such that $(0,\overline{y},z) \in B$ for some $z>0$. Set $\overline{z}:= \max(z>0;\; (0,\overline{y},z) \in B)$ and $\overline p:=(0,\overline{y}, \overline{z})$. Then $\overline p\in \partial B$ and for all $0< w  < \overline{y}/2$, we have
\begin{equation} \label{e:main3}
D(\overline p,w) \cap B = \{ \overline p\}~.
\end{equation}
\end{lemma}

\begin{proof}
By contradiction, assume that there is some point $q \in D(\overline p,w) \cap B$ with $q\not = \overline p$. Then, by definition of $D(\overline p,w)$, we have $q=(-t,\overline{y},\overline{z} - tw+u)$ for some $u, t\geq 0$. Since $q\not = \overline p$ and by choice of $\overline{z}$, we have $t>0$ and hence $x_q=-t\not=0$. By Lemma~\ref{lem:main2}, it follows that $\sigma_q \subset B$. In particular, the end point $(0,\overline{y},\overline{z}-tw+u+t\overline{y}/2)$ of $\sigma_q$ belongs to $B$. By choice of $\overline{z}$, we must then have $-tw+u+t\overline{y}/2 \leq 0$, and 
hence
\begin{equation*}
t\, \left(w-\frac{\overline y}{2}\right) \geq u \geq 0~.
\end{equation*}
This contradicts the assumption $0< w  < \overline{y}/2$ and concludes the proof.
\end{proof}

Given $p=(0,y_p,z_p)\in \HH $ with $y_p, z_p>0$ and $v > 0$, we set 
$$S(p,v) := \{  (  0, (1-s) y_p , z_p + sv  + u) \,;\, s\in [0,1],  u \geq  0\}~.$$ 
It is the two dimensional region in the plane $\{x=0\}$ above the segment from the $z$-axis to $p$ with slope $-v/y_p$.

\begin{lemma} \label{lem:main4}
There exist $p=(0,y_p,z_p)\in \partial B$
with $y_p, z_p>0$ and $v > 0$ such that
 $$ S(p,v) \cap B = \{p\}$$
and such that, for all $0\leq y \leq y_p$, $(0,y,z) \in B$ for some $z>0$.
\end{lemma}

\begin{proof}
The only property of the unit ball $B$ used in this proof is the fact that $B$ is a compact neighborhood of the origin. Let $p_0=(0,y_0,z_0)$ with $y_0, z_0>0$ be a point in the interior of $B$ such that $(0,y,z) \in B$ for all $0\leq y \leq y_0$ and all $0\leq z \leq z_0$. For $q = (0,y_q,z_q)$ with $y_q, z_q>0$, let $$I(q):= \{(0,y,z_q);\; 0\leq y \leq y_q\}$$ denote the horizontal segment from the $z$-axis to $q$ and let $$I^+(q) := \{(0,y,z_q+u);\; 0\leq y \leq y_q, u>0\}$$
denote the two dimensional infinite rectangular strip strictly above $I(q)$ in the plane $\{x=0\}$. Set $$\overline{z}:= \inf\{z>z_0;\; I^+((0,y_0,z)) \cap B = \emptyset\}~.$$
Since $B$ is bounded and $p_0$ belongs to the interior of $B$, we have $z_0 < \overline{z} <+\infty$. We set $\overline q:=(0,y_0,\overline z)$. We have $I^+(\overline q) \cap B = \emptyset$ and, since $B$ is closed, $I(\overline q) \cap B \not = \emptyset$.

If there is some point $p =(0,y_p,\overline z) \in I(\overline q) \cap B$ with $y_p>0$, then, for any $v>0$, we have 
$S(p,v)\setminus \{p\} \subset I^+(\overline q)$ hence $S(p,v)
 \cap B = \{p\}$. Note that $p\in \partial B$.
 
Otherwise $I(\overline q) \cap B  = \{(0,0,\overline z)\}$. For $q = (0,y_0,z_q)$ with $z_q>0$, let $$J(q):= \{(0,(1-t)y_0, z_q + t(\overline{z}-z_0);\; t\in [0,1]\}$$ denote the segment in the plane $\{x=0\}$ from the $z$-axis to $q$ with slope $-(\overline z - z_0)/y_0$ and let $$J^+(q):=\{(0,(1-t)y_0, z_q + t(\overline{z}-z_0)+u;\; t\in [0,1],\; u>0 \} $$ the two dimensional region strictly above $J(q)$ in the plane $\{x=0\}$. 
Set $$\hat{z}:= \inf\{z>z_0;\; J^+((0,y_0,z)) \cap B = \emptyset\}~.$$
Arguing as above, we have $z_0 < \hat{z} \leq \overline z$. We set $\hat{q}:= (0,y_0,\hat z)$. We have $J^+(\hat{q}) \cap B = \emptyset$ and $J(\hat{q}) \cap B \not = \emptyset$. Since $\hat z - z_0>0$, $J(\hat{q})$ meets the $z$-axis at $(0,0, \overline z +\hat z - z_0)$ which belongs to $I^+(\overline q)$. Hence it cannot belong to $B$ and there is some point $p=(0,y_p,z_p)\in J(\hat{q}) \cap B$ with $y_p>0$. Note that $z_p>0$ and that $p$ belongs to $\partial B$. Then for any $v > (\overline z - z_0) y_p/y_0$, we have $S(p,v) \setminus \{p\} \subset J^+(\hat{q})$ hence $S(p,v) \cap B = \{p\}$.

Finally, in both cases, we have $y_p\leq y_0$. Then, for all $0\leq y \leq y_p$, we have $0\leq y\leq y_0$ and, by choice of $p_0$,  we get that $(0,y,z) \in B$ for all $0\leq z\leq z_0$. In particular, $(0,y,z) \in B$ for some $z>0$, which concludes the proof.
\end{proof}

Lemma~\ref{lem:main5} below is a consequence of Lemma~\ref{lem:main3} together with  Lemma~\ref{lem:main4}. Given $p=(0,y_p,z_p)\in \HH$ with $y_p, z_p>0$ and $v > 0$, we set 
\begin{equation*} 
R(p,v) := \left\{\left(-t,(1-s)y_p,z_p+sv -\frac{t(1-s)y_p}{4} + u\right); \; u, t \geq 0, \, s\in [0,1)\right\}~.
\end{equation*}
It is the three dimensional region obtained by the following union. For $s\in [0,1]$, let $p_s:=(0,(1-s)y_p,z_p+sv)$. Note that $\{p_s;\; s\in [0,1]\}$ is the lower boundary of $S(p,v)$. Then, for $s\in [0,1)$,  $D(p_s,(1-s)y_p/4)$ is the intersection of $R(p,v)$ with the plane $\{y=(1-s)y_p\}$ and we have
\begin{equation*} 
R(p,v) = \bigcup_{s\in [0,1)} D(p_s,(1-s)y_p/4)~.
\end{equation*}

\begin{lemma} \label{lem:main5}
There exists a point $p=(0,y_p,z_p)\in \partial B$
with $y_p, z_p>0$ and $v > 0$ such that $R(p,v) \cap B = \{p\}$.
\end{lemma}

\begin{proof}
Let $p=(0,y_p,z_p) \in \partial B$ with $y_p, z_p>0$ and $v>0$ be given by Lemma~\ref{lem:main4}. Since $S(p,v) \cap B = \{p\}$, we have $z_p = \max\{z>0;\; (0,y_p,z) \in B\}$. Then it follows from Lemma~\ref{lem:main3} that 
\begin{equation} \label{e:1}
D(p,y_p/4) \cap B = \{p\}~.
\end{equation}
Recall that $D(p,y_p/4)$ is the intersection of $R(p,v)$ with the plane $\{y=y_p\}$.

For $s\in (0,1)$, set $q_s:=(0,(1-s)y_p,z_s)$ where $z_s := \max\{z>0;\; (0,(1-s)y_p,z) \in B\}$. Note that $z_s$ is well defined since for all $s\in (0,1)$, we have $(0,(1-s)y_p,z) \in B$ for some $z>0$ by Lemma~\ref{lem:main4}. We have $q_s \in \partial B$ and it follows from Lemma~\ref{lem:main3} that $D(q_s,(1-s)y_p/4) \cap B = \{q_s\}$. On the other hand, $q_s \not \in S(p,v)$ since $S(p,v) \cap B = \{p\}$ and hence $z_s < z_p + s v$. It follows that, for all $s\in (0,1)$,
$$D(p_s,(1-s)y_p/4) \cap B \subset (D(q_s,(1-s)y_p/4) \cap B) \setminus \{q_s\}~,$$
where $p_s:=(0,(1-s)y_p,z_p+sv)$. Hence,
\begin{equation} \label{e:2}
D(p_s,(1-s)y_p/4) \cap B = \emptyset~.
\end{equation}
Recalling that $D(p_s,(1-s)y_p/4)$ is the intersection of $R(p,v)$ with the plane $\{y=(1-s)y_p\}$, the lemma finally  follows from \eqref{e:1} and \eqref{e:2}.
\end{proof}

\begin{lemma} \label{lem:main6}
There exists a point $q=(0,y_q,z_q)\in \partial B$
with $y_q, z_q < 0$ and $v > 0$ such that $\hat R(q,v) \cap B_d(q,1) = \{0\}$ where 
\begin{equation} \label{e:hatR}
\hat R(q,v) := \left\{\left(-t,s y_q, sv +\frac{t(3-s)y_q}{4} + u\right); \; u, t \geq 0, \, s\in [0,1)\right\}~.
\end{equation}
\end{lemma}

\begin{proof}
Let $p=(0,y_p,z_p)\in \partial B$
with $y_p, z_p>0$ and $v>0$ be given by Lemma~\ref{lem:main5} and set $q:= p^{-1}$. By left-invariance of $d$, we have $$ q \cdot R(q^{-1},v) \cap B_d(q,1) = p^{-1} \cdot (R(p,v) \cap B) = \{0\}~.$$
Noting that $q \cdot R(q^{-1},v) = p^{-1} \cdot R(p,v) = \hat R(q,v)$, we get the required conclusion.
\end{proof}

The next lemma gives a geometric property of dilations of the region $\hat R(q,v)$.

\begin{lemma} \label{lem:main7}
Let $q=(0,y_q,z_q)\in \HH$
with $y_q, z_q < 0$, $v > 0$ and let $\hat R(q,v)$ be given by~\eqref{e:hatR}. Then, for all $\hat q=(x_{\hat q},y_{\hat q},z_{\hat q})$ with $x_{\hat q},y_{\hat q}<0$, there exists $\hat\lambda >0$ such that, for all $\lambda >\hat\lambda$, $\hat q \in \delta_\lambda(\hat R(q,v))$.
\end{lemma}

\begin{proof}
Let $\hat q=(x_{\hat q},y_{\hat q},z_{\hat q})$ with $x_{\hat q},y_{\hat q}<0$ be given. To prove that $\hat q \in \delta_\lambda(\hat R(q,v))$, we have to find $t\geq 0$, $u\geq 0$ and $s\in [0,1)$ such that
\begin{equation} \label{e:main7}
 \begin{cases}
 \, \hat x_q = -t \lambda\\
 \, \hat y_q = \lambda^\alpha s y_q\\
 \, \hat z_q = \lambda^{\alpha+1} (sv +\displaystyle\frac{t(3-s)y_q}{4} + u)~.
 \end{cases}
\end{equation} 
From the first equation, we get $t= - \hat x_q / \lambda >0$ since $\hat x_q <0$. From the second equation, we get $s = \lambda^{-\alpha} \hat y_q y_q^{-1}$. We have  $s>0$ since $\hat y_q <0$ and $y_q <0$. We also have $s<1$ for all $\lambda>0$ large enough. The third equation gives
\begin{equation*}
\begin{split}
u &= - \displaystyle\frac{t(3-s)y_q}{4}- sv + \lambda^{-\alpha-1} \hat z_q \\
&= \displaystyle\frac{\hat x_q y_q}{4 \lambda} \, \left(3-\lambda^{-\alpha} \frac{\hat y_q}{y_q}\right) - \lambda^{-\alpha} \frac{\hat y_q}{y_q} v + \lambda^{-\alpha-1} \hat z_q = \frac{3 \hat x_q y_q}{4} \frac{1}{\lambda}  + o(\frac{1}{\lambda}).
\end{split}
\end{equation*}
It follows that $u>0$ for $\lambda >0$ large enough. All together, we get that, for all $\lambda >0$ large enough, one can find $t\geq 0$, $u\geq 0$ and $s\in [0,1)$ such that \eqref{e:main7} holds as wanted.
\end{proof}

We are now going to conclude the proof of Theorem~\ref{thm:heisenberg-case}.

\begin{proof} [Proof of Theorem~\ref{thm:heisenberg-case}]
Recall that we are arguing by contradiction. We consider a continuous self-similar quasi-distance $d$ on the non-standard Heisenberg group of exponent $\alpha$ and we are assuming that $d$ satisfies WBCP.
To get a contradiction, we are going to construct with the help of Lemma~\ref{lem:main6} and Lemma~\ref{lem:main7} families of Besicovitch balls with arbitrarily large cardinality.

Let us choose $q_1=(x_1,y_1,z_1)$ with $x_1, y_1 <0$ and set $r_1:=d(0,q_1)$. By induction assume that $q_1=(x_1,y_1,z_1), \dots, q_m=(x_m,y_m,z_m)$ have already been chosen so that $x_i, y_i <0$ for all $i=1,\dots,m$ and so that $\{B_d(q_i,r_i)\}_{i=1}^m$ is a family of Besicovitch balls where $r_i:=d(0,q_i)$.

Let $q=(0,y_q,z_q)\in \partial B$
with $y_q, z_q < 0$ and $v > 0$ be given by Lemma~\ref{lem:main6}. For all $k\geq 1$, we have 
$$\delta_{\rho^{-k}} (\hat R(q,v)) \cap B_d(\delta_{\rho^{-k}}(q),\rho^{-k}) = \delta_{\rho^{-k}} (\hat R(q,v) \cap B_d(q,1)) = \{0\}~.$$
On the other hand, it follows from Lemma~\ref{lem:main7} that 
$$q_i \in \delta_{\rho^{-k}} (\hat R(q,v))$$
for all $i=1,\dots,m$ and all $k \geq 1$ large enough. Hence $q_i \not \in B_d(\delta_{\rho^{-k}}(q),\rho^{-k})$ for all $i=1,\dots,m$ and all $k \geq 1$ large enough.

Next, by continuity of $d$ with respect to the manifold topology, we have
$$\lim_{k \rightarrow +\infty}  d(\delta_{\rho^{k}}(q_i), q) = d(0,q)$$
for all $i=1,\dots,m$. It follows that 
$$\lim_{k \rightarrow +\infty} d(q_i,\delta_{\rho^{-k}}(q)) = \lim_{k \rightarrow +\infty} \rho^{-k} d(\delta_{\rho^{k}}(q_i), q) = +\infty$$
for all $i=1,\dots,m$. Hence we can also choose $k\geq 1$ large enough so that
$$d(q_i,\delta_{\rho^{k}}(q)) > r_i$$
for all $i=1,\dots,m$. 

All together, we have proved that we can find $k\geq 1$ large enough so that the balls $B_d(q_1,r_1)$, $\dots$, $B_d(q_m,r_m)$, $B_d(\delta_{\rho^{-k}}(q),\rho^{-k})$ form a family of Besicovitch balls. We have $\delta_{\rho^{-k}}(q) = (0,\rho^{-k\alpha} y_q, \rho^{-k(\alpha+1)} z_q)$ with $\rho^{-k\alpha} y_q <0$, and $\rho^{-k} = d(0,\delta_{\rho^{-k}}(q))$. Then, using Remark~\ref{rmk:Besicovitch:open} below, we can choose $q_{m+1} = (x_{m+1},y_{m+1},z_{m+1})$ with $x_{m+1}, y_{m+1}<0$ close enough to $\delta_{\rho^{-k}}(q)$ so that, setting $r_{m+1} := d(0,q_{m+1})$, the family $\{B_d(q_i,r_i)\}_{i=1}^{m+1}$ is a family of Besicovitch balls. 
\end{proof}

\begin{remark} [Being a family of Besicovitch balls is an open condition] \label{rmk:Besicovitch:open}
Let $\{B_d(q_i,  r_i)\}_{i=1}^m$ be a family of Besicovitch balls where $r_i:=d(e,q_i)$ in a graded group $G$ with identity $e$ and equipped with a continuous self-similar quasi-distance $d$. One can find $U_1,\dots,U_m$ open neighborhoods of $q_1,\dots,q_m$ respectively, such that, for all $(q'_1,\dots,q'_m) \in U_1 \times \cdots \times U_m$, $\{B_d(q_i',  r_i')\}_{i=1}^m$ is a family of Besicovitch balls. Here we have set $r_i':=d(e,q_i')$. Indeed we have $d(q_i, q_j) -  d(e,q_i)>0$ for all $i\not = j$. By continuity of $d$ on $G\times G$ with respect to the manifold topology, one can find $U_1,\dots,U_m$ open neighborhoods of $q_1,\dots,q_m$ respectively such that, if $(q'_1,\dots,q'_m) \in U_1 \times \cdots \times U_m$, then $d(q'_i, q'_j) - d(e,q'_i) >0 $ for all $i\not = j$. Hence $\{B_d(q_i',  r_i')\}_{i=1}^m$ is a family of Besicovitch balls.
\end{remark}

\subsection{Topological properties of self-similar distances} \label{subsect:self-similar-topology}

As already mentioned, one of the main differences between self-similar and homogeneous quasi-distances are their topological properties. One cannot  extend Proposition~\ref{prop:homogeneous-quasi-distance-topology} to self-similar quasi-distances. There are indeed examples of self-similar distances on homogeneous groups such that the distance from the identity $e$ is not continuous at $e$ with respect to the manifold topology. Hence the topology induced by such self-similar distances does not coincide with the manifold topology.

One such example is the following. We consider $\R$ equipped with the usual addition as a group law and the dilations $\delta_\lambda(x) := \lambda x$. We take $(v_i)_{i \in I}$ a basis of $\R$ viewed as a vector space over $\Q$ and we choose it in such a way that some sequence $v_{i_j}$ converges to $0$ as $j$ goes to $+\infty$ for the usual topology of $\R$. For $x\in \R$, we write $x = \sum x_i v_i$ where $x_i \in \Q$ and all but finitely many of the $x_i$'s are $0$ and we consider the left-invariant distance $d$ such that $d(0, x) = \sum |x_i|$. This distance is $\Q$-homogeneous, i.e., $d(\delta_{q}(x),\delta_{q}(y)) = q d(x,y)$ for all $x$, $y\in\R$ and all $q\in \Q$, and hence self-similar. We have $d(0,v_i) = 1$ for all $i\in I$. In particular $d(0,v_{i_j}) =1$ for all $j$ whereas $v_{i_j}$ converges to $0$ as $j$ goes to $+\infty$ for the usual topology of $\R$. Hence $d(0,\cdot)$ is not continuous at $0$ with respect to the manifold topology on $\R$.

However, with the additional assumption of the continuity of $d(e,\cdot)$ at $e$ with respect to the manifold topology, one can extend  Proposition~\ref{prop:homogeneous-quasi-distance-topology} in the following way.

\begin{proposition}\label{prop:self-similar-quasi-distance-topology}
Let $G$ be a graded group with identity $e$. Let $d$ be a self-similar quasi-distance on $G$. Assume that $d(e,\cdot)$ is continuous at $e$ with respect to the manifold topology. Then the topology induced by $d$ coincides with the manifold topology. Moreover, a set is relatively compact 
if and only if it is bounded  with respect to $d$.
\end{proposition}

\begin{proof}
First, the fact that $d(e,\cdot)$ is assumed to be continuous at $e$ with respect to the manifold topology $\mathcal{T}_m$ implies that $\mathcal{T}_d \subset\mathcal{T}_m$ where $\mathcal{T}_d$ denotes the topology induced by $d$.

Then the proof can be completed with technical modifications of the proof of Proposition~\ref{prop:homogeneous-quasi-distance-topology} that we briefly sketch below. In the rest of this proof, the convergence of some sequence of points means convergence with respect to the manifold topology.

Let $0<\lambda<1$ be such that $d(\delta_\lambda(p),\delta_\lambda(q)) = \lambda d(p,q)$ for all $p$, $q\in G$. Let $(p_k)$ be a sequence such that $d(e,p_k)$ goes to 0 and let us prove that $p_k$ converges to $e$. Using the conventions and notations introduced in the proof of Proposition~\ref{prop:homogeneous-quasi-distance-topology}, we argue by contradiction and assume that, up to a subsequence, there exists $\varepsilon>0$ such that $\|p_k\| > \varepsilon$ for all $k$ (see~\eqref{e:norm} for the definition of $\|\cdot\|$).
For each fixed $k$, $\|\delta_{\lambda^l}(p_k)\|$ converges to $0$ when $l$ goes to $+\infty$. Hence one can find a sequence of integers $l_k\geq 1$ such that, for all $k$, $\|\delta_{\lambda^{l_k}}(p_k)\| \leq \varepsilon \leq \|\delta_{\lambda^{l_k-1}}(p_k)\|$. Then 
\begin{equation*}
\|\delta_{\lambda^{l_k}}(p_k)\| = \sum_{i=1}^n \lambda^{l_k d_i} |P_i(p)| \geq \lambda^{\overline\alpha}\sum_{i=1}^n \lambda^{(l_k-1) d_i} |P_i(p)| = \lambda^{\overline\alpha} \|\delta_{\lambda^{l_k-1}}(p_k)\| \geq \varepsilon \lambda^{\overline\alpha}
\end{equation*}
where $\overline\alpha := \max_{1\leq i\leq n} d_i$.  Hence $ \varepsilon \lambda^{\overline\alpha} \leq \|\delta_{\lambda^{l_k}}(p_k)\| \leq \varepsilon$ for all $k$. By compactness with respect to the manifold topology of $\{p\in G;\; \varepsilon \lambda^{\overline\alpha} \leq \|p\| \leq \varepsilon\}$, we get that, up to a subsequence, $\delta_{\lambda^{l_k}}(p_k)$ converges to some $p\in G$ such that $\varepsilon \lambda^{\overline\alpha} \leq \|p\| \leq \varepsilon$. In particular $p\not = e$ and $d(e,p) >0$. On the other hand, we have \begin{eqnarray*}
0<d(e, p)&\leq & C (d(e, \delta_{\lambda^{l_k}}(p_k))  +  d(\delta_{\lambda^{l_k}}(p_k), p)) \\
&= & C (\lambda^{l_k} d(e,  p_k)  +  d(e,p^{-1}\cdot \delta_{\lambda^{l_k}} (p_k)) )\\
&\leq & C (d(e,  p_k) + d(e,p^{-1}\cdot \delta_{\lambda^{l_k}} (p_k)) )~.
\end{eqnarray*}
Since $d(e,p_k)$ converges to $0$ and since $d(e,\cdot)$ is assumed to be continuous at $e$ with respect to the manifold topology and $p^{-1}\cdot \delta_{\lambda^{l_k}} (p_k)$ converges to $e$, we get that the last upper bound in the above inequalities goes to $0$, which gives a contradiction. It follows $\mathcal{T}_m \subset \mathcal{T}_d$ and, all together, we get that $\mathcal{T}_m = \mathcal{T}_d$.

The proof that a set is relatively compact if and only if it is bounded  with respect to $d$ can be achieved with similar technical modifications of the arguments of the proof of Proposition~\ref{prop:homogeneous-quasi-distance-topology}.
\end{proof}

Proposition~\ref{prop:self-similar-quasi-distance-topology} implies the following generalization of Proposition~\ref{prop:submetries-morphism-graded-algebra} to self-similar quasi-distances. In the statement and the proof of Proposition~\ref{prop:submetries-morphism-self-similar} continuity means continuity with respect to the manifold topology.

\begin{proposition} \label{prop:submetries-morphism-self-similar}
Let $\hat{G}$ and $G$ be graded groups with graded Lie algebra $\hat{\g}$ and $\g$  respectively. Assume that there exists a surjective morphism of graded Lie algebras $\phi:\hat{\g}\rightarrow\g$. Let $\varphi:\hat{G}\rightarrow G$ denote the unique Lie group homomorphism such that $\varphi_* = \phi$ and let $\hat{d}$ be a continuous self-similar quasi-distance on $\hat{G}$. Then
$$d(p,q):= \hat{d}(\varphi^{-1}(\{p\}),\varphi^{-1}(\{q\}))$$
defines a continuous self-similar quasi-distance on $G$ and $\varphi:(\hat{G},\hat{d})\rightarrow (G,d)$ is a submetry. 
\end{proposition}

\begin{proof}
One proves that $d$ is a self-similar quasi-distance on $G$ and that $\varphi:(\hat{G},\hat{d})\rightarrow (G,d)$ is a submetry with the same arguments as in the  proof of Proposition~\ref{prop:submetries-morphism-graded-algebra}. To prove that $d$ is continuous on $G \times G$, we first prove that $d(e, \cdot)$ is continuous at $e$. Here $e$ denotes the identity in $G$ and below $\hat{e}$ will denote the identity in $\hat{G}$. Let $(p_k)$ be a sequence converging to $e$ (with respect to the manifold topology). Since $\varphi$ is an open map and since $\hat{e}\in \varphi^{-1}(\{e\})$, one can find a sequence $\hat{p}_k \in \varphi^{-1}(\{p_k\})$ converging to $\hat{e}$. We have $d(e,p_k) \leq \hat{d}(\hat{e},\hat{p}_k)$. Since $\hat{d}$ is assumed to be continuous, $\hat{d}(\hat{e},\hat{p}_k)$ converges to $0$. It follows that  $d(e,p_k)$ converges to $0$ as well and hence $d(e,\cdot)$ is continuous at $e$. The proof can now be completed using Proposition~\ref{prop:self-similar-quasi-distance-topology} and following the arguments of the proof of  Proposition~\ref{prop:submetries-morphism-graded-algebra}.
\end{proof}

\subsection{Arbitrary graded groups with two different layers not commuting}
\label{subsect:nobcp-noncommuting-layers}

We conclude in this section the proof of Theorem~\ref{thm:sect-nobcp-self-similar-noncommuting}.

\begin{proof}[Proof of Theorem~\ref{thm:sect-nobcp-self-similar-noncommuting}]
Let $G$ be a graded group with associated positive grading $\g =\oplus_{t>0} V_t$ of its Lie algebra. Let $t<s$ be such that $[V_t,V_s] \not= \{0\}$. Let $d$ be a continuous self-similar quasi-distance on $G$. We argue by contradiction and assume that WBCP holds on $(G,d)$.

By Proposition~\ref{prop:structure-algebra-noncommuting-layers}, there exists a graded subalgebra $\hat{\g}$ of $\g$ and a surjective morphism of graded Lie algebras from $\hat{\g}$ onto $\h$, where $\h$ is the $t$-power of the non-standard Heisenberg Lie algebra of exponent $s/t$.

We denote by $\hat{G}:= \exp(\hat{\g})$ the graded group whose Lie algebra $\hat{\g}$ is endowed with the positive grading induced by the given positive grading of $\g$. The restriction of $d$ to $\hat{G}$, still denoted by $d$, is a continuous self-similar quasi-distance on $\hat{G}$. Since we assume that WBCP holds on $(G,d)$, WBCP also holds on $(\hat{G},d)$ by Proposition~\ref{prop:BCP-subset}. 

Next, it follows from Proposition~\ref{prop:submetries-morphism-self-similar} that there exists a continuous self-similar quasi-distance $d_H$ on the $t$-power, denoted by $H$, of the non-standard Heisenberg group of exponent $s/t$, and a submetry from $(\hat{G},d)$ onto $(H,d_H)$. Since WBCP  holds on $(\hat{G},d)$, WBCP holds on $(H,d_H)$ by Proposition~\ref{prop:bcp-submetry}. Finally, we get that WBCP would hold for the continuous self-similar quasi-distance $(d_H)^t$ on the non-standard Heisenberg group of exponent $s/t$. This contradicts Theorem~\ref{thm:heisenberg-case} and concludes the proof.
\end{proof}

\section{Differentiation of measures} \label{sect:differentiation-measures}

In this section, we give applications to differentiation of measures and we prove Theorem~\ref{thm:intro-preiss-improved}.

Following the terminology of~\cite{Preiss83}, if $(X,d)$ is a metric space, we say that $d$ is \textit{finite dimensional} on a subset $Y\subset X$ if there exist $C^*\in [1,+\infty)$ and $r^* \in (0,+\infty]$ such that $\card \mathcal{B} \leq C^*$ for every family $\mathcal{B} = \{B=B_d(x_B,r_B)\}$ of Besicovitch balls in $(X,d)$ such that $x_B \in Y$ and $r_B<r^*$ for all $B\in \mathcal{B}$ (see Definition~\ref{def:BesicovitchBalls} for the definition of families of Besicovitch balls). If we need to specify the constants $C^*$ and $r^*$, we say that $d$ is finite dimensional on $Y$ with constants $C^*$ and $r^*$. We say $d$ is \textit{$\sigma$-finite dimensional} if $X$ can be written as a countable union of subsets on which $d$ is finite dimensional. Note that WBCP holds on $(X,d)$ if and only if $d$ is finite dimensional on $X$ for some constant $C^*\in [1,+\infty)$ and with $r^* = +\infty$. 

For homogeneous distances on homogeneous groups, we first prove the following strongest equivalence between $\sigma$-finite dimensionality and validity of (W)BCP.

\begin{proposition} \label{prop:BCP-vs-finitedimensional}
Let $G$ be a homogeneous group and let $d$ be a homogeneous distance on $G$. Then BCP holds on $(G,d)$ if and only if $d$ is $\sigma$-finite dimensional. 
\end{proposition}

\begin{proof}
If BCP holds on $(G,d)$ then WBCP holds as well and it follows from the definitions that $d$ is finite dimensional on $G$. To prove the converse, let $m$ denote the Hausdorff dimension of $(G,d)$ and let $\mu$ denote the $m$-dimensional Hausdorff measure on $(G,d)$. It is well-known that $\mu$ is a Haar measure on $G$. Since the distance $d$ from which this $m$-dimensional Hausdorff measure is constructed is homogeneous, $\mu$ is moreover $m$-homogeneous with respect to the associated dilations $(\delta_\lambda)_{\lambda>0}$ on $G$, i.e., $\mu(\delta_\lambda(A)) = \lambda^m \mu(A)$ for all $A\subset G$ and all $\lambda >0$. In particular $\mu$ is a doubling outer measure on $(G,d)$. It follows that for every subset $A\subset G$, $\mu$-a.e. point $p\in A$ is a $\mu$-density point for $A$, i.e., 
\begin{equation*} \label{e:density1points}
\lim_{r\downarrow 0} \frac{\mu(A\cap B_d(p,r))}{\mu(B_d(p,r))} = 1~.
\end{equation*}
Assume that $d$ is $\sigma$-finite dimensional, i.e., $G= \cup_{n \in \N} G_n$ where $d$ is finite dimensional on each $G_n$. Then one can find $n\in\N$ such that $\mu(G_{n})>0$ and we set $A:=G_{n}$. Let $C^*\in [1,+\infty)$ and $r^* \in (0,+\infty]$ be such that $d$ is finite dimensional on $A$ with constant $C^*$ and $r^*$. Since $\mu(A)>0$, one can find a point $p\in A$ that is a $\mu$-density point for $A$. Up to a translation, one can assume that $p=e$ where $e$ denotes the identity in $G$. 
 Next, by homogeneity, for every $\lambda\geq 1$, $d$ is finite dimensional on $\delta_\lambda(A)$ with constants $C^*$ and $r^*$. 

Let us prove that WBCP, and hence BCP by Corollary~\ref{cor:bcp-wbcp-homogeneous-qdist}, holds in $(G,d)$. We consider $\{B_d(p_i,r_i)\}_{i=1}^k$ a family of Besicovitch balls in $(G,d)$. Up to a dilation, one can assume with no loss of generality that $r_i < r^*$ for all $i=1,\dots,k$. Up to a translation, one can further assume that $e\in \cap_{i=1}^k B_d(p_i,r_i)$. Shrinking balls if necessary, one can also assume that $r_i = d(e,p_i)$ for all $i=1,\dots,k$.

By Remark~\ref{rmk:Besicovitch:open} (recall that homogeneous distances on a homogeneous group are continuous, see Corollary~\ref{cor:continuity-homogeneous-distance}), one can find $\varepsilon^* >0$ such that $\{B_d(q_i,s_i)\}_{i=1}^k$ is a family of Besicovitch balls in $(G,d)$ as soon as $d(p_i,q_i) \leq \varepsilon^*$ for all $i=1,\dots,k$ and where $s_i:=d(e,q_i)$. Moreover one can choose $\varepsilon^*$ small enough, namely $\varepsilon^* < r^* - \max_{1\leq i \leq k} r_i$, so that $d(e,q) < r^*$ as soon as $d(q,p_i) \leq \varepsilon^*$ for some $i\in\{1,\dots,k\}$.

Fix $R \geq 1$ such that $R \geq 2 \max_{1\leq i \leq k} r_i$ and $\varepsilon>0$ such that $\varepsilon \leq \min (\varepsilon^*,\max_{1\leq i \leq k} r_i)$. Let $r<1$ be such that 
 \begin{equation*}
  \frac{\mu(A\cap B_d(e,r))}{\mu(B_d(e,r))} > 1 - \left(\frac{\varepsilon}{R}\right)^m~.
  \end{equation*}
Then we claim that $B_d(p_i,\varepsilon) \cap \delta_{R/r}(A) \not= \emptyset$ for all $i=1,\dots,k$. Indeed, arguing by contradiction, assume that $B_d(p_i,\varepsilon) \cap \delta_{R/r}(A) = \emptyset$. By choice of $R$ and $\varepsilon$, we have $B_d(p_i,\varepsilon) \subset B_d(e,R)$. Using the left-invariance and the homogeneity of $\mu$, it follows that
\begin{equation*}
\begin{split}
1 - \left(\frac{\varepsilon}{R}\right)^m &< \frac{\mu(A\cap B_d(e,r))}{\mu(B_d(e,r))} = \frac{\mu(\delta_{R/r}(A)\cap B_d(e,R))}{\mu(B_d(e,R))}  \\
&\leq \frac{\mu(B_d(e,R) \setminus B_d(p_i,\varepsilon))}{\mu(B_d(e,R))} = \frac{\mu(B_d(e,R)) - \mu(B_d(p_i,\varepsilon))}{\mu(B_d(e,R))}\\
 &= 1 - \left(\frac{\varepsilon}{R}\right)^m
\end{split}
\end{equation*}
which gives a contradiction.

Hence one can find $q_i \in B_d(p_i,\varepsilon) \cap \delta_{R/r}(A)$ for all $i=1,\dots,k$. Since $\varepsilon \leq \varepsilon^*$, we get that $\{B_d(q_i,s_i)\}_{i=1}^k$ is a family of Besicovitch balls in $(G,d)$ with $q_i \in \delta_{R/r}(A)$ and $s_i:=d(e,q_i) < r^*$ for all $i=1,\dots,k$. By choice of $R$ and $r$, we have $R/r \geq 1$ and it follows that $k \leq C^*$. Hence WBCP holds in $(G,d)$.
\end{proof}

Recall now that if $(X,d)$ is a metric space and $\mu$ is a locally finite Borel measure on $X$, we say that \textit{the differentiation theorem holds on $(X,d)$ for the measure $\mu$} if
\begin{equation*}
\lim_{r\downarrow 0^+} \frac{1}{\mu(B_d(p,r))} \int_{B_d(p,r)} f(q) \, d\mu(q) = f(p)
\end{equation*}
for $\mu$-almost every $p\in X$ and all $f\in L_{\rm loc}^1(\mu)$.

The connection between $\sigma$-finite dimensionality and measure differentiation in the general metric setting is given by the following result due to D. Preiss.

\begin{theorem} [{\cite{Preiss83}}]\label{thm:preiss-diff}
Let $(X,d)$ be a complete separable metric space. The  differentiation theorem holds on $(X,d)$ for all locally finite Borel measures if and only if $d$ is $\sigma$-finite dimensional.
\end{theorem}

The proof of Theorem~\ref{thm:intro-preiss-improved} can now be completed using  Theorem~\ref{thm:preiss-diff} and Proposition~\ref{prop:BCP-vs-finitedimensional}.

\section{Sub-Riemannian distances} \label{sec:nobcp-ccdist}

In this section we consider sub-Riemannian distances on stratified groups, and more generally on sub-Riemannian manifolds. We prove Theorem~\ref{thm:intro-no:BCP:Carnot} and Theorem~\ref{thm:intro-subriemannian-diff-measures}. The proofs of these results are independent of the main result in this paper, namely independent of Theorem~\ref{thm:intro-main-result}, but use some of the techniques developed here, in particular Proposition~\ref{prop:submetries-morphism-graded-algebra}.

\subsection{No BCP in sub-Riemannian Carnot groups}
\label{subsect:noBCP-sub-riemannian-groups}

Let $G$ be a stratified group with associated stratification of its Lie algebra given by $\g = V_1 \oplus \dots \oplus V_s$. We say that an absolutely continuous curve $\gamma:I \rightarrow G$ is \textit{horizontal} if, for a.e.~$s\in I$, one has $\dot{\gamma}(s) \in \Span\{X(\gamma(s));\, X\in V_1\}$. For a given scalar product inducing a norm $\|\cdot\|$ on the first layer $V_1$ of the stratification, we define the length, with respect to $\|\cdot\|$, of a horizontal curve $\gamma$ by
$$l_{\|\cdot\|}(\gamma) := \int_I \| \dot{\gamma}(s)\|\, ds~.$$
We say that $d$ is a \textit{sub-Riemannian distance} on $G$ if there exists a norm $\|\cdot\|$ induced by a scalar product on $V_1$ such that $$d(p,q) = \inf \left\{l_{\|\cdot\|}(\gamma); \; \gamma \text{ horizontal curve from } p \text{ to } q \right\}~.$$
It is well-known that such a $d$ defines a homogeneous distance on $G$.

\begin{definition}[Sub-Riemannian Carnot groups] \label{def:sub-riem-carnot-group} We say that $(G,d)$ is a sub-Riemannian Carnot group if $G$ is a stratified group equipped with a sub-Riemannian distance $d$.
\end{definition}

It is proved in~\cite{Rigot} that BCP does not hold on sub-Riemannian Carnot groups of step $\geq 2$ under some assumptions on the regularity of length-minimizing curves and of the sub-Riemannian distance (see~\cite[Theorem~1]{Rigot}). This result applies in particular to sub-Riemannian distances on the stratified first Heisenberg group as we recall now.

\begin{theorem} [\cite{Rigot}] \label{thm:nobcp-heisenberg}
Let $d$ be a sub-Riemannian distance on the stratified first Heisenberg group $\mathbb H^1$. Then BCP does not hold on $(\mathbb H^1,d)$.
\end{theorem}

We extend in Theorem~\ref{thm:intro-no:BCP:Carnot} the results of~\cite{Rigot} to any sub-Riemannian Carnot group of step $\geq 2$ without any further regularity assumptions.

Theorem~\ref{thm:intro-no:BCP:Carnot} will be proved by showing that from any such sub-Riemannian Carnot group, there exists a surjective morphism onto the stratified $n$-th Heisenberg group for some $n\in \N^*$. Moreover, the distance on the $n$-th Heisenberg group induced by this morphism is a sub-Riemannian distance (see Proposition~\ref{prop:submetries-morphism-graded-algebra} for the definition of this induced distance). Then the conclusion will follow from Theorem~\ref{thm:nobcp-heisenberg} and Lemma~\ref{lem:no:BCP:Carnot} below. Lemma~\ref{lem:no:BCP:Carnot} gives the existence of an isometric copy of a sub-Riemannian first Heisenberg group inside every sub-Riemannian $n$-th Heisenberg group. We refer to Example~\ref{ex:heisenberg} for the definition of the $n$-th Heisenberg Lie algebra $\h_n$ and the stratified $n$-th Heisenberg group $\mathbb{H}^n$. We first prove the following preliminary result.

\begin{proposition}\label{prop:no:BCP:Carnot}
Let $(G,d)$ be a step 2 sub-Riemannian Carnot group with associated stratification of its Lie algebra given by $\g = V_1 \oplus V_2$. Assume that $\dim V_2=1$ and that $V_2$ is the center of $\g$. Then $\g$ is the $n$-th Heisenberg Lie algebra $\h_n$. Moreover there exist positive real numbers $a_1,\dots, a_n$ and a standard basis $(X_1, \ldots, X_n, Y_1, \ldots, Y_n, Z)$ of $\h_n$ such that $d$ is the sub-Riemannian distance with respect to the norm induced by the scalar product on $V_1$ for which
$(a_i X_i, a_i Y_i)_{1\leq i \leq n}$ is orthonormal.
\end{proposition}

\begin{proof}
Let $Z\in V_2$ be fixed so that $V_2=\R Z$. Since $V_2$ is the center of $\g$, the Lie bracket restricted to $V_1 \times V_1$ can be identified with a non degenerate skew-symmetric bilinear form on $V_1$. Then the following facts follow from classical results of linear algebra. First, $\dim V_1$ is even, $\dim V_1 = 2n$ for some $n\in \N^*$. Second, considering $V_1$ equipped with the scalar product with respect to which the sub-Riemannian $d$ is defined, one can find an orthonormal basis $(\tilde X_1, \ldots, \tilde X_n, \tilde Y_1, \ldots, \tilde Y_n)$ of $V_1$ such that $[X_i,Y_j] = \delta_{ij}\alpha_i Z$ for some positive real numbers $\alpha_1 , \dots , \alpha_n$. We set $X_i :=  (\sqrt{\alpha_i})^{-1}\tilde X_i$. Then $(X_1,\dots,X_n,Y_1,\dots,Y_n,Z)$ is a basis of $\g$ with the only non-trivial bracket relations $[X_i,Y_i] = Z$ for $1\leq i \leq n$. It follows that $\g$ is the $n$-th Heisenberg Lie algebra $\h_n$ and $(X_1,\dots,X_n,Y_1,\dots,Y_n,Z)$ is a standard basis of $\h_n$. Moreover, setting $a_i:= \sqrt{\alpha_i}$, we get that the sub-Riemannian distance $d$ is the sub-Riemannian distance for which $(a_i X_i, a_i Y_i)_{1\leq i \leq n}$ is orthonormal.
\end{proof}

Lemma~\ref{lem:no:BCP:Carnot} below relies on classical results from sub-Riemannian geometry and more specifically on a classical study of length-minimizing curves in the sub-Riemannian Heisenberg groups. In particular length-minimizing curves can be explicitly computed using Pontryagin's maximum principle and the fact that there are no strictly abnormal curves. Also each horizontal curve can be recovered from its projection on the first layer of the stratification via the lift of the spanned area in each coordinate plane.

\begin{lemma}\label{lem:no:BCP:Carnot}
Let $(X_1, \ldots, X_n, Y_1, \ldots, Y_n, Z)$ be a standard basis of $\h_n$. Let 
$0<a_1\leq \ldots\leq a_n$ be positive real numbers and let $d$ denote the sub-Riemannian distance on $\mathbb{H}^n$ for which
$(a_i X_i, a_i Y_i)_{1\leq i\leq n}$ is orthonormal.
Then $H:=\exp({\Span}\{X_n, Y_n, Z\})$ is geodesically closed and the distance $d$ restricted to $H$ is the sub-Riemannian distance on $H$ for which $(a_nX_n,a_nY_n)$ is orthonormal.
\end{lemma}

\begin{proof}
In the statement $H$ is geodesically closed means that for any two points $p$, $q \in H$ there exists a length-minimizing curve, i.e., a horizontal curve $\gamma$ such that $d(p,q) = l_{\|\cdot\|}(\gamma)$, joining $p$ and $q$ and whose image is contained in $H$. Here $\|\cdot\|$ denotes the norm induced by the scalar product for which $(a_i X_i, a_i Y_i)_{1\leq i\leq n}$ is orthonormal. Then the fact that the distance $d$ restricted to $H$ is the sub-Riemannian distance on $H$ for which $(a_nX_n,a_nY_n)$ is orthonormal follows immediately.

We use in this proof exponential coordinates $(x_1,\dots,x_n,y_1,\dots,y_n,z)$ of the first kind  with respect to the basis $(a_1 X_1, \ldots, a_n X_n, a_1 Y_1, \ldots, a_n Y_n, Z)$ and $V_1$ denotes the first layer of the stratification, $V_1=\Span\{X_i,Y_i;\; 1\leq i \leq n\}$. By left-invariance, it is enough to show that given any point $p\in H$, there exists a  length-minimizing curve joining $0$ to $p$ whose image is contained in $H$. As already said, the proof below relies on classical results from sub-Riemannian geometry. We refer to e.g.~\cite{Monti-CCballsHeisenberg} or~\cite{Lerario-Rizzi} for more details.

Let us first consider $p=\exp(Z)=(0,\dots,0,1)$. Writing down the normal geodesic equation, one can see that if $\gamma:I\to \mathbb{H}^n$ is a length-minimizing curve from $0$ to $p$ with $\gamma(s) = (x_1(s),\dots,x_n(s),y_1(s),\dots,y_n(s),z(s))$ then its projections $s \mapsto (x_i(s),y_i(s))$ on each $(x_i,y_i)$-plane is a circle passing through the origin. Moreover, if we denote by $(v_i,v_{n+i})$ the center of each one of these circles, we have 
$$l_{\|\cdot\|}(\gamma) = 2\pi \left(\sum _{i=1}^{2n}  v_{ i}^2 \right)^{1/2}~.$$
On the other hand, since $\gamma$ is a horizontal curve, the last coordinate of its end point is equal to $\sum _{i=1}^n a_i^2\pi(v_{i}^2+v_{n+i}^2)$. Hence we must have $$\sum _{i=1}^n a_i^2\pi(v_{i}^2+v_{n+i}^2)=1~.$$
Notice that the minimum of  $\sqrt{\sum _{i=1}^{2n}  v_{ i}^2 }$ on the ellipsoid
$\{  \sum _{i=1}^n a_i^2\pi(v_{i}^2+v_{n+i}^2) = 1\}$ is attained on the smallest axis of the ellipsoid. It follows that the horizontal curve $\gamma$ contained in $H$ and whose projection on the $(x_n,y_n)$-plane is a circle passing through the origin and centered at $(1/a_n\sqrt{\pi},0)$ is a length-minimizing curve between $0$ and $p$.

Next, from the above length-minimizing curve between $0$ and $\exp(Z)$ and using dilations and rotations in $H$ around the $z$-axis, it can be seen that one can join the origin to any point in $H\setminus \exp(V_1)$ with a length-minimizing curve that stays in $H$. 

Finally, it can be checked that for $p\in H \cap \exp(V_1)$, the segment from $0$ to $p$ is the length minimizing curve from $0$ to $p$ and that this segment stays obviously inside $H$.
\end{proof}

We now conclude this section with the proof of the non-validity of BCP on sub-Riemannian Carnot groups of step $\geq 2$.

\begin{proof}[Proof of Theorem \ref{thm:intro-no:BCP:Carnot}]
  We argue by contradiction and we assume that $(G,d)$ is a sub-Riemannian Carnot group of step $\geq 2$ on which BCP hold. We prove that it is enough to consider the case of a sub-Riemannian distance on the stratified first Heisenberg group.
  
To prove this claim, we will perform a series of quotients. At the level of the Lie algebras, these quotients are surjective morphisms of stratified Lie algebras. Moreover, if we start from a sub-Riemannian distance, the sub-Riemannian distance  induced by each one of these morphisms (as in Proposition~\ref{prop:submetries-morphism-graded-algebra}), and for which BCP also holds by Propositions~\ref{prop:bcp-submetry} and~\ref{prop:submetries-morphism-graded-algebra}, will still be a sub-Riemannian distance. Indeed, if the stratification of the source Lie algebra is given by $V_1 \oplus \dots \oplus V_s$, we will only consider  morphisms of stratified Lie algebras whose kernel is contained either in $V_2\oplus \dots \oplus V_s$ or in the center of the source Lie algebra.
  
By a first quotient we can assume that the step of $G$ is exactly 2. Note that in view of Corollary~\ref{cor:intro-homogeneous-groups}, we could have skipped this first step. However this first quotient gives a proof of Theorem~\ref{thm:intro-no:BCP:Carnot} that is independent of our other results about more general graded groups. By a second quotient we 
can assume that the second layer of $G$ is one-dimensional. By a third quotient we can assume that the center of the Lie algebra of $G$ is the second layer of its stratification.

By Proposition~\ref{prop:no:BCP:Carnot} we know that the Lie algebra of $G$ is the $n$-th Heisenberg Lie algebra. We also know that there exist a standard basis $(X_1, \dots, X_n, Y_1, \dots, Y_n, Z)$ of $\h_n$ and positive real numbers $a_1,\dots,a_n$ such that $d$ is the sub-Riemannian distance for which $(a_i X_i, a_i Y_i)_{1\leq i \leq n}$ is orthonormal. Up to reordering the $a_j$'s, it follows from Lemma~\ref{lem:no:BCP:Carnot} that there is a stratified subgroup $H$ of $G$ isomorphic to the stratified first Heisenberg group and such that the sub-Riemannian distance $d$ restricted to $H$ is a sub-Riemannian distance on $H$. Since we started from a sub-Riemannian distance on $G$ for which BCP holds, it follows from Proposition~\ref{prop:BCP-subset} that BCP should hold for some sub-Riemannian distance on the stratified first Heisenberg group. This contradicts Theorem~\ref{thm:nobcp-heisenberg} and concludes the proof.
\end{proof}

\subsection{Differentiation of measures on sub-Riemannian manifolds}
\label{subsect:diff-mesaures-subriemannian-manifolds}

In this section we prove that sub-Riemannian distances on sub-Riemannian manifolds are not $\sigma$-finite dimensional, see Theorem~\ref{thm:subriemannian-not-finite-dimensional}. Then Theorem~\ref{thm:intro-subriemannian-diff-measures} follows from Theorem~\ref{thm:subriemannian-not-finite-dimensional} and Theorem~\ref{thm:preiss-diff}. We refer to Section~\ref{sect:differentiation-measures} for the definition of $\sigma$-finite dimensionality and its connection with measure differentiation. The proof of Theorem~\ref{thm:subriemannian-not-finite-dimensional} follows from the fact that at regular points, the metric tangent space to a sub-Riemannian manifold is isometric to a sub-Riemannian Carnot group of step $\geq 2$ together with Theorem~\ref{thm:intro-no:BCP:Carnot}.

To state our result, we first recall well-known facts about sub-Riemannian manifolds (see e.g.~\cite{Montgomery}). A sub-Riemannian manifold is a smooth Riemannian $n$-manifold $(M,g)$ equipped with a smooth distribution $\Delta$ of $k$-planes where $k<n$. We also assume that $M$ is connected and that the distribution satisfies H\"{o}rmander's condition. Namely, for all $p\in M$,  we assume that, if $X_1, \dots, X_k$ is a local basis of vector fields for the distribution near $p$, these vector fields, along with all their commutators, span $T_p M$. We denote by $V_i(p)$ the subspace of $T_p M$ spanned by all the commutators of the $X_j$'s of order $\leq i$ evaluated at $p$. Then the distribution satisfies H\"ormander's condition if, for all $p\in M$, there is some $i$ such that $V_i(p) = T_p M$. Next, an absolutely continuous curve in $M$ is said to be horizontal if it is a.e.~tangent to the distribution $\Delta$. Then the sub-Riemannian distance $d$ between two points $p$, $q\in M$ is defined by 
$$d(p,q) : = \inf \left\{\operatorname{length}_g (\gamma); \; \gamma \text{ horizontal curve from } p \text{ to } q \right\}~.$$
We say that $p\in M$ is a regular point if, for each $i$, $\operatorname{dim} V_i(p)$ remains constant near $p$. Regular points form an open dense set in $M$.

Our main result in this section reads as follows.

\begin{theorem} \label{thm:subriemannian-not-finite-dimensional}
Let $M$ be a sub-Riemannian manifold and let $d$ be its sub-Riemannian distance. Then $d$ is not $\sigma$-finite dimensional.
\end{theorem}

Notice that our definition of sub-Riemannian manifolds does not includes Riemannian ones. We recall that it is known that the Riemannian distance on a Riemannian manifold of class $C^2$ is $\sigma$-finite dimensional (see\cite[2.8]{Federer}).

Sub-Riemannian manifolds equipped with their sub-Riemannian distance include sub-Riemannian Carnot groups of step $\geq 2$. Besides, as said before, the following fact plays a crucial role for our purposes.

\begin{theorem} [{\cite[Theorem~7.36]{bellaiche}, \cite[Theorem~1]{Mitchell}}] \label{thm:tangent-subriemannian}
A sub-Riemannian manifold equipped with its sub-Riemannian distance admits a metric tangent space in Gromov's sense at every point. At regular points, this space is isometric to a sub-Riemannian Carnot group of step $\geq 2$.
\end{theorem}

A metric tangent space in Gromov's sense at a point $p$ in a metric space $(X,d)$ is defined as a Hausdorff limit of some sequence of pointed metric spaces $(X,d/\lambda_i,p)$ with $\lambda_i\rightarrow 0$ as $i\rightarrow +\infty$. We refer to~\cite{Gromov} for more details about metric tangent spaces.

\begin{proof}[Proof of Theorem~\ref{thm:subriemannian-not-finite-dimensional}]
We first prove that one only needs to consider the case where all points in $M$ are regular and moreover, for each $i$, $\operatorname{dim} V_i(p)$ remains constant in $M$. Indeed, since regular points form a nonempty open set, one can find an connected open set $O$ in $M$ such that, for each $i$, $\operatorname{dim} V_i(p)$ remains constant in $O$. We equip $O$ with the distribution of $k$-planes which is the restriction to $O$ of the given distribution on $M$, thus making $O$ a sub-Riemannian manifold. We denote by $d_O$ the associated sub-Riemannian distance on $O$. By definition we have for all $p$, $q \in O$,
$$d_O(p,q) = \inf \left\{\operatorname{length}_g (\gamma); \; \gamma \text{ horizontal curve from } p \text{ to } q \text{ in } O\right\}~.$$
In particular, $d(p,q)\leq d_O(p,q)$ for all $p$, $q \in O$, with possibly strict inequality. However it can easily be checked that, for all $p\in O$ and all $r<\operatorname{dist}(p, M\setminus O) /3$, one has $B_{d}(p,r) = B_{d_O}(p,r)$. It follows that if $d$ is $\sigma$-finite dimensional on $M$, then $d_O$ is $\sigma$-finite dimensional on $O$ as well. Indeed, assume that $M$ can be written as $M= \cup_{n\in \N} M_n$ where $d$ is finite dimensional on each $M_n$ with constants $C_n^*$ and $r_n^*$. Set $O_k := \{p\in O;\; \operatorname{dist}(p, M\setminus O) >1/k\}$. Then $B_{d}(p,r) = B_{d_O}(p,r)$ for all $p\in O_k$ and all $r<1/3k$. It follows that, for each $k$ and $n$, $d_O$ is finite dimensional on $O_k \cap M_n$ with constants $C_n^*$ and $\min(r_n^*, 1/3k)$. Since $O = \cup_{k,n\in \N} (O_k \cap M_n)$, this shows that $d_O$ is $\sigma$-finite dimensional on $O$.

Hence, from now on in this proof, let us assume that the distribution on $M$ is such that, for each $i$, $\operatorname{dim} V_i(p)$ remains constant in $M$ and let $\operatorname{dim} V_i$ denote its constant value. Such distributions are called generic in~\cite{Mitchell}. Arguing by contradiction we also assume that the sub-Riemannian distance $d$ on $M$ is $\sigma$-finite dimensional, $M= \cup_{n\in \N} M_n$ where $d$ is finite dimensional on each $M_n$.

We let $m:=\sum_i i(\operatorname{dim} V_i - \operatorname{dim} V_{i-1})$ denote the Hausdorff dimension of $(M,d)$ (see~\cite[Theorem 2]{Mitchell}) and let $\mu$ denote the $m$-dimensional Hausdorff measure in $(M,d)$. It is proved in~\cite{Mitchell} that $\mu$ is doubling on each compact set. This implies in particular that for any subset $A\subset M$, $\mu$-a.e.~point $p\in A$ is a $\mu$-density point for $A$. 
On the other hand since $\mu(M)>0$ and $M= \cup_{n\in \N} M_n$, one can find $n\in \N$ such that $\mu(M_n) >0$, and hence one can find some $\mu$-density point $p\in M_n$ for $M_n$. Then Proposition~3.1 in~\cite{LeDonne6} and Theorem~\ref{thm:tangent-subriemannian} imply that $(M_n,d)$ admits a metric tangent space in Gromov's sense at $p$ and this space is isometric to a sub-Riemannian Carnot group $(G,d_\infty)$ of step $\geq 2$. In particular, by definition of a metric tangent space, there exist a sequence $(\lambda_i)$ with $\lambda_i \rightarrow 0$ as $i\rightarrow +\infty$ and maps $\phi_i: G \rightarrow M_n$ such that, for all $p$, $q \in G$, 
$$ \frac{1}{\lambda_i} d(\phi_i(p),\phi_i(q)) \underset{i \rightarrow + \infty}{\longrightarrow} d_\infty(p,q)~.$$
Since $d$ is finite dimensional on $M_n$, Lemma~\ref{lem:tangent-finite-dimensionality} below implies that $d_\infty$ is a finite dimensional sub-Riemannian distance on $G$. This contradicts Proposition~\ref{prop:BCP-vs-finitedimensional} and Theorem~\ref{thm:intro-no:BCP:Carnot} and concludes the proof.
\end{proof}

\begin{lemma} \label{lem:tangent-finite-dimensionality}
Let $(X,d)$ and $(X_\infty,d_\infty)$ be metric spaces. Assume that there exist a sequence $(\lambda_i)$ with $\lambda_i \rightarrow 0$ as $i\rightarrow +\infty$ and maps $\phi_i: X_\infty \rightarrow X$ such that for all $x$, $y \in X_\infty$,
$$ \frac{1}{\lambda_i} d(\phi_i(x),\phi_i(y)) \underset{i \rightarrow + \infty}{\longrightarrow} d_\infty(x,y)~.$$
Assume that $d$ is finite dimensional on $X$. Then $d_\infty$ is finite dimensional on $X_\infty$.
\end{lemma}

\begin{proof} Assume that $d$ is finite dimensional on $X$ with constant $C^*\in [1,+\infty)$ and $r^*\in [0,+\infty]$. Let $\{B_{d_\infty} (x_1, r_1),\ldots,
B_{d_\infty} (x_N, r_N)\}$, $N\in \N$, 
be a family of Besicovitch balls in $(X_\infty,d_\infty)$ with $r_l<r^*$ for all $l=1,\dots,N$. Let $x_0 \in \cap_{1\leq l \leq N} B_{d_\infty} (x_l, r_l)$. Let $\varepsilon >0$ be small enough so that, for all $l,k = 1,\dots, N$ with $l\not= k$, $d_\infty(x_k,x_l) > \max (r_l,r_k)+2\varepsilon$ and so that $r_l+\varepsilon < r^*$ for all $l = 1,\dots, N$. Let $i$ be large enough so that $\lambda_i \leq 1$ and $$|\frac{1}{\lambda_i} d(\phi_i(x_l),\phi_i(x_k)) - d_\infty(x_l,x_k)| < \varepsilon$$ for all $l,k = 0,\dots,N$. Then for all $l,k = 1,\dots, N$ with $l\not= k$, we have
$$\frac{1}{\lambda_i} d(\phi_i(x_l),\phi_i(x_k)) > d_\infty(x_l,x_k) - \varepsilon > \max (r_l,r_k) + \varepsilon$$
and, for all $l=1,\dots,N$,
$$\frac{1}{\lambda_i} d(\phi_i(x_0),\phi_i(x_l)) < d_\infty(x_0,x_l) + \varepsilon \leq r_l + \varepsilon~.$$
Hence $\{B_{d} (\phi_i(x_1), \lambda_i (r_1+\varepsilon)),\ldots,B_{d} (\phi_i(x_N), \lambda_i (r_N+\varepsilon))\}$ is a family of Besicovitch balls in $(X,d)$ with radii $< r^*$. It follows that $N\leq C^*$ which proves that $d_\infty$ is finite dimensional on $X_\infty$.
\end{proof}

\begin{remark} One can relax the hypothesis of constant rank for the distribution $\Delta$ and the proof of Theorem~\ref{thm:subriemannian-not-finite-dimensional} can be generalized provided there exists a regular point $p \in M$ such that $\dim \Delta_p < \dim M$.
\end{remark}

\section{Final remarks} \label{sect:finalremarks}

\subsection{Assouad's embedding theorem and BCP}
Assouad's embedding Theorem for snowflakes of doubling metric spaces has the following consequence in connection with the Besicovitch Covering Property.

\begin{proposition}\label{prop:Assouad}
Let $(X,d)$ be a doubling metric space. Then there exists a continuous quasi-distance $\rho$ on $X$ that is biLipschitz equivalent to $d$ and such that $(X,\rho)$ satisfies BCP.
\end{proposition}

\begin{proof}
By Assouad's embedding theorem (see~\cite{Assouad83}), the snowflaked metric space $(X,d^{1/2})$ admits a biLipschitz embedding $F:(X,d^{1/2})\to \R^n$ into some Euclidean space. Since Euclidean spaces satisfy BCP, it follows that $F(X)$ equipped with the restriction of the Euclidean distance satisfies BCP (see Proposition~\ref{prop:BCP-subset}). Set $\rho(x,y):=\norm{F(x)-F(y)}^2$ for $x$, $y\in X$. This defines a continuous quasi-distance on $X$ that satisfies BCP. Finally, if $L$ is the biLipschitz constant of the embedding $F$, then $\rho$ is $L^2$-biLipschitz equivalent to the distance $d$.
\end{proof}

In particular, if $(G,d)$ is a homogeneous group equipped with a homogeneous distance, Proposition~\ref{prop:Assouad} gives the existence of a continuous quasi-distance biLipschitz equivalent to $d$ and for which BCP holds. In case $G$ is a homogenous group for which there are two different layers of the grading that do not commute, it follows from Theorem~\ref{thm:intro-main-result} that such a quasi-distance cannot be homogeneous, i.e., cannot be both left-invariant and one-homogeneous with respect to the associated dilations. What is not known is whether one can find one such $\rho$ with one of these properties. It is not known either whether one can find a distance, rather than a quasi-distance, biLipschitz equivalent to $d$ and for which BCP holds.

Notice that there are examples of non doubling metric spaces  satisfying BCP. Hence the doubling assumption in Proposition~\ref{prop:Assouad} is not a necessary condition to get the existence of biLipschitz equivalent (quasi-)distances satisfying BCP.

\subsection{Metric spaces of negative type}
For the stratified Heisenberg groups $\HH ^n$, J.~Lee and A.~Naor proved in~\cite{Lee_Naor} that the homogeneous distance $d$ such that 
$$d(0,p)= \left( \left(\sum_{j=1}^n x_j^2 + y_j^2\right)^2 + \sqrt{\left(\sum_{j=1}^n x_j^2 + y_j^2\right)^2 +16\, z^2}\right)^{1/4}$$
in exponential coordinates of the first kind relative to a standard basis of $\h_n$  is of negative type (see Example~\ref{ex:heisenberg} for the definition of standard basis of $\h_n$). Up to a multiplicative constant, this distance turns out to coincide with the Hebisch-Sikora's distance $d_R$ when $R=2$, namely, $d_2 = 8^{-1/4} d$ (see Example~\ref{ex:hebisch-sikora-distance} for the definition of $d_R$).

Recall that a metric space $(X,d)$ is said to be of negative type if $(X,\sqrt{d})$ is isometric to a subset of a Hilbert space. One could wonder whether the validity of BCP and the property of being of negative type may have some connections. One can for instance wonder whether Corollary~\ref{cor:intro-homogeneous-groups} could give some hints towards the existence of  homogeneous distances of negative type on homogeneous groups with commuting different layers, such as stratified groups of step 2. One can also wonder whether Corollary~\ref{cor:intro-homogeneous-groups} could give some hints towards the non existence of homogeneous distances of negative type on a homogeneous group for which there are two different layers of the grading that do not commute, such as stratified groups of step 3 or higher. Unfortunately, it turns out that one can always find subsets in a Hilbert space that, when equipped with the restriction of the Hilbert norm, are doubling metric spaces for which BCP does not hold. Hence it is not clear whether one can easily find connections between the validity of BCP and the property of being of negative type.

\subsection{Finite topological dimension and an open problem}
It is simple to show that if a  metric space $X$ satisfies BCP then $X$ has finite topological dimension. 
Here,  topological dimension means
Lebesgue covering dimension.
Indeed,
let $N$ be the constant in the definition of BCP, see~Definition~\ref{def:bcp}.
Without loss of generality we may assume that $X$ is bounded.
Then, given an open cover $\mathcal U$ of $X$, we shall prove that there is a refinement of $\mathcal U$ with multiplicity at most $N$, deducing that the  topological dimension is at most $N-1$.
 For every point $a$ in $X$ take a ball of small enough radius so that it is included in one element of $\mathcal U$. By the BCP this family of balls has a subfamily with multiplicity $N$ covering the space $X$.
 
If a metric space is assumed to be separable, then by the last argument, together with Remark \ref{rmk:WBCP:cover}, we have that 
WBCP  implies finite topological dimension. With a longer argument, in \cite{Nagata64}
J.~Nagata showed that also WBCP implies finite topological dimension, even if the space is not separable.

To our knowledge, there is no known example of a metric space $(X,d)$ of finite topological dimension for which there is no distances biLipschitz equivalent to $d$ satisfying BCP.

\bibliography{general_bibliography_TEMP} 
\bibliographystyle{amsalpha}

\end{document}